\documentclass[12pt]{article}
\usepackage[cp1251]{inputenc}
\usepackage[english]{babel}
\usepackage{amsfonts,amsmath,amsxtra,amsthm,amssymb,latexsym}
\pdfoutput=2
\usepackage{hyperref}
\usepackage[usenames,dvipsnames]{color}
\usepackage{verbatim}

\usepackage{dsfont}
\usepackage{kbordermatrix}
\usepackage{color}

\newtheorem{theorem}{Theorem}[section]
\newtheorem{corollary}[theorem]{Corollary}
\newtheorem{lemma}[theorem]{Lemma}
\newtheorem{proposition}[theorem]{Proposition}

\theoremstyle{definition}
\newtheorem{definition}[theorem]{Definition}
\newtheorem{remark}[theorem]{Remark}

\numberwithin{equation}{section}

\DeclareMathOperator{\supp}{supp}
\DeclareMathOperator{\im}{Im}

\DeclareMathOperator{\dom}{dom}
\DeclareMathOperator{\ran}{ran}
\DeclareMathOperator{\Ext}{Ext}

\DeclareMathOperator{\ac}{ac}

\DeclareMathOperator{\diag}{diag}
\DeclareMathOperator{\loc}{loc}

\DeclareMathOperator{\tr}{tr}

\def\Ext{{\rm Ext}}

\def\mul{{\rm mul\,}}
\def\wt#1{{{\widetilde #1} }}
\def\wt#1{{{\widetilde #1} }}

\newcommand\I{{\rm{i}}}
\newcommand\gr{{\rm{gr}}}

\newcommand\R{{\mathbb{R}}}
\newcommand\C{{\mathbb{C}}}

\newcommand\K{{\mathbb{K}}}

\newcommand{\bC}{\mathbb{C}}

\newcommand{\bO}{\mathbb{O}}

\newcommand{\bR}{\mathbb{R}}

\newcommand\cB{{\mathcal{B}}}
\newcommand\cC{{\mathcal{C}}}
\newcommand\cH{{\mathcal{H}}}

\newcommand\cN{{\mathcal{N}}}

\newcommand{\kA}{{\mathcal A}}
\newcommand{\kB}{{\mathcal B}}

\newcommand{\kF}{{\mathcal F}}
\newcommand{\kG}{{\mathcal G}}
\newcommand{\kH}{{\mathcal H}}

\newcommand{\kL}{{\mathcal L}}
\newcommand{\kM}{{\mathcal M}}

\newcommand{\kR}{{\mathcal R}}

\newcommand{\kT}{{\mathcal T}}

\newcommand{\kZ}{{\mathcal Z}}

\newcommand{\gG}{{\Gamma}}
\newcommand{\gT}{{\Theta}}
\newcommand{\gl}{{\lambda}}

\newcommand{\ga}{{\alpha}}

\newcommand{\gL}{{\Lambda}}

\newcommand{\gO}{{\Omega}}

\newcommand{\gs}{{\sigma}}

\newcommand{\f}{{\varphi}}

\newcommand\gH{{\mathfrak{H}}}
\newcommand{\gotH}{{\mathfrak{H}}}

\title{\bf Non-compact quantum graphs with summable matrix potentials}

\author{\bf Yaroslav Granovskyi, Mark Malamud and \fbox{Hagen Neidhardt}\footnote {Hagen Neidhardt deceased on 23 March 2019}}
\date{}

\begin{document}

\maketitle
%----------Author 1
%\author{Yaroslav Granovskyi}
%\address{Institute of Applied Mathematics and Mechanics\\
%R. Luxemburg Str. 74\\
%Donetsk\\
%Ukraine}
%\email{yarvodoley@mail.ru}
%
%\thanks{This work was completed with the support of the RUDN 5-100 program.}
%----------Author 2
%\author{Mark Malamud}
%\address{Peoples Friendship University of Russia (RUDN University)\\
%Miklukho-Maklaya Str. 6\\
%Moscow\\
%Russian Federation 117198}
%\email{malamud3m@gmail.com}
%
%\author{\fbox{Hagen Neidhardt}}
%\address{(Deceased)}
%
%\thanks{Research supported by
%the Czech Science Foundation ($GA\check{C}R$) under Grant No. 17-01706S and the European Union within the Project
%CZ.02.1.01/0.0/0.0/16_019/0000778 (P.E.), by the Austrian Science Fund (FWF) under Grant No. P28807 (A.K.),
%the ''RUDN University Program 5-100'' (M.M.) and by the European Research Council (ERC) under Grant No. AdG 267802 "AnaMultiScale" (H.N.).}
%\subjclass{Primary 34B45; Secondary 81Q10, 35Pxx, 58J50}

{\bf Abstract.}
Let $\mathcal{G}$ be a metric  noncompact  connected graph with finitely many edges.
%%Assume that the length of at least one of the edges is infinite.
The main object of the paper is the Hamiltonian ${\bf H}_{\alpha}$
associated in $L^2(\mathcal{G};\mathbb{C}^m)$ with a matrix Sturm-Liouville expression and boundary delta-type conditions at each vertex.
Assuming that the potential matrix is summable and applying the technique of boundary triplets and the corresponding Weyl functions,
we show that the singular continuous spectrum of the Hamiltonian ${\bf H}_{\alpha}$
as well as  any other self-adjoint realization of the Sturm-Liouville expression is empty.
We also indicate conditions on the graph ensuring  pure absolute continuity of the positive part of ${\bf H}_{\alpha}$.
Under an additional condition on the potential matrix, a Bargmann-type estimate
for the number of negative eigenvalues of  ${\bf H}_{\alpha}$ is obtained.
Additionally, for a star graph $\mathcal{G}$ a formula is found for the scattering matrix of the pair $\{{\bf H}_{\alpha}, {\bf H}_D\}$, where  ${\bf H}_D$ is the Dirichlet operator on $\mathcal{G}$.


%\begin{abstract}
%Let $\mathcal{G}$ be a metric  noncompact  connected graph with finitely many edges.
%%Assume that the length of at least one of the edges is infinite.
%The main object of the paper is the Hamiltonian ${\bf H}_{\alpha}$
%associated in $L^2(\mathcal{G};\mathbb{C}^m)$ with a matrix Sturm-Liouville expression and boundary delta-type conditions at each vertex.
%Assuming that the potential matrix is summable and applying the technique of boundary triplets and the corresponding Weyl functions,
%we show that the singular continuous spectrum of the Hamiltonian ${\bf H}_{\alpha}$
%as well as  any other self-adjoint realization of the Sturm-Liouville expression is empty.
%We also indicate conditions on the graph ensuring  pure absolute continuity of the positive part of ${\bf H}_{\alpha}$.
%Under an additional condition on the potential matrix, a Bargmann-type estimate
%for the number of negative eigenvalues of  ${\bf H}_{\alpha}$ is obtained.
%Additionally, for a star graph $\mathcal{G}$ a formula is found for the scattering matrix of the pair $\{{\bf H}_{\alpha}, {\bf H}_D\}$, where  ${\bf H}_D$ is the Dirichlet operator on $\mathcal{G}$.
%\end{abstract}

%\begin{document}

%\maketitle
%{\bf Abstract.}
%We describe the ...
%{\bf Keywords.}
%...
\tableofcontents

\section{Introduction}

The spectral theory of quantum graphs with a finite or infinite number of edges has been actively developed over the last
 three decades (see ~\cite{BCFK06,BerKuch13,DavPush11,EKMN18,Gerasim1988,GerPav88,Pan2012,Post2012}
 and the references therein). In particular,  Schr\"odinger and Laplace operators on the lattices, carbon nano-structures and periodic metric graphs have attracted a lot of attention (see e.g. \cite{KorLapt18,KorSab15,KuchPost07,KuchZeng03}).

In this paper we consider  a noncompact connected quantum graph
 $\mathcal{G}=(\mathcal{V},\mathcal{E})$ with finitely many edges $\mathcal{E}$ and vertices $\mathcal{V}$
 assuming  that  at least one of its  edges is  infinite.
We assume that $\mathcal{G}$ has no "loops" ("tadpoles") and multiple edges, i.e.
no edge starts and ends at the same vertex, and no edges connecting two same vertices;
this can always be achieved by introducing additional vertices, if necessary.
The main object of the paper is the Hamiltonian
${\bf H}_{\alpha}:={\bf H}_{\alpha,Q}$ associated in $L^2(\mathcal{G}; \mathbb{C}^m)$ with the matrix
Sturm-Liouville expression $\mathcal{A}=-\frac{d^2}{dx^2}+Q$  with summable and (pointwise) self-adjoint potential matrix
$Q=Q^*\in L^1(\mathcal{G};\mathbb{C}^{m\times m})$ and boundary delta-type conditions at each vertex $v\in \mathcal{V}$:
  \begin{equation}\label{delt-intro}
\begin{cases}
f \quad\text{is continuous at}\quad v,\\
\sum_{e\in E_v}f'_e(v)=\alpha(v)f(v),
\end{cases} \quad v\in \mathcal{V},
  \end{equation}
(see \cite{BerKuch13}, \cite{Post2012}, and formula \eqref{delt} below),
where $\alpha: \mathcal{V} \rightarrow \mathbb{C}^{m\times m}$, $\alpha(\cdot)=\alpha(\cdot)^*$ is a matrix function.
For $\alpha=0$, condition \eqref{delt-intro} turns into the well-known Kirchhoff condition,
and we denote by ${\bf H}_{\text{kir}}:={\bf H}_{0,Q}$  the corresponding Hamiltonian. To treat the Hamiltonian
${\bf H}_{\alpha,Q}$ in the framework of extension theory we introduce the minimal operator $A_{\min} = A=\bigoplus_{e\in \mathcal{E}}A_e$ associated
with the expression  $\mathcal{A}$ on $\mathcal G$ and being a direct sum of the minimal operators $A_e$ on each edge
$e\in \mathcal{E}$.

The paper consists of two  parts. In the first part we consider certain spectral problems for
arbitrary  noncompact connected quantum graphs  $\mathcal{G}$ with finitely many edges.
Namely, we show that each realization $\wt A$ of $\mathcal{A}$ (each extension of $A$),  including ${\bf H}_{\alpha}$,
has no singular  continuous spectrum  although it  may have a discrete set of  positive eigenvalues embedded in
the absolutely continuous  spectrum $\sigma_{ac}(\wt A) = [0,\infty)$.
We  also investigate  negative spectrum of ${\bf H}_{\alpha, Q}$  and  establish  Bargmann type estimates.

In the second part we consider only  \emph{quantum star graphs}
with finitely many edges. In particular,  we  indicate  explicit expressions for the numbers
of negative eigenvalues  $\kappa_-({\bf H}_{\alpha, Q})$ of   ${\bf H}_{\alpha, Q}$  in the case
of non-negative $A\ge 0$.
Moreover,  we compute the scattering matrix
 $\{S({\bf H}_{\alpha},{\bf H}_{D};\lambda)\}_{\lambda\in\mathbb{R}_+}$ for  the pair
$\{{\bf H}_{\alpha},{\bf H}_{D}\}$  where ${\bf H}_{D}:={\bf H}_{D,Q}$ is the Dirichlet operator
on the graph $\mathcal{G}$.

In Section \ref{sect.Prelim} we present necessary information on boundary triplets, the corresponding Weyl functions,
 and the scattering matrix.
We also present necessary facts on description of absolutely continuous and singular
continuous spectra of extensions in terms of  the limit behavior of the Weyl function on
the real axis. In Section \ref{sect.GMNP} following \cite{GMNP17}  we present the
necessary information on the spectral theory of Schr\"odinger operators on $L^2(\Bbb
{R}; \mathbb{C}^m)$ with summable potential matrix  $Q=Q^*$.

In Section \ref{sect.quagr} we prove our  first main result.
Its part reads as follows (see Theorem \ref{sc_spec} for more details).
   \begin{theorem}\label{sc_spec-intro}
Assume that graph $\mathcal{G}$ consists of $p_1>0$ leads and $p_2(\geq 0)$ edges,
let $Q\in L^1(\mathcal{G};\mathbb{C}^{m\times m})$.
Then for any  self-adjoint realization of $\mathcal{A}$  (extension of  the minimal operator $A$
generated by $\mathcal{A}$ on $\mathcal{G}$) the following holds:

\item[\;\;\rm (i)] The singular continuous spectrum $\sigma_{sc}(\widetilde{A})$ is empty, i.e.
$\sigma_{sc}(\widetilde{A}) = \emptyset$;

\item[\;\;\rm (ii)] The absolutely continuous spectrum $\sigma_{ac}(\widetilde{A})$ fills in the half-line $\mathbb{R}_+$, \\ $\sigma_{ac}(\widetilde{A})=[0,\infty)$
and it is of the constant multiplicity $mp_1$.

   \end{theorem}
In particular, the Hamiltonians  ${\bf H}_{\alpha,Q}$ have no sc-spectrum.

Theorem \ref{sc_spec-intro} generalizes the main  result of the paper \cite{GMNP17}
to the case of quantum graphs under discussion.  A special case of
Theorem \ref{sc_spec-intro} with $Q=0$  was  established  by B.-S. Ong \cite{Ong2006}  by using the limit
absorption principle.

In  Section \ref{sect.BargEst},  Theorem \ref{Bargkappa},  we establish  the following Bargmann type estimate for
the Hamiltonian ${\bf H}_{\alpha,Q}$  assuming that  $xQ\in L^1(\mathcal{G}; \mathbb{C}^{m\times m})$:
  \begin{equation}\label{Main_Barg_type_est_for_H_alpha,Q-intro}
\kappa_{-}({\bf H}_{\alpha,Q})\leq
\sum_{e\in \mathcal{E}} \left[\int_{e}x_e \cdot {\text{tr}}(Q_{e,-}(x))\,dx\right] + m|\mathcal{V}|.
  \end{equation}
Here $[a]$ is  the integer part of a number $a\in \Bbb R$, and $Q_{e,-}(\cdot)$  denotes the "negative"
part of the  potential matrix  $Q_{e}(\cdot)$ on the edge $e$.

We also  complete this formula by showing  that the numbers of negative eigenvalues
of the Hamiltonians  ${\bf H}_{\alpha, Q}$ and ${\bf H}_{\text{kir}}$ are related by the inequality
    \begin{equation}\label{eq:barg_estim_if_AN>0-intro}
\kappa_-({\bf H}_{\alpha, Q}) \le \kappa_-({\bf H}_{\text{kir}}) \ + \  \sum_{v\in
\mathcal{V}}\kappa_-(\alpha(v)).
    \end{equation}
In particular, if the  Kirchhoff  realization ${\bf H}_{\text{kir}}$ of $\mathcal A$ is
non-negative, ${\bf H}_{\text{kir}} \ge 0$,  then  $\kappa_-({\bf H}_{\alpha, Q}) \le
\sum_{v\in
\mathcal{V}}\kappa_-(\alpha(v)).$
Note that the proof of both estimates \eqref{Main_Barg_type_est_for_H_alpha,Q-intro}--\eqref{eq:barg_estim_if_AN>0-intro}
relies  on  formula \eqref{eq:main_quadr_form} for the quadratic
form associated with the Hamiltionian   ${\bf H}_{\alpha, Q}$ (see Proposition  \ref{lem_form_t_H_alpha,Q}).

In  Section \ref{sect.StarGr} we consider only  quantum \emph{star graphs}
with finitely many edges.  First, in Subsection \ref{sect.NegSpec}  assuming that \emph{the minimal
operator $A$ is non-negative}  and  employing the Weyl function technique we establish the equality
 $\kappa_{-}({\bf H}_{\alpha,Q})=\kappa_{-}(T)$  with a certain matrix  $T\in \mathbb{M}((p_2+1)m)$
(see Theorem \ref{th.kappaminus}), clarifying  and  completing  estimates  \eqref{Main_Barg_type_est_for_H_alpha,Q-intro}
and  \eqref{eq:barg_estim_if_AN>0-intro}.
It follows that  $\kappa_-({\bf H}_{\alpha,Q})\leq (p_2+1)m$.
Moreover,  assuming in addition that  the star graph $\mathcal G$ has no finite edges $(p_2=0)$ we show that
   \begin{equation}\label{ksglo-intro}
\kappa_-({\bf H}_{\alpha,Q})=\kappa_- \left(\alpha(0)-\sum_{e\in \mathcal{E}_{\infty}}M_e(0)\right)\leq m.
  \end{equation}
In particular,  ${\bf H}_{\alpha,Q}\geq 0$ if and only if $\alpha(0)\geq\sum_{e\in \mathcal{E}}M_e(0),$
where $M_e(0)$ is a limit value  at zero  of the Weyl function of the Dirichlet operator in $L^2(e;\mathbb{C}^m)$ for
each \emph{infinite edge (lead)} $e\in \mathcal{E}_{\infty}$.
For instance, we show that the Bargmann type estimates \eqref{Main_Barg_type_est_for_H_alpha,Q-intro}--\eqref{eq:barg_estim_if_AN>0-intro} are  not sharp.
Besides, assuming that  $Q=0$
we prove that $\kappa_-({\bf H}_{\alpha,0})=\kappa_-(T_1)\leq\sum_{k=0}^{p_2}\kappa_-(\alpha(v_k))$ (see \eqref{matrt-1}--\eqref{kapQz}).

Finally, in  Subsection \ref{sect.ScatMatr}  we compute the scattering matrix $\{S({\bf H}_{\alpha},{\bf H}_{D};\lambda)\}_{\lambda\in\mathbb{R}_+}$ for the pair
$\{{\bf H}_{\alpha},{\bf H}_{D}\}$, where ${\bf H}_{D}:={\bf H}_{D,Q}$ is the  Dirichlet  realization
on the star graph $\mathcal{G}$.
Namely, we show that with respect to  the spectral representation $L^2(\bR_+,d\gl;\mathcal{H}_{ac})$ the scattering matrix $\{S({\bf H}_\ga, {\bf H}_D;\gl)\}_{\gl\in\bR_+}$   admits the representation for a.e. $\lambda\in\mathbb{R}_+$:
\begin{equation}\label{scatmatr-intro}
\begin{split}
&S({\bf H}_\ga,{\bf H}_D;\gl)\\
&=I_{\mathcal H_{ac}}+\frac{i}{2\sqrt{\lambda}}
(N_1(\gl)^*)^{-1}\left((\ga(0)-K(\gl))^{-1}\otimes E_{p_1}\right)\cdot N_1(\gl)^{-1}.
\end{split}
\end{equation}
Here $K(\lambda)=\sum_{j=1}^{p_1}M_j(\lambda+i0)$ and  $M_j(\lambda+i0)$ is the limit value of the Weyl
function $M_j(\cdot)$,  corresponding to  the Dirichlet operator in $L^2(e_j; \mathbb{C}^{m})$ for
\emph{the  lead} $e_j\in \mathcal{E}_{\infty}$,
   $E_{p_1}\in \mathbb{C}^{p_1\times p_1}$ is the matrix consisting of $p_1^2$ units  (see \eqref{EpEp-1}),
and $N_1(\lambda)$ is a certain matrix function  expressed by means of the  potential matrix $Q$.
The proof  substantially relies on the results of \cite{BerMalNei08, BerMalNei17}
(see  Theorem \ref{th.ScS} for details).

The main results of the paper have been announced in \cite{GraMalNei19}.

It happens that we have completed the paper without our dear friend,
colleague/coauthor, and deep excellent  mathematician   Hagen Neidhardt with whom one of us
collaborated a lot in spectral and scattering theory.

{\bf Notation.}\
Through the paper $\mathfrak{H}$ and $\mathcal{H}$ denote separable Hilbert spaces;  $\mathcal{B}(\mathfrak{H})$ and $\mathcal{C}(\mathfrak{H})$ denote the spaces of
bounded and closed operators in $\mathfrak{H}$, respectively.
As usual  $\sigma_{ac}(T)$, $\sigma_{sc}(T)$, $\sigma_{p}(T)$, and $\sigma_{d}(T)$
denote the absolutely continuous, singular continuous, point and discrete spectra of an operator $T=T^*\in\mathcal{C}(\mathfrak{H})$;
 $\sigma_{p}(T)$ is the closure of $\sigma_{pp}(T)$, the set of eigenvalues of $T$.
$E_T(\cdot)$ is the spectral measure of $T$, and $\kappa_-(T):=\dim E_T(-\infty,0)$;
$T^{ac}$ is the absolutely continuous part of $T$;   $N_T(\cdot)$ is the  multiplicity function of $T$;

$\mathbb{M}(n):=\mathbb{C}^{n\times n}$ is the set of $n\times n$-matrices  with complex entries;
$\mathbb{N}_0:= \mathbb{N}\cup \{0\}$;
$\bR_+ := (0,\infty)$.

\section{Preliminaries}\label{sect.Prelim}

\subsection{Boundary triplets and Weyl functions}

Let us recall some basic facts  of the theory of abstract boundary
triplets and the corresponding  Weyl functions, cf. \cite{DerMal91, DerMal95, GorGor91}.

The set $\widetilde\cC(\cH)$ of closed linear relations in $\cH$ is the
set of closed linear subspaces of $\cH\oplus\cH$.
Recall that $\dom(\Theta) =\bigl\{ f:\{f,f'\}\in\Theta\bigr\} $, $\ran(\Theta) =\bigl\{
f^\prime:\{f,f'\}\in\Theta\bigr\} $, and $\mul(\Theta) =\bigl\{
f^\prime:\{0,f'\}\in\Theta\bigr\} $ are the domain, the range, and the multivalued part
of $\Theta$. A closed linear operator $A$ in $\cH$ is identified with its graph
$\gr(A)$, so that the set  $\cC(\cH)$  of closed linear operators in $\cH$ is viewed as
a subset of $\widetilde\cC(\cH)$. In particular, a linear relation $\Theta$ is an
operator if and
only if
$\mul(\Theta)$ is trivial.
We recall that the adjoint relation $\Theta^*\in\widetilde\cC(\cH)$ of $\Theta\in
\widetilde\cC(\cH)$ is defined by
\begin{equation*}
\Theta^*= \left\{
\begin{pmatrix} h\\h^\prime
\end{pmatrix}: (f^\prime,h)_{\cH}=(f,h^\prime)_{\cH}\,\,\text{for all}\,
\begin{pmatrix} f\\f^\prime\end{pmatrix}
\in\Theta\right\}.
\end{equation*}
A linear relation $\Theta$ is said to be {\it symmetric} if $\Theta\subset\Theta^*$ and
self-adjoint if $\Theta=\Theta^*$.

For a symmetric linear relation $\Theta\subseteq\Theta^*$ in $\cH$ the multivalued part
$\mul(\Theta)$ is the orthogonal complement of $\dom(\Theta)$ in $\cH$. Therefore
setting $\cH_{\rm op}:=\overline{\dom(\Theta)}$ and $\cH_\infty=\mul(\Theta)$, one
arrives at the orthogonal decomposition  $\Theta= \Theta_{\rm op}\oplus \Theta_\infty$
where  $\Theta_{\rm op}$ is a  symmetric operator in $\cH_{\rm op}$, the operator part
of $\Theta,$ and
$\Theta_\infty=\bigl\{\bigl(\begin{smallmatrix} 0 \\ f'
\end{smallmatrix}\bigr):f'\in\mul(\Theta)\bigr\}$,  a ''pure'' linear relation
 in $\cH_\infty$. % (see \cite{RB_85}).\\

Let $A$ be a densely defined closed symmetric operator in a separable Hilbert space
$\gH$, $\cN_z:=\ker(A^*-z)$ its  defect subspaces, and let
$\mathrm{n}_\pm(A)=\dim(\cN_{\pm \I}) \leq \infty$ be its deficiency indices.
%%with equal deficiency indices $\mathrm{n}_\pm(A)=\dim(\cN_{\pm \I}) \leq \infty,$ where
%
   \begin{definition}[\cite{GorGor91}]\label{def_ordinary_bt}
A triplet $\Pi=\{\cH,\gG_0,\gG_1\}$ is called a {\rm boundary triplet} for
the adjoint operator $A^*$ if $\cH$ is an auxiliary Hilbert space and
$\Gamma_0,\Gamma_1:\  \dom(A^*)\rightarrow \cH$ are linear mappings such that the
abstract Green identity
\begin{equation}\label{II.1.2_green_f}
(A^*f,g)_\gH - (f,A^*g)_\gH = (\gG_1f,\gG_0g)_\cH - (\gG_0f,\gG_1g)_\cH, \quad
f,g\in\dom(A^*),
\end{equation}
holds and the mapping $\gG:=\begin{pmatrix}\Gamma_0\\\Gamma_1\end{pmatrix}:  \dom(A^*)
\rightarrow \cH \oplus \cH$ is surjective.
\end{definition}
Note that a boundary triplet for $A^*$ exists whenever  $\mathrm{n}_+(A)=
\mathrm{n}_-(A)$. Moreover, $\mathrm{n}_\pm(A) = \dim(\cH)$ and $\ker(\Gamma) =
\ker(\Gamma_0) \cap \ker(\Gamma_1)= \dom(A)$. Note also that $\Gamma$ is a bounded
mapping from $\gotH_+ = \dom(A^*)$ equipped with the graph norm of $A^*$ to
$\cH\oplus\cH.$  A boundary triplet for $A^*$ is not unique.

  \begin{definition}
\item[\;\;\rm (i)]  A closed extension $A'$ of $A$ is called a \emph{proper
extension}, if $A\subset A' \subset A^*$.  The set of all proper extensions of  $A$
completed by the (non-proper) extensions $A$ and $A^*$ is  denoted  by $\Ext_A$.

\item[\;\;\rm(ii)]  Two  extensions $A', A''\in \Ext_A$  are called disjoint
if $\dom( A')\cap \dom( A'') = \dom( A)$ and  transversal if in addition $\dom( A') +
\dom( A'') = \dom( A^*).$
   \end{definition}

Any self-adjoint extension $\wt A$ of $A$ is proper, i.e. $\wt A\in \Ext_A$.
Fixing a boundary triplet $\Pi$ one can parameterize the set $\Ext_A$ in the following
way.
\begin{proposition}[\cite{DerMal95}]\label{prop_II.1.2_01}
Let $A$ be as above and let $\Pi=\{\cH,\gG_0,\gG_1\}$ be a boundary triplet for $A^*$.
Then the mapping
     \begin{equation}\label{II.1.2_01A}
\Ext_A\ni \widetilde A \to  \Gamma \dom(\widetilde A) =\{\{\Gamma_0 f,\Gamma_1f \} : \
f\in \dom(\widetilde A) \} =: \Theta \in \widetilde\cC(\cH)
     \end{equation}
establishes  a bijective correspondence between the sets $\Ext_A$ and
$\widetilde\cC(\cH)$. We put $A_\Theta :=\widetilde A$ where $\Theta$ is defined by
\eqref{II.1.2_01A}, i.e.  $A_\Theta:= A^*\upharpoonright
\Gamma^{-1}\Theta=A^*\upharpoonright \bigl\{f\in\dom(A^*): \ \{\Gamma_0f,\Gamma_1f\}
\in\Theta\bigr\}.$   Moreover:

\item[\;\;\rm(i)] $A_\Theta$ is symmetric (self-adjoint) if and only if $\Theta$ is symmetric (self-adjoint), i.e. $\Theta \subseteq \Theta^*$\ $(\Theta = \Theta^*)$.  Besides,  $\mathrm{n}_\pm(A_\Theta)=\mathrm{n}_\pm(\Theta)$;

\item[\;\;\rm(ii)]  The extensions $A_\Theta$ and $A_0$ are disjoint
(transversal) if and only if $\Theta$  is the graph of an operator $B\in \mathcal C(\cH)$.
In this case  $A_B := A_\Theta$ is given by
  \begin{equation}\label{II.1.2_01AB}
\widetilde A =  A_B = A^*\!\upharpoonright\ker(\gG_1 - B\gG_0).
  \end{equation}
Moreover, the extensions $A_B$ and $A_0$ are transversal if and only if $B \in
\mathcal{B}(\cH)$.
      \end{proposition}

The linear relation $\gT$ (the operator $B$) in the correspondence \eqref{II.1.2_01A}
(resp.  \eqref{II.1.2_01AB}) is  called \emph{the boundary relation (the boundary
operator)}.
We emphasize that for differential operators
parametrization \eqref{II.1.2_01A}--\eqref{II.1.2_01AB}  leads to a description of
the set of proper extensions directly in terms of boundary conditions.

It follows immediately from Proposition \ref{prop_II.1.2_01}(i) that the extensions
   \begin{equation}\label{II.Ext-s_A_0_and_A_1}
A_0:=A^*\!\upharpoonright\ker(\gG_0)\quad \text{and}\quad
A_1:=A^*\!\upharpoonright\ker(\gG_1)
   \end{equation}
are self-adjoint. Indeed,  $A_j=A_{\Theta_j}, \ j\in \{0,1\},$ where the subspaces
$\gT_0:= \{0\} \times \cH$ and $\gT_1 := \cH \times \{0\}$ are self-adjoint relations in
$\cH$.

Moreover, for any self-adjoint extension
$\wt A := \wt A^* \in \Ext_A$  there exists a boundary triplet $\Pi=\{\cH,\gG_0,\gG_1\}$
for $A^*$ such that $\ker(\Gamma_0) = \dom(\wt A)$.

Next following  \cite{DerMal91, DerMal95}  we introduce the concept of abstract Weyl
function. Emphasize that its role  in the extension theory is similar to that of the
classical Weyl--Titchmarsh $m$-function in the spectral theory of singular
Sturm-Liouville operators.

\begin{definition}[{\cite{DerMal91}}]\label{def_Weylfunc}

Let $A$ be a densely defined closed symmetric operator in $\gH$ with
equal deficiency  indices  and let $\Pi=\{\cH,\gG_0,\gG_1\}$ be a boundary triplet for $A^*$. The operator
valued functions $\gamma(\cdot) :\rho(A_0)\rightarrow  \mathcal{B}(\cH,\gH)$ and
$M(\cdot):\rho(A_0)\rightarrow  \mathcal{B}(\cH)$ defined by
  \begin{equation}\label{II.1.3_01}
\gamma(z):=\bigl(\Gamma_0\!\upharpoonright\cN_z\bigr)^{-1} \qquad\text{and}\qquad
M(z):=\Gamma_1\gamma(z), \qquad z\in\rho(A_0),
  \end{equation}
are called the {\em $\gamma$-field} and the {\em Weyl function}, respectively,
corresponding to the boundary triplet $\Pi.$
\end{definition}
In the case of Schr{\"o}dinger  operator in domains (inner or outer) with compact
boundary the mappings $\Gamma_0$ and $\Gamma_1$ can be selected as regularized Dirichlet
and Neumann traces (see \cite{Grubb68} and \cite{Malamud2010}). In this case  the
corresponding Weyl function $M(\cdot)$ coincides with is a version  of the
Dirichlet-to-Neumann map depending on $z$ see \cite{Malamud2010}. For Sturm-Liouville
operator this fact is much simpler and is immediate from Lemma \ref{lemma 3.15}. For
quantum graph it is more cumbersome although it is  valid (see the proof of Theorem
\ref{sc_spec}).

The $\gamma$-field $\gamma(\cdot)$ and the Weyl function $M(\cdot)$ in \eqref{II.1.3_01}
are well defined and holomorphic on $\rho(A_0)$.
The Weyl function $M(\cdot)$ is $\mathcal{B}(\cH)$-valued {\it Nevanlinna function},
($M(\cdot) \in  R[\cH]$), i.e. it  is holomorphic function on $\C\setminus \R$
satisfying
\begin{equation}\label{II.1.3_03}
 \im z\cdot\im M(z)\geq 0,\qquad  M(z)^*=M(\overline
z),\qquad  z\in \C\setminus \R.
\end{equation}
Moreover,    $M(\cdot)\in R^u[\cH],$ i.e.
it satisfies $0\in \rho(\im M(i)).$  \\

It is well known  that each $R[\cH]$-function, in particular, the Weyl function $M(\cdot)$,
admits an integral representation (see, for instance,  \cite{AkhGlz81}, \cite{Bro06})
   \begin{equation}\label{WF_intrepr}
M(z)=C_0+\int_{\R}\left(\frac{1}{t-z}-\frac{t}{1+t^2}\right)d\Sigma_M(t),\quad  \int_\R
{\frac{d\Sigma_M(t)}{1+t^2}} \in \mathcal{B}(\cH)
%%z\in\rho(A_0),
  \end{equation}
for $z\in\rho(A_0)$, where $\Sigma_M(\cdot)$ is an operator-valued Borel measure on $\R$
and $C_0 = C_0^*\in \mathcal{B}(\cH)$. The integral in
(\ref{WF_intrepr}) is understood in the strong sense.
A linear term $C_1z$ is missing in \eqref{WF_intrepr} because   $A$  is densely
defined (see \cite{DerMal91}).

Recall that a symmetric operator $A$ in $\gotH$ is said to be {\it
simple} if it does not admit a non-trivial orthogonal decomposition
$A=A'\oplus S$ where $A'$ is a symmetric operator and $S=S^*$.
It is  well-known  that $A$ is simple if and only if  the closed linear span of $\{\mathfrak{N}_z(A): z\in\C\setminus\R \}$
coincides with $\gotH$.

If $A$ is  simple, then  the Weyl function $M(\cdot)$ determines the boundary triplet
$\Pi$ uniquely up to the unitary equivalence  (see \cite{DerMal91}). In particular,
$M(\cdot)$ contains the full information about the spectral properties of $A_0$.
Moreover, the spectrum  of any  proper  (not necessarily self-adjoint) extension
$A_\Theta\in \Ext_A$ can be described  by means of $M(\cdot)$ and  the   boundary
relation $\Theta$.
   \begin{proposition}[{\cite[Theorem 2.2]{DerMal91}}]\label{prop_II.1.4_spectrum}
Let $\Pi=\{\cH,\gG_0,\gG_1\}$ be a boundary triplet for $A^*$ and let  $M(\cdot)$  and
$\gamma(\cdot)$ be the corresponding Weyl function and the  $\gamma$-field. Let also $\widetilde A = A_\Theta \in \Ext_A$
and $\rho(A_\Theta)\not = \emptyset$. Then:

\item[\;\;\rm (i)]  The  following Krein type formula holds
   \begin{equation}\label{II.1.4_01}
(A_\Theta - z)^{-1} - (A_0 - z)^{-1} = \gamma(z) (\Theta - M(z))^{-1}\gamma^*({\overline z}), \quad z\in \rho(A_0)\cap \rho(A_\Theta).
  \end{equation}
\item[\;\;\rm (ii)]   If $A$  is simple, then for any  $z\in \rho(A_0)$  the following
equivalence holds
\begin{displaymath}
z\in\sigma_j(A_\Theta) \quad \Longleftrightarrow\quad 0\in
\sigma_j(\Theta-M(z)),\qquad j\in\{\rm pp, c, r \}.
\end{displaymath}
\end{proposition}
We complete formula \eqref{II.1.4_01}  for $\Theta = \ker \left(C\,\,\, D\right)$ with $C,D\in\mathcal{B}(\mathcal{H})$ (see \cite{MalMog02}, \cite{DerMal17}). Then
  \begin{equation}\label{themM-Prelim}
(\Theta - M(z))^{-1}=\left(C - DM(z)\right)^{-1}D.
  \end{equation}

Formula \eqref{II.1.4_01} generalizes the classical  Krein formula for canonical
resolvents (cf. \cite{Kre47}, \cite{DerMal17}). %%{KreLang71}
It establishes  a one-to-one
correspondence between the set of proper extensions $\wt A = A_\gT$ with
$\rho(A_\Theta)\not = \emptyset$, and the set of linear relations $\Theta$ in $\cH$.
Note also that all parameters  entered  in \eqref{II.1.4_01} are expressed in terms of
the boundary triplet $\Pi$ (see formulas \eqref{II.1.2_01AB} and \eqref{II.1.3_01}) (cf.
\cite{DerMal91,DerMal95}).

\begin{remark}[{\cite[Ch. VIII]{AkhGlz81}}, \cite{Rofe-Bek1969}]\label{A_CD}
 In the case of $\mathrm{n}_{\pm}(A)=m<\infty$, the set of all self-adjoint extensions of the operator $A$ is parameterized as follows:
\begin{equation}\label{C-D}
\begin{gathered}
\Ext_A\ni\widetilde{A}=\widetilde{A}^*=A_{C,D}=A^*\upharpoonright\ker(D\Gamma_1-C\Gamma_0), \\
\text{where} \quad CD^*=DC^*,\quad \det(CC^*+DD^*)\neq 0, \quad C, D\in\mathbb{C}^{m\times m}.
\end{gathered}
\end{equation}
\end{remark}

In the following proposition describing self-adjoint extensions
with  finite negative spectrum we restrict ourselves to the case of finite deficiency indices.
   \begin{proposition}[{\cite{DerMal91, DerMal95}}]\label{A_0}
Let $A$ be  a densely defined non-negative symmetric operator  in $\gotH$, $\mathrm{n}_{\pm}(A)=m<\infty,$
let $\Pi=\{\cH,\gG_0,\gG_1\}$ be a boundary triplet for  $A^*$ such that
$A_0=\widehat{A}_F$   is the Friedrichs extension of $A$. Further, let
$M(\cdot)$ be the corresponding Weyl function and  let
$A_{C,D}$  be a selfadjoint extension  of $A$ given by~\eqref{C-D}.  Then
   \begin{equation}\label{kappa}
 \kappa_-(A_{C,D})=\kappa_-(CD^*-DM(0)D^*).
\end{equation}
In particular, $A_{C,D}\geq 0$ if and only if $CD^*-DM(0)D^*\geq 0.$
\end{proposition}

\subsection{Weyl function and spectrum}
In the following we are going to characterize the spectrum of the extension $A_0$ in terms of the Weyl function.
To this end let $\Phi(\cdot)$ be a scalar Nevanlinna function.
In what follows by $\lim_{z \to\!\succ x}\Phi(z)$ we mean
 that the limit $\lim_{r\downarrow 0}\Phi(x + r e^{i\theta})$, $x \in \R$, exist uniformly in $\theta \in [\varepsilon,\pi-\varepsilon]$
for each $\varepsilon \in (0,\pi/2)$. Let us introduce the following sets:
\begin{align}
\gO_s(\Phi)    &:= \{x \in {\bR}: |\Phi(z)| \to +\infty \; \mbox{as} \; z \to\!\succ x\},\label{3.26-9}\\
\gO_{pp}(\Phi) &:= \{x \in {\bR}: \lim_{z \to\!\succ x}(z-x)\Phi(z) \not= 0\},\label{3.26-10}\\
\gO_{sc}(\Phi) &:= \{x \in {\bR}: |\Phi(z)| \to +\infty \; \mbox{and} \;(z-x)\Phi(z) \to 0 \; \mbox{as} \; z \to\!\succ x\},\label{3.26-11}\\
\gO_{ac}(\Phi) &:= \{x \in {\bR}: 0 < \im\Phi(x+i0) <+\infty\}, \; \Phi(x+i0) = \lim_{y\downarrow 0}\Phi(x+iy). \nonumber
%%\label{3.26-12}
\end{align}
Any scalar R-function $\Phi(\cdot)$ admits the
representation
\begin{equation}\label{3.26-13}
\Phi(z) = C_0 + C_1z + \int_\R\left(\frac{1}{t-z} - \frac{t}{1 +
t^2}\right)d\mu(t), \quad z \in {\bC}_+,
\end{equation}
with constants  $C_0 \in {\bR}$, $C_1 \ge 0$, and the  scalar Borel measure $\mu(\cdot)$
satisfying  $\int_\R (1+ t^2)^{-1}{d\mu(t)} < \infty.$
   \begin{theorem}[{\cite[Theorem 4.3]{BraMalNei02}}]\label{IV.3}
Let $A$ be a simple  closed symmetric operator in $\gotH$ with $n_{\pm}(A) = n<\infty$.
Let $\Pi = \{\kH,\gG_0,\gG_1\}$ be a boundary triple of $A^*$, let  $M(\cdot)$  be  the
corresponding  Weyl function, $M_h(z) := (M(z)h,h)$,  $h \in \mathcal H$,  and let $\kT
= \{h_k\}^n_{k=1}$ be  a basis in $\kH$. Then:

\item[\;\;\rm(i)] the  extension $A_0 (=A^*\!\upharpoonright\ker(\gG_0))$ (see
\eqref{II.Ext-s_A_0_and_A_1})  has no point spectrum within the interval $(a,b)$, i.e.
$\gs_{pp}(A_0) \cap (a,b) = \emptyset$, if and only if
\begin{equation}\label{4.10}
\lim_{y \downarrow 0}yM_{h_k}(x + iy) = 0  \quad  \text{for each} \quad  k \in\{ 1,2,\ldots,n\},
\end{equation}
for all $x \in (a,b)$. In this case the following relation holds
\begin{eqnarray}\label{4.11}
\lefteqn{\hspace{-1.5cm} \gs(A_0) \cap (a,b) = \gs_c(A_0) \cap (a,b) = }\\
& & \left(\overline{\cup^n_{k=1}\gO_{sc}(M_{h_k})} \; \cup \;
\overline{\cup^n_{k=1}\gO_{ac}(M_{h_k})}\right) \cap (a,b).\nonumber
\end{eqnarray}

\item[\;\;\rm(ii)] The operator  $A_0$  has no singular
continuous spectrum within the interval $(a,b)$, i.e. $\gs_{sc}(A_0) \cap (a,b) =
\emptyset$, if for each $k \in \{1,2,\ldots,n\}$ the set $\gO_{sc}(M_{h_k}) \cap (a,b)$
is at most countable. In particular, it happen if $(a,b) \setminus \gO_{ac}(M_{h_k})$ is
at most countable.
\end{theorem}

\subsection{Weyl function and scattering matrix}
\subsubsection{Direct integral and spectral representation}
Let $\kH$ be a separable Hilbert space and let $\kT = \{\kH_\gl\}_{\gl \in \bR}$ be a family of closed subspaces
$\kH_\gl \subseteq \kH$.
Assume that the family of  orthogonal projections $P(\gl)$  from $\kH$ onto $\kH_\gl$  defines a
weakly  measurable family of projections,
\begin{displaymath}
(Pf)(\gl) := P(\gl)f(\gl), \qquad f \in L^2(\bR,d\gl;\kH).
\end{displaymath}
This family  defines the measurable family $\{\kH_\gl\}_{\gl \in \bR}$ of subspaces.
 We set
\begin{displaymath}
L^2(\bR,d\gl;\kH_\gl) := PL^2(\bR,d\gl;\kH)
\end{displaymath}
and note that  $P$ defines an orthogonal projection from $L^2(\bR,d\gl;\kH)$ onto $L^2(\bR,d\gl;\kH_{\lambda})$.
Recall that $L^2(\bR,d\gl;\kH_\gl)$ is called the direct integral associated with the measurable family
 of subspaces $\kT$ (see \cite[Ch. 7]{BirSol1987}).

In the direct integral $L^2(\bR,d\gl;\kH_\gl)$
one defines  the multiplication $\kM_\kT$,
\begin{equation*}
\begin{split}
&(\kM_\kT f)(\gl) = \gl f(\gl), \\
& \dom(\kM_\kT) := \{f \in L^2(\bR,d\gl;\kH_\gl): \gl f(\gl) \in L^2(\bR,d\gl;\kH)\}.
\end{split}
\end{equation*}
It is  well known  (see \cite[Chapter 7]{BirSol1987}) that  the operator $\kM_\kT$ is self-adjoint.

Let $T = T^*\in  \mathcal C(\gotH)$ be an operator with $ac$-spectrum.
Then the  direct integral $L^2(\bR,d\gl;\kH_\gl)$ is called a spectral representation of $T$
if there is an isometry $V$  from $\gotH$ onto $L^2(\bR,d\gl;\kH_\gl)$  such that $VTV^* = \kM_\kT.$
This implies that  for any operator  $G (\in  \mathcal{B}(\gotH))$ commuting   with $T$,
the operator
$\kM_\kT(G) := VGV^*$ commutes with $\kM_\kT$, hence is the multiplication operator:
\begin{equation}\label{eq.decomposable_oper_in_dir_integr}
(\kM_\kT(G)f)(\gl) := G_\kT(\gl)f(\gl), \quad f \in L^2(\bR,d\gl;\kH_\gl),
\end{equation}
Here $\{G_\kT(\gl)\}_{\gl \in \bR}$,  is a measurable family
of  operators, $G_\kT(\gl)\in \mathcal{B}(\kH_\gl).$

It is known \cite{BirSol1987}   that
the commutant $\{\kM_\kT\}'$ of the operator  $\kM_\kT$ consists of bounded operators
of the form  \eqref{eq.decomposable_oper_in_dir_integr} which are called decomposable.
The operator  $G$ in  $\gotH$ is unitary  if and only if
$G_\kT(\gl)$ is unitary  for a.e. $\gl \in \bR$.
\begin{remark}\label{rem:2.9}
\item[\;\;\rm(i)]
If $\kH_\gl = \kH$ for $\gl \in \bR$, then the direct integral $L^2(\bR,d\gl;\kH_\gl)$
becomes  $L^2(\bR,d\gl;\kH_\gl) = L^2(\bR,d\gl;\kH)$;

\item[\;\;\rm(ii)]
Let
$\kF_\kT := \{\gl \in \bR: \dim(\kH_\gl) > 0\}$. Then we set
$L^2(\kF,d\gl;\kH_\gl) := L^2(\bR,d\gl;\kH_\gl).$
If $\kH_\gl = \{0\}$ for $\gl \in (-\infty,0)$ and
$\kH_\gl = \kH$ for $\gl \in \bR_+$, then
$L^2(\bR,d\gl;\kH_\gl) = L^2(\bR_+,d\gl,\kH).$
\end{remark}

\subsubsection{Scattering matrix}
As above, let  $A$ be a densely defined closed
symmetric operator  in  $\gotH$ with  deficiency indices $\mathrm{n}_{\pm}(A) = n <\infty$. Let
$\Pi=\{\kH,\Gamma_0,\Gamma_1\}$ be a boundary triplet for $A^*$ and
let $\gamma(\cdot)$ and $M(\cdot)$ be the corresponding
$\gamma$-field and Weyl function, respectively.  Let  $A_0 = A^*\!\upharpoonright\ker(\Gamma_0) = A_0^*$,
and  $A_\Theta=A^*\upharpoonright\Gamma^{-1}\Theta  \in \Ext_A$ where $\Theta = \Theta^* \in\widetilde{\mathcal{C}}(\kH)$
(see formula \eqref{II.1.2_01A}).

Since  $\mathrm{n}_{\pm}(A) = n <\infty$,  the Kato-Rosenblum theorem (see \cite{AkhGlz81}, \cite{ReedSim80}, \cite{Yaf92})
ensures the existence of the {\it wave operators}
\begin{equation*}
W_\pm(A_\Theta,A_0) := s - \lim_{t\to\pm\infty}e^{itA_\Theta}e^{-itA_0}P^{ac}(A_0),
\end{equation*}
and their completeness.   Here $P^{ac}(A_0)$ denotes the orthogonal
projection onto the ac-subspace $\gotH^{ac}(A_0)$
of $A_0$. Completeness means that the ranges of
$W_\pm(A_\Theta,A_0)$ coincide with the ac-subspace $\gotH^{ac}(A_\Theta)$
of $A_\Theta$, cf. \cite{BaumWall83,Kato66,Weid2003}.
The {\it scattering operator} $S(A_\gT,A_0)$ of the {\it scattering system}
$\{A_\Theta,A_0\}$ is
\begin{equation}\label{st}
S(A_\gT,A_0):= W_+(A_\Theta,A_0)^*W_-(A_\Theta,A_0).
\end{equation}
The scattering operator regarded as an operator in
$\gotH^{ac}(A_0)$ is unitary and commutes with the ac-part
$A^{ac}_0:=A_0\upharpoonright \dom(A_0)\cap\gotH^{ac}(A_0)$ of $A_0$.

Since the deficiency indices of $A$ are finite,
the limits  $M(\gl) := \lim_{y\downarrow 0}M(\gl + iy)$
exists for a.e $\gl \in \bR$ in the operator norm of  $\kH$. Setting
\begin{equation}\label{eq:2.23}
\kH_\gl := \overline{\ran(\im M(\gl))}, \quad \gl \in \bR,
\end{equation}
we introduce   the  family $\kT =\{\kH_\gl\}_{\gl \in \bR}$ of subspaces of $\kH$
which is well-defined for a.e. $\gl \in \bR$ and measurable.
 Let $L^2(\bR,d\gl;\kH_\gl)$ be the corresponding direct integral associated with
the family of subspaces  $\{\kH_\gl\}_{\gl \in \bR}$. Then the scattering operator
$S(A_\gT,A_0)$ of the pair $\{A_\Theta,A_0\}$ induces  the decomposable
operator  (the $S(A_\gT,A_0)$-matrix)
in  $L^2(\bR,d\gl;\cH_\gl)$
of the form \eqref{eq.decomposable_oper_in_dir_integr}:
$$(\kM_\kT(S(A_\gT,A_0))f)(\gl) := S_\kT(A_\gT,A_0;\gl)f(\gl), \quad
f \in L^2(\bR,d\gl;\kH_\gl).$$
\begin{theorem}[{\cite[Theorem 3.8]{BerMalNei08}}]\label{scattering}
Let $A$ be as above and let
 $\Pi= \{\kH,\Gamma_0,\Gamma_1\}$ be a boundary triplet for $A^*$
with the corresponding Weyl function $M(\cdot)$.  Further, let $\kT = \{\cH_\gl\}_{\gl\in\bR}$
be given  by \eqref{eq:2.23},  let
$A_0=A^*\!\upharpoonright\ker(\Gamma_0)$,
$A_\Theta=A^*\upharpoonright\Gamma^{-1}\Theta = A_\Theta^*$, and let
$S(A_\gT,A_0)$ be the scattering operator of the scattering system $\{A_\gT,A_0\}$.
Then the following holds:

\item[\;\;\rm(i)] The direct integral $L^2(\bR,d\gl;\cH_\gl)$ is a spectral representation of $A^{ac}_0$;

\item[\;\;\rm(ii)] Let  $\{S_\kT(A_\gT,A_0;\gl)\}_{\gl \in \bR}$ be the scattering matrix
of the scattering system  $\{A_\Theta,A_0\}$
with respect to $L^2(\bR,d\gl;\cH_\gl)$. Then for a.e. $\gl \in \bR$
\begin{equation}\label{scatformula}
\begin{split}
S_\kT&(A_\gT,A_0;\gl) \\
=& I_{\kH_\gl} +
2i\sqrt{\im M(\gl)}\bigl(\Theta-M(\gl)\bigr)^{-1}
\sqrt{\im M(\gl)}\in \mathcal{B}(\cH_\lambda). %%\ \text{for a.e.}\ \gl \in \bR.
\end{split}
\end{equation}
\end{theorem}
\begin{remark}\label{rem:2.11}
{\rm
In Section \ref{sect.ScatMatr}  it happen that  $\im M(\gl)$ admits   the factorization
\begin{displaymath}
\im M(\gl) = Z(\gl)^* Z(\gl)\qquad  \text{for a.e.} \   \gl \in \bR,
\end{displaymath}
where $\{Z(\gl)\}_{\gl \in \bR}$ is a measurable family
of bounded operators $Z(\gl) \in \mathcal{B}(\cH_\lambda)$.
In that case there is a measurable family  $\{V(\gl)\}_{\gl \in \bR}$ of partial
isometries  from $\mathcal{Z}_\gl := \overline{\ran(Z(\gl))}$ onto
$\kH_\gl$ such that  the polar decomposition holds
\begin{equation}\label{eq:2.25}
\sqrt{M(\gl)} = V(\gl)Z(\gl)  \qquad  \text{for a.e.} \   \gl \in \bR.
\end{equation}
Combining  \eqref{eq:2.25} with  representation  \eqref{scatformula} yields
\begin{displaymath}
S_\kT(A_\gT,A_0)(\gl) = I_{\kH_\gl} +
2iV(\gl)Z(\gl)\bigl(\Theta-M(\gl)\bigr)^{-1}
Z(\gl)^*V(\gl)^*
\end{displaymath}
for a.e. $\gl \in \bR$. The family of subspaces  $\mathcal{Z} = \{\mathcal{Z}_\gl\}_{\gl\in\bR}$
is  measurable.
Let us introduce the direct integral
$L^2(\bR,d\gl;\mathcal{Z}_\gl)$ and  the multiplication operator
\begin{displaymath}
(\mathcal{V} f)(\gl) : =V(\gl)f(\gl), \qquad f \in L^2(\bR,d\gl;\mathcal{Z}_\gl).
\end{displaymath}
Clearly, $\mathcal{V}$ is an isometry  from $L^2(\bR,d\gl;\mathcal{Z}_\gl)$ onto $L^2(\bR,d\gl;\kH_\gl)$
and  $\mathcal{V}^*\kM_\kT \mathcal{V} = \kM_\kZ,$ hence
the direct integral $L^2(\bR,d\gl;\kZ_\gl)$ is also a spectral representation of $A^{ac}_0$. Further,
let $\{S_\kZ(A_\gT,A_0)(\gl)\}_{\gl \in \bR}$ be the $S(A_\gT,A_0)$-matrix with respect to $L^2(\bR,d\gl;\kZ_\gl)$.
Then
\begin{displaymath}
S_\kZ(A_\gT,A_0)(\gl) = I_{\kZ_\gl} +
2iZ(\gl)\bigl(\Theta-M(\gl)\bigr)^{-1}
Z(\gl)^*  \quad \text{for a.e.}\quad  \gl \in \bR.
\end{displaymath}
Comparing this formula with  \eqref{scatformula}    yields
\begin{displaymath}
S_\kT(A_\gT,A_0)(\gl) = V(\gl)S_\kZ(A_\gT,A_0)(\gl)V(\gl)^* \quad \text{for a.e.}\quad  \gl \in \bR.
\end{displaymath}
}
\end{remark}

\section{Vector-valued Sturm-Liouville operators}\label{sect.GMNP}

In this section we recall the main results from the paper \cite{GMNP17}.
We will assume that $Q(\cdot)=Q(\cdot)^*\in L^1(\mathbb{R}_+;\mathbb{C}^{m\times m}).$
Note that self-adjointness of a potential matrix $Q(\cdot)$ means that $Q(x)=Q(x)^*$ for a.e. $x\in\mathbb{R}_+.$
Let us consider the Sturm-Liouville differential expression
$$
\mathcal{L}_Q(f(x)):=-\frac{d^2 f(x)}{dx^2} + Q(x)f(x), \quad f=(f_1,\ldots,f_m)^T, \quad x\in\mathbb{R}_+.
$$
We are interested in matrix-valued solutions $Y(x,z)$ of the equation
  \begin{equation}\label{eq:10}
\kL_Q(Y(x,z)) = zY(x,z) , \quad x\in\R_+, \quad z \in \C.
  \end{equation}
Let $C(x,z)$ and $S(x,z)$ be the matrix-valued solutions of the equation  \eqref{eq:10}
satisfying the initial  conditions
\begin{displaymath}
C(0,z) = S'(0,z)= I_m \quad \text{and}\quad   S(0,z)= C'(0,z)=\mathbb{O}_m, \quad z \in \C.
\end{displaymath}

Let $Q(\cdot) = Q(\cdot)^* \in L^1(\R_+;\C^{m\times m})$. Denote by $AC_{\loc}(\R_+)$
the set of locally absolutely continuous functions  on $\R_+$, i.e. $f\in
AC_{\loc}(\R_+)$ if $f\in AC[0, b]$ for any $b\in \R_+$.  We set
\begin{align}
A_{\max}f &:= \kL_Q(f), \quad x \in \R_+,\quad f \in \dom(A_{\max}),\label{eq:3.34}\\
\dom(A_{\max}) &:= \left\{f \in L^2(\R_+;\C^m):
\begin{matrix}
&f,f' \in AC_{\loc}(\R_+; \C^m),  \\
& \kL_Q(f) \in L^2(\R_+;\C^m)
\end{matrix}
\right\}  \nonumber
\end{align}
and  note (see \cite{AkhGlz81}, \cite{Najm1968}) that the operator $A_{\max}$ coincides with
adjoint operator $A^*$ associated with expression  $\kL_Q$.
The domain of the minimal operator $A:=A_{\min}$ is given by
\begin{equation}
\dom(A) := \dom(A_{\min})= \left\{f \in \dom(A_{\max}): f(0)=f'(0)=0
\right\}. \nonumber
\end{equation}
It is important to note that the operator $A$ is simple.
\begin{lemma}[{\cite[lemma 3.5]{GMNP17}}]\label{lemma 3.15}
Let $Q(\cdot)= Q(\cdot)^* \in L^1(\R_+;\C^{m\times m})$.  Then a triplet
$\Pi=\left\{\mathcal{H}, \Gamma_0, \Gamma_1 \right\}$ with
   \begin{equation}\label{eq: 53-1}
\mathcal{H}=\mathbb{C}^{m}, \quad \Gamma_{0}f=f(0), \quad \Gamma_{1}f=f'(0), \quad
f=\left(f_1,...,f_m\right)^{T}\in{\dom}(A_{\max}),
  \end{equation}
    is a boundary triplet for the operator $A^* = A_{\max}$.

\end{lemma}
\begin{proposition}[{\cite[Proposition 3.6]{GMNP17}}]\label{lemma 3.16-1}
Let $Q(\cdot) = Q(\cdot)^* \in L^1(\R_+;\C^{m\times m})$ and let $M(\cdot)$ be the Weyl
function corresponding to the boundary triplet \eqref{eq: 53-1}.
Let $\widehat{\mathbb{C}}$ be the complex plane with a cut along $[0,\infty)$.
Then the following holds:

\item[\;\;\rm(i)] The matrix-valued functions $N_1(z)$,
\begin{equation}\label{eq: 55}
N_{1}(z)=\frac{I_m}{2i\sqrt{z}}+\frac{1}{2i\sqrt{z}}\int\limits_{0}^{\infty}e^{it\sqrt{z}}Q(t)S(t, z)dt,
\end{equation}
and
\begin{equation}\label{eq: 56}
N_{2}(z)=\frac{I_m}{2}-\frac{1}{2i\sqrt{z}}\int\limits_{0}^{\infty}e^{it\sqrt{z}}Q(t)C(t, z)dt,
\end{equation}
are both well defined for $z \in \widehat{\mathbb{C}}_0:=\widehat{\mathbb{C}}\setminus\{0\}$
and continuous as well as holomorphic in $\C \setminus [0,\infty)$;

\item[\;\;\rm(ii)] The following relation holds
  \begin{equation}\label{eq: 54}
N_1(z) M(z) =  N_2(z), \quad z\in \C_+.
  \end{equation}
 \end{proposition}

\begin{lemma}[{\cite[Lemma 3.7]{GMNP17}}]\label{lem:3.7}
Let $Q(\cdot) = Q(\cdot)^*\in L^1(\R_+;\C^{m\times m})$.
Let also $N_1(\cdot)$ and $N_2(\cdot)$ be the functions given  by \eqref{eq: 55} and
\eqref{eq: 56}, respectively. In addition, the following holds:

\item[\;\;\rm(i)] The non-tangential limits
$N_1(\lambda) = \lim_{z\to\!\succ \gl}N_1(z)$ and
$N_2(\lambda) = \lim_{z\to\!\succ \gl}N_2(z)$ ,
exist and are invertible for any $\lambda\in \R_+$;

\item[\;\;\rm(ii)]
For each bounded interval $[a,b] \subseteq \R_+$ there is a rectangle
$\kR(a,b;\varepsilon) := [a,b] \times [0,\varepsilon]$, $\varepsilon > 0$ such that $N_1(\cdot)^{-1}$ and $N_2(\cdot)^{-1}$ exist and are continuous
in $\kR(a,b;\varepsilon)$. In particular, it holds
 \begin{equation}\label{eq:47}
\lim_{z\to\!\succ \gl}N_1(z)^{-1} = N_1(\lambda)^{-1},\qquad
\lim_{z\to\!\succ \gl}N_2(z)^{-1} = N_2(\lambda)^{-1}, \quad  \gl \in \R_+.
\end{equation}
\end{lemma}

\begin{theorem}[{\cite[Theorem 3.8]{GMNP17}}]\label{corollary 141}
Let $Q(\cdot) = Q(\cdot)^* \in L^1(\R_+;\C^{m\times m})$. Further, let ${\bf H}_D$ be the
Dirichlet realization of \eqref{eq:10} and let $M(\cdot)$ be the Weyl function
corresponding to the boundary triplet \eqref{eq: 53-1}. Then the following holds:

\item[\;\;\rm (i)]  The non-tangential boundary values $M(\lambda + i0) := \lim_{z\to\!\succ \gl}M(z)$ exist
\emph{for each} $\lambda\in \R_+$ and
  \begin{equation}\label{eq: 141}
    \begin{split}
M(\lambda + i0) =& N_1(\gl)^{-1}N_2(\gl),
\quad \lambda\in \R_+.
\end{split}
\end{equation}
In particular, one has
    \begin{equation}\label{eq:3.51}
    \text{\emph{Im}}(M(\lambda + i0)) = \frac{1}{4\sqrt{\gl}}\left(N_{1}(\lambda)^*
N_{1}(\lambda)\right)^{-1}, \quad \gl \in \R_+.
\end{equation}

\item[\;\;\rm (ii)] The determinant $d_1(z) = \det(N_1(z))$ is holomorphic in $\C \setminus [0,\infty)$ and the set
of its zeros $\gL_1$ is discrete. The Weyl function admits the representation
\begin{equation}\label{eq:3.52}
M(z) = N_1(z)^{-1}N_2(z), \quad z \in \C \setminus \gL_1.
\end{equation}

\item[\;\;\rm (iii)]   The corresponding spectral  measure  $\Sigma_M(\cdot)$ (see \eqref{WF_intrepr}) on
$\R_+$ is  given by
\begin{equation}
\Sigma_M(t) =
 \frac{1}{4\pi}\int\limits_{0}^{t}\frac{1}{\sqrt{\lambda}}\left(N_{1}(\lambda)^* N_{1}(\lambda)\right)^{-1}d\lambda.
\end{equation}
In particular, $\Sigma_M(\cdot)$   is absolutely continuous with \emph{continuous}
density $d\Sigma_M(\lambda)/d\lambda$  of maximal rank, i.e.
$\text{rank} (d\Sigma_M(\lambda)/d\lambda) =m$ for every  $\lambda\in \R_+$.
\end{theorem}

Similar result is also valid for any selfadjoint extension $\widetilde A$ of the minimal operator $A$. In particular,
the  positive part $\widetilde A E_{\widetilde A}(0,\infty)$ of $\widetilde A$ is purely absolutely continuous. In turn, this results implies the following one.
    \begin{corollary}[{\cite[Proposition 4.1]{GMNP17}}]\label{corof 3.4}
Let $Q = Q^*\in L^1(\mathbb{R}_+;\mathbb{C}^{m\times m})$.
Then the operator $A^*=A_{\emph{max}}$ has no positive eigenvalues, i.e. $\sigma_p(A_{\emph{max}})\cap\mathbb{R}_+=\emptyset$.
  \end{corollary}

\section{General non-compact  quantum graphs with finitely many edges}\label{sect.quagr}

\subsection{Framework}\label{frame}
Let us set up the framework. Let $\mathcal{G}=(\mathcal{V}, \mathcal{E})$ be a noncompact metric and connected graph,
consisting of finitely many edges $\mathcal{E}$  and  vertices  $\mathcal{V}$. Since graph $\mathcal{G}$ is noncompact,
at least one of the edges  has infinite length. For two vertices $v,u\in \mathcal{V}$ we  write $v\sim u$ if there is an edge $e_{u,v}\in \mathcal{E}$ connecting $v$ with $u.$ For every $v\in \mathcal{V},$ we denote the set of edges incident to the vertex $v$ by $\mathcal{E}_v$ and
\begin{equation}\label{eq_deg}
\text{deg}(v)= \#\{e: e\in \mathcal{E}_v\}
\end{equation}
is called the degree of a vertex $v\in \mathcal{V}.$
Let us note that the vertices $v\in \mathcal{V}: \text{deg}(v)=1$ are called the loose ends.
A path $\mathcal{P}$ of length $n\in\mathbb{N}$ is a subset of vertices $\{v_0, v_1,...,v_n\}\subset \mathcal{V}$ such that $n$ vertices $\{v_0, v_1,...,v_{n-1}\}$ are distinct and $v_{k-1}\sim v_k$ for all $k\in\{1,...,n\}.$ A graph $\mathcal{G}$ is called connected if for any two vertices $u$ and $v$ there is a path $\mathcal{P}=\{v_0, v_1,...,v_n\}$ connecting $u$ and $v,$ that is, $u=v_0$ and $v=v_n.$

Let us assign each finite edge $e\in \mathcal{E}_{\text{fin}}$ with length $|e|\in (0,\infty)$ and direction, that is, each finite edge $e\in \mathcal{E}_{\text{fin}}$
has one initial $v_o$ and one terminal vertex $v_{in}.$
Also we assign each infinite edge (lead) $e\in \mathcal{E}_{\infty}$ with $[0,\infty),$ and every finite edge $e\in\mathcal{E}_{\text{int}}$   with the interval $(0, |e|)$;
we also denote $\mathcal{E}=\mathcal{E}_{\text{fin}}\cup \mathcal{E}_{\infty}.$
In the sequel we assume that graph $\mathcal{G}$ consists of $p_1>0$ leads and $p_2\geq 0$ finite edges,
i.e. $|\mathcal{E}_{\infty}|=p_1$ and $|\mathcal{E}_{\text{fin}}|=p_2$, and put  $p:=p_1+p_2$.
Let $\mathcal{G}_0\subset\mathcal{G}$ be a maximal
compact subgraph of the metric graph  $\mathcal{G}$, i.e. $\mathcal{G}_0$ consists of  $p_2$  edges. Moreover,
each edge is equipped with coordinate (denoted $x$) that identifies  this edge with a bounded interval.
We choose some subset $\mathcal{V}_{\text{ext}}$ of vertices of $\mathcal{G}_0,$ to be called external vertices, and attach one or more copies of $[0,\infty)$ to each external vertex; the point 0 in a lead is thus identified with the relevant external vertex.  Denote by  $\mathcal{V}_{\text{int}}=\mathcal{V}\setminus \mathcal{V}_{\text{ext}}$  the set of internal vertices.

First we introduce the Hilbert space $L^2(\mathcal{G};\mathbb{C}^m)$ of functions $f: \mathcal{G}\rightarrow\mathbb{C}^m$
by setting:
  \begin{equation}\label{Hs}
L^2(\mathcal{G};\mathbb{C}^m) = \bigoplus_{e\in \mathcal{E}} L^2(e;\mathbb{C}^m)=\left\{f=\{f_{e}\}_{e\in \mathcal{E}}: f_{e}\in L^2(e;\mathbb{C}^m)\right\},
\end{equation}
where
  \begin{equation}\label{f-vec}
f_{e}=(f_{e,1},...,f_{e,m})^T=\begin{pmatrix}f_{e,1}\\ \vdots\\ f_{e,m}\end{pmatrix} \in L^2(e;\mathbb{C}^m).
  \end{equation}

Let us equip $\mathcal{G}$ with the Sturm-Liouville operator. Assume also that for each edge $e\in \mathcal{E}$
  \begin{equation}\label{aux-3.1.1}
Q_{e}(\cdot)=Q_{e}(\cdot)^*\in L^1(e;\mathbb{C}^{m\times m}),
 \end{equation}
and equality in \eqref{aux-3.1.1} is understood  in the following sense:
$Q_{e}(x)=Q_{e}(x)^*$ for a.e. $x\in e$.

For every $e\in \mathcal{E}$ consider the maximal operator $A_{e,\text{max}}$ associated with the
Sturm-Liouville differential expression
\begin{equation}\label{LQ}
\mathcal{A}_{e}:=-\frac{d^2}{dx^2}+Q_{e}
\end{equation}
on the domain
\begin{equation}\label{He}
\dom(A_{e,\max})=
\left\{f_{e} \in L^2(e;\C^m):
\begin{matrix}
&f_{e},f'_{e} \in AC_{\loc}(e; \C^m),  \\
& \kA_{e}(f_e) \in L^2(e;\C^m), %%\\
\end{matrix}
\right\}, \quad e\in \mathcal{E}.
\end{equation}
It is well known (see for instance \cite[Proposition 9.5(i)]{DerMal17})
that for each $e\in \mathcal{E}_{\text{fin}}$ $A_{e,\max}$ has the following regularity property:
  \begin{equation}\label{dom_H_max}
\dom(A_{e,\max}) = \{f_e\in W^{2,1}(e;\mathbb{C}^m):\mathcal{A}_{e}(f_{e})\in L^2(e;\mathbb{C}^m)\}\subset W^{2,1}(e;\mathbb{C}^m).
  \end{equation}
Therefore there exist the limit values
$f_e(v_o), f_e(v_{in}), f'_e(v_o), f'_e(v_{in})$ for each $e\in \mathcal{E}_{\text{fin}}$.
Moreover, the regularity  property similar to \eqref{dom_H_max} remains valid for
infinite edges (leads) $e\in \mathcal{E}_{\infty}$ provided that $Q_e\in L^1(e; \mathbb{C}^{m\times m})$.
Therefore  there exist  the limit values $f_e(v), f'_e(v)$  for each $e\in \mathcal{E}_{\infty}$
and $v\in \mathcal{V}_{\text{ext}}$.

For every $e\in \mathcal{E}$ we define the minimal operator $A_e:=A_{e,\min}$ associated in $L^2(e;\mathbb{C}^m)$ with differential expression \eqref{LQ}.
In the case of lead $e\in \mathcal{E}_{\infty}$ the domain of the minimal operator $A_e$ is of the form
\begin{equation}\label{DOMa}
\dom(A_e) = \left\{f_e \in \dom(A_{e,\max}):
f_e(v)=f'_e(v) =0, \quad v\in \mathcal{V}_{\text{ext}}
\right\}.
\end{equation}
In the case of finite edge $e\in \mathcal{E}_{\text{fin}}$ the domain of the minimal operator $A_e$ is
\begin{equation}\label{DOMa-1}
\dom(A_e) = \left\{f_e \in \dom(A_{e,\max}):
f_e(v_o)=f_e(v_{in})=f'_e(v_o) =f'_e(v_{in})=0
\right\}.
\end{equation}
Clearly, $A_e$ is closed symmetric operator, and $A_e^*=A_{e,\max}$ (see Section 3).

Note also  that for $Q_e(\cdot)\in L^2(e;\mathbb{C}^{m\times m})$, instead of \eqref{dom_H_max}
we have equality  $\dom(A_{e,\max})=W^{2,2}(e;\mathbb{C}^m),$\  $e\in \mathcal{E}.$

Finally, in $L^2(\mathcal{G};\mathbb{C}^m)$ we define the closed minimal symmetric operator
\begin{equation}\label{symop}
A:=A_{\min}:=\bigoplus_{e\in \mathcal{E}}A_e=A_{\mathcal{E}_{\infty}}\bigoplus A_{\mathcal{E}_{\text{fin}}}
\end{equation}
associated with the expression $\mathcal{A}:=\bigoplus_{e\in \mathcal{E}}\mathcal{A}_e.$
Since $A_{e,\max}=A_e^*$, $e\in \mathcal{E},$ the adjoint operator is given by $A^*=A_{\max}=\bigoplus_{e\in \mathcal{E}}A_e^*.$

\subsection{Absence of singular continuous spectrum for any realization}\label{abssc}

Now we state the main result of the section ensuring the absence of singular continuous
spectrum for each  self-adjoint extension $\wt A\in \Ext_A$.
   \begin{theorem}\label{sc_spec}

Assume that graph $\mathcal{G}$ consists of $p_1>0$ leads and $p_2\geq 0$ edges.
Let $Q_{e_j}:=Q_j\in L^1(e_j;\mathbb{C}^{m\times m})$
for  $j\in\{1,...,p\}.$
Let also $\widetilde{A}$ be an arbitrary self-adjoint extension of the minimal operator $A:=\bigoplus_{e\in \mathcal{E}}A_e$ defined
by \eqref{symop}. Then:

\item[\;\;\rm (i)] The singular continuous spectrum $\sigma_{sc}(\widetilde{A})$ is empty, i.e. $\sigma_{sc}(\widetilde{A})\cap\mathbb{R}=\emptyset$;

\item[\;\;\rm (ii)] The absolutely continuous spectrum $\sigma_{ac}(\widetilde{A})$ fills in the half-line
$\mathbb{R}_+$, \\ $\sigma_{ac}( \widetilde{A}) =[0,\infty),$  and is of the constant multiplicity $mp_1$;

\item[\;\;\rm (iii)] Operator $\widetilde{A}$ is semi-bounded from below and its negative spectrum is either finite
or forms countable sequence tending to zero;

\item[\;\;\rm (iv)] Possible positive eigenvalues embedded in the $ac$-spectrum $\sigma_{ac}(\widetilde{A})$
form a discrete set, i.e. the set  $\sigma_{pp}(\widetilde{A})\cap\mathbb{R_+}$ is discrete.
   \end{theorem}
\begin{remark}
We emphasize  that in contrast to the situation with absolutely continuous spectrum,
 the singular spectrum is unstable with respect to one-dimensional
perturbations (see \cite{ReedSim80}). There exist such operators
$-\frac{d^2}{dx^2}+Q$ in $L^2(\mathbb{R}_+;\mathbb{C}^m)$ for which the Dirichlet
realization has absolutely continuous spectrum while the  Neumann one  contains
sc-component (see \cite{Aron57} and recent publication  \cite{Malamud2019}).
This phenomenon substantially depends on the behavior of the Weyl function and does not
occur under the conditions of Theorem \ref{sc_spec}.
\end{remark}
   \begin{proof}
Let us briefly explain the idea of the proof. We fix  an extension $A_B$ of the
minimal operator $A$ and investigate the non-tangential boundary behavior of the
corresponding Weyl function $(B-M(\cdot))^{-1}$ relying on  Theorem \ref{corollary
141}.
%%The graph is a combination of finitely many edges, each of which does not have
%%sc-component for any choice of vertex conditions.}

$(\text{i})_1.$ Let $A_j$ be the minimal operator associated with differential expression $\kA_{j}$ on
$L^2(e_j;\C^m)$, $j\in\{1, \ldots, p_1\}$,  for  each lead  $e_j\in \mathcal{E}_{\infty}$. Due to the assumption
$Q_j \in L^1(e_j;\C^{m\times m})$ one has $\mathrm{n}_{\pm}(A_j) = m$.
Moreover,  it is easily seen  that the operator $A_{j}$, $j\in\{1, \ldots, p_1\}$,
is unitarily equivalent to the minimal operator  defined in $L^2(\mathbb{R}_+;\C^m)$ by
the differential  expression \eqref{eq:3.34}  with the matrix potential $Q_{j}(\cdot).$

By Lemma \ref{lemma 3.15} a triplet  $\Pi_j=\{\mathcal{H}_j, \Gamma_0^{(j)},
\Gamma_1^{(j)}\}$ with
  \begin{equation}\label{eq: 410}
\mathcal{H}_{j}=\mathbb{C}^m, \quad \Gamma_0^{(j)} f_{e_j}=f_{e_j}(v), \quad \Gamma_1^{(j)} f_{e_j}=f'_{e_j}(v),
\quad v\in \mathcal{V}_{\text{ext}},
  \end{equation}
is a boundary triplet for $A_{j}^{\ast}$.  Let $M_{j}(\cdot)$, $j\in\{1,...,p_1\}$,  be the  corresponding Weyl function.
By Proposition \ref{lemma 3.16-1}, each  $M_j(\cdot)$ admits the following representation
    \begin{equation}\label{eq:Weyl_fNew}
N_{1,j}(z)M_j(z) = N_{2,j}(z), \quad z \in \widehat{\mathbb{C}}_0, \qquad j\in\{1,...,p_1\},
 \end{equation}
where
\begin{equation}\label{eq:M1_M2New}
\begin{gathered}
N_{1,j}(z) :=\frac{I_m}{2i\sqrt{z}} + \frac{1}{2i\sqrt{z}}\int_{e_j} e^{i\sqrt{z} t}Q_j(t)S_j(t, z)dt,\\
N_{2,j}(z) :=\frac {I_m}{2} - \frac{1}{2i\sqrt{z}}\int_{e_j} e^{i\sqrt{z} t}Q_j(t)C_j(t, z)dt,
\end{gathered}
\end{equation}
and $S_{j}(\cdot,z)$, $C_{j}(\cdot,z)$ are the solutions of the equations
\begin{equation}\label{eq:4.13}
\kA_{j}(S_{j}(x,z)) = zS_{j}(x,z) \quad \mbox{and} \quad \kA_{j}(C_{j}(x,z)) = zC_{j}(x,z), \quad x \in e_j,
\end{equation}
satisfying  the initial  conditions
\begin{equation}\label{eq:4.14}
C_{j}(v,z) = S'_{j}(v,z) = I_m \quad\text{and}\quad S_{j}(v,z) = C'_{j}(v,z) = \mathbb{O}_m, \quad z \in \C.
\end{equation}

Let $A_j$ be the minimal operator associated with the differential expression $\kA_{j}$
on $L^2 (e_j;\C^m)$, $j \in \{p_1+1,\ldots,p\}$, i.e. $e_j\in \mathcal{E}_{\text{fin}}$. Due to the assumption
$Q_j \in L^1(e_j;\C^{m\times m})$ one has $\mathrm{n}_{\pm}(A_j) = 2m$,\ \  $j\in \{p_1+1,\ldots,p\}$. In accordance with \eqref{dom_H_max}
for every $f_{e_j}\in\dom(A_j^*)$
there are exist boundary values
$f_{e_j}(v_o), \ f_{e_j}(v_{in}), \ f'_{e_j}(v_o), \  f'_{e_j}(v_{in}).$
Therefore the boundary triplet $\Pi_j=\{\mathcal{H}_j, \Gamma_0^{(j)}, \Gamma_1^{(j)}\}$ for the operator $A_j^{\ast}$
can be chosen  in the following form:
\begin{equation}\label{eq: 400}
\mathcal{H}_j=\mathbb{C}^{2m}, \quad  \Gamma_0^{(j)} f_{e_j}=\begin{pmatrix}f_{e_j}(v_o)\\f_{e_j}(v_{in})\end{pmatrix}, \quad
\Gamma_1^{(j)} f_{e_j}=\begin{pmatrix}f'_{e_j}(v_o)\\f'_{e_j}(v_{in})\end{pmatrix}.
\end{equation}
The corresponding  Weyl function  is given by
\begin{equation}\label{eq: 3.24A}
M_j(z)=
\begin{pmatrix}
-S_{j}(v_{in},z)^{-1}C_{j}(v_{in},z)  & S_{j}(v_{in},z)^{-1}\\
(S_{j}(v_{in},\overline{z})^*)^{-1}   &   -S'_{j}(v_{in},z)S_{j}(v_{in},z)^{-1}
\end{pmatrix},
\quad z \in \C_\pm.
\end{equation}
Setting $\Pi:=\bigoplus_{j=1}^p \Pi_j=\{\mathcal{H},\Gamma_0,\Gamma_1\}$, where
\begin{equation}\label{triples}
\mathcal{H}:=\bigoplus_{j=1}^{p}\mathcal{H}_j:=\mathcal{H}_{\infty}\bigoplus\mathcal{H}_{\text{fin}},
\qquad \Gamma_a:=\bigoplus_{j=1}^{p}\Gamma_a^{(j)}, \quad a\in \{0,1\},
\end{equation}
and  $\Gamma_a^{(j)}$ are given  by  \eqref{eq: 410} and \eqref{eq: 400}  for $j\leq
p_1$ and  $j>p_1$, respectively,
we obtain a boundary triplet
for $A^*$. Clearly,
\begin{equation}\label{npmA}
\mathrm{n}_{\pm}(A)=\dim\left(\mathcal{H}_{\infty}\bigoplus\mathcal{H}_{\text{fin}}\right)=(p_1+2p_2)m.
\end{equation}

It is easily seen that the Weyl function corresponding to the boundary triplet $\Pi$ admits a decomposition $M(z)=M_{\infty}(z)\bigoplus M_{\text{fin}}(z),$
where
  \begin{equation}\label{Wfcomm-1}
M_{\infty}(z):=\bigoplus\limits_{j=1}^{p_1} M_j(z) \quad \text{and}\quad
M_{\text{fin}}(z):=\bigoplus\limits_{j=p_1+1}^{p} M_j(z)
  \end{equation}
The Weyl function $M_{\text{fin}}(\cdot)$ is a meromorphic matrix function in
$\C$ because due to \eqref{eq: 3.24A} each summand is meromorphic.
Hence the  singularities  of $M_{\text{fin}}(\cdot)$ constitute  at most countable discrete
set $\Omega_{\text{fin}} := \{\omega_r\}_{r=1}^\infty$ of real  poles.

$\text{(i)}_2.$ Assume for the beginning  that  $\wt{A} = \wt {A}^*$  is disjoint with $A_0 :=
A^*\upharpoonright \ker \Gamma_0$. Then, in accordance with Proposition
\ref{prop_II.1.2_01}(ii), there exists a bounded operator $B = B^*\in \mathcal{B}\left(\mathcal{H}\right)$
such that
  \begin{displaymath}
\wt{A} = A^*\upharpoonright \ker\left(\Gamma_1-B\Gamma_0\right)=A_B, \qquad
B\in\mathcal{B}\left(\mathcal{H}\right).
  \end{displaymath}
Further, let
\begin{equation}\label{eq: 200}
B=\left(\begin{array}{@{}cc@{}}B_{11}&B_{12}\\B_{12}^{\ast}&B_{22}\end{array}\right)=B^{\ast},
\quad B_{11}\in\mathbb{C}^{mp_1\times mp_1}, \quad B_{22}\in\mathbb{C}^{2mp_2\times 2mp_2},
 \end{equation}
be the block-matrix representation of  $B$  with respect to the orthogonal decomposition $\mathcal{H}=\mathcal{H}_{\infty}\bigoplus\mathcal{H}_{\text{fin}}$.

Then
\begin{equation}\label{eq: 230}
 B-M(z)=\left(\begin{array}{@{}cc@{}}B_{11}-M_{\infty}(z)&B_{12}\\B_{12}^{\ast}&B_{22}-M_{\text{fin}}(z)\end{array}\right).
\end{equation}
Since $M(\cdot)$ is the Weyl function, its imaginary part  $M_I(\cdot)$ is positive
definite in $\C_+$, i.e. $M_I(z)\ge \varepsilon(z)>0$ and
$0\in\rho\bigl(M_I(z)\bigr)$ for $z\in\C_+$. Hence both matrices
$$
\text{Im} \bigl(M(z) -
B\bigr) = \text{Im} M(z) = M_I(z) \quad \text{and} \quad
\text{Im} \bigl(M_{\text{fin}}(z)-B_{22}\bigr) = \text{Im} M_{\text{fin}}(z)
$$
are also positive definite for $z\in\C_+.$  Hence  both matrices
$M(z) - B$ and $M_{\text{fin}}(z) - B_{22}$ are  invertible for $z\in\C_+.$ Therefore the
inverse matrix $(B - M(z))^{-1}$ exists for each $z\in\C_+$ and can be computed  by
the Frobenious formula
 \begin{equation}\label{eq: 240}
 \left(B-M(z)\right)^{-1}=\left(\begin{array}{@{}cc@{}}M_B^{-1}(z)&K_{12}(z)\\K_{21}(z)&K_{22}(z)\end{array}\right),
 \qquad z\in\C_+,
  \end{equation}
 where
\begin{align}
M_B(z):=  & B_{11}-M_{\infty}(z)-B_{12}\left(B_{22}-M_{\text{fin}}(z)\right)^{-1}B_{12}^*, \label{eq: 3.31A} \\
K_{12}(z)=&-M_B^{-1}(z)B_{12}\left(B_{22}-M_{\text{fin}}(z)\right)^{-1}, \label{eq: 3.31C}\\
K_{21}(z)=&-\left(B_{22}-M_{\text{fin}}(z)\right)^{-1}B_{12}^*M_B^{-1}(z), \label{eq: 3.31D}\\
K_{22}(z)=&\left(B_{22}-M_{\text{fin}}(z)\right)^{-1}+\label{eq: 3.31B}\\
          &\left(B_{22}-M_{\text{fin}}(z)\right)^{-1}B_{12}^*M_B^{-1}(z)B_{12}\left(B_{22}-M_{\text{fin}}(z)\right)^{-1}. \nonumber
\end{align}

We are going to show that $\gs_{sc}(\wt {A})\cap \R_+ = \emptyset$.
Since the matrix function $B_{22}-M_{\text{fin}}(\cdot)$ is meromorphic in $\C$ and
invertible in $\C_+$, its determinant $d_1(z):=\det\bigl(B_{22}-M_{\text{fin}}(z)\bigr)$
has at most a countable set of isolated zeros in $\R\setminus\Omega_{\text{fin}}$. Denoting this
null set by $\Omega$, we get
   \begin{equation}\label{AcSpec4.20A}
\lim_{z\to\!\succ\gl}(B_{22}- M_{\text{fin}}(z))^{-1} = \left(B_{22}- M_{\text{fin}}(\gl)\right)^{-1},
\quad \lambda\in\mathbb{R}_+ \setminus
(\Omega \cup \Omega_{\text{fin}}).
   \end{equation}
It follows  from \eqref{eq: 3.31A}, \eqref{AcSpec4.20A}  and Theorem \ref{corollary 141} (i) that the
limit $M_B(\lambda+i0) := \lim_{z\to\succ\!\gl} M_B(z)$
exists for each $\lambda\in\mathbb{R}_+ \setminus (\Omega \cup \Omega_{\text{fin}})$ and
\begin{equation}\label{AcSpec4.21}
M_B(\lambda+i0)= B_{11} - M_{\infty}(\lambda+i0) - B_{12}\left(B_{22}-
M_{\text{fin}}(\gl)\right)^{-1}B_{12}^*,
  \end{equation}
$\lambda\in\mathbb{R}_+ \setminus (\Omega \cup \Omega_{\text{fin}})$. Since $B_{22} = B_{22}^*$, $B_{11} = B_{11}^*$,
and $M_{\text{fin}}(\lambda) =
M_{\text{fin}}(\lambda)^*$  for every $\lambda\in\mathbb{R}_+ \setminus (\Omega \cup
\Omega_{\text{fin}})$, the later  identity yields
\begin{equation}
\text{Im}\left(-M_B(\lambda+i0)\right)=\text{Im}\left(M_{\infty}(\lambda+i0)\right),  \quad
 \lambda\in\mathbb{R}_+ \setminus (\Omega \cup \Omega_{\text{fin}}). \label{eq: 500}
\end{equation}
Since $\text{Im}\left(M_{\infty}(\lambda+i0)\right)$ is invertible for $\lambda\in\mathbb{R}_+ \setminus (\Omega \cup \Omega_{\text{fin}})$
one gets that the limit $M_B(\lambda+i0)$  is invertible for each $\lambda\in\mathbb{R}_+ \setminus (\Omega \cup \Omega_{\text{fin}})$  and
\begin{equation}\label{AcSpec4.21A}
M_B(\gl + i0)^{-1} = \left( B_{11} - M_{\infty}(\lambda+i0) -
B_{12}\left(B_{22}- M_{\text{fin}}(\lambda)\right)^{-1}B_{12}^*\right)^{-1}
\end{equation}
for $\lambda\in
\mathbb{R}_+ \setminus (\Omega \cup \Omega_{\text{fin}}).$
Since the limit $M_B(\gl + i0)$ is invertible we get from
\eqref{AcSpec4.20A},\eqref{AcSpec4.21A} and formulas \eqref{eq: 240}--\eqref{eq: 3.31B} that the limit $B - M(\gl + i0) := \lim_{z\to\!\succ\gl}(B - M(z))$ exists and  for each $\gl \in \R_+ \setminus (\gO \cup \gO_{\text{fin}})$ is invertible,
and is given by
\begin{equation}\label{B-Weyl_func-n}
 B-M(\gl+i0)=
\begin{pmatrix}
B_{11}-M_{\infty}(\gl+i0) & B_{12}\\
B_{12}^{\ast}               & B_{22}-M_{\text{fin}}(\gl)
\end{pmatrix}\,.
\end{equation}
Moreover, we have
   \begin{equation}\label{Limit_cond-n}
 (B - M(\lambda +i0))^{-1} = \lim_{z\to\!\succ\gl}({B} - M(z))^{-1}, \quad \lambda \in \mathbb{R}_+ \setminus (\Omega \cup\Omega_{\text{fin}}).
   \end{equation}

Further, one easily checks that a triplet  $\Pi_B = \{\kH,\gG^B_0,\gG^B_1\}$  with
\begin{displaymath}
\gG^B_0 := \gG_1 - B\gG_0 \quad \mbox{and} \quad \gG^B_1 := - \gG_0
\end{displaymath}
is a boundary triplet for $A^*$ and $A^*\upharpoonright\ker(\gG^B_0) = \wt A$.
The corresponding Weyl function  coincides  with
$M_B(z) = (B - M(z))^{-1}$, $z \in \C_\pm$. Since the set $\Omega \cup \Omega_{\text{fin}}$ is at most countable,
Theorem \ref{IV.3}(ii) applies  and ensures  that $\gs_{sc}(A_B)\cap \R_+ = \emptyset$.

Let us show that $\gs_{sc}(\wt {A})\cap \R_- = \emptyset$. Let us denote by ${\bf H}_{j,D}$ the Dirichlet realization
of the differential expression $\kA_{j}$,  $j\in\{1,\ldots, p\}$, and put ${\bf H}_D:=\bigoplus_{j=1}^{p} {\bf H}_{j,D}$.
In accordance with Theorem \ref{corollary 141},(iv), we have $\gs_{sc}({\bf H}_{j,D})\cap \R_- = \emptyset$
for each $j\in \{1,\ldots,p_1\}$. Moreover, each  operator  ${\bf H}_{j,D}$  for \  $j\in \{p_1+1,\ldots,p\}$,  has discrete
spectrum. Thus, the direct sum ${\bf H}_D =\bigoplus_{j=1}^{p} {\bf H}_{j,D}$  has empty singular continuous spectrum on $\R_-$
too, i.e.  $\gs_{sc}({\bf H}_D)\cap \R_- = \emptyset$.

On the other hand,  the resolvent difference
  \begin{equation}\label{eq.3_finite_resol_difference}
(\wt {A} -i)^{-1} - ({\bf H}_D - i)^{-1}\qquad \text{is finite dimensional.}
    \end{equation}
To prove that  $\gs_{sc}(\wt {A})\cap \R_- = \emptyset$ it remains to
apply  the Weyl theorem on the stability  of the continuous spectrum  under compact
perturbations.

It remains to consider the case of $\wt A$  not disjoint with $A_0$. In this case it suffices to apply  Lemma 2.12 of \cite{MalNei2011}. We leave the details to the reader.

(ii)--(iii).  Similar results for the Dirichlet realization  ${\bf H}_D$ is proved in \cite{GMNP17}  (see Theorem \ref{corollary 141}). Now the results follow from \eqref{eq.3_finite_resol_difference} by applying  the  Kato-Rosenblum and Weyl theorems
on the stability  of the
 absolutely continuous and continuous spectra of ${\bf H}_D$ respectively, under trace class perturbations
(see \cite{AkhGlz81}, \cite{ReedSim80}, \cite{Yaf92}).

(iv). Since the matrix function $M_{\text{fin}}(z)$ is meromorphic, and the set
$\Omega\cup\Omega_{\text{fin}}$ is discrete,
it follows from \eqref{AcSpec4.21} that  the possible  positive eigenvalues of $\widetilde{A}$
can  accumulate  at infinity only.
\end{proof}
\begin{remark}
For the scalar quantum graph $(m=1)$ with zero potential $Q=0$ Theorem \ref{sc_spec}
was obtained earlier by Ong \cite{Ong2006} by using the limiting absorption principle.
\end{remark}
\begin{remark}
In conclusion we mention several papers  devoted to the investigation of ac-spectrum
of certain realizations of Schr\"odinger expression on
graphs  with \emph{infinitely  many edges}. In particular,
 E. Korotyaev and N. Saburova \cite{KorSab15}
shown that the spectrum of Schr\"odinger operators on periodic discrete graphs consists of an
absolutely continuous part and a finite number of eigenvalues of  infinite multiplicity.
The ac-spectrum has  band-zone structure with finitely many gaps.

P. Kuchment and O. Post \cite{KuchPost07} investigated
spectra of graphene operator $H$. Namely,
starting  with the Hill operator
$H^{\text{per}}:=-\frac{d^2}{dx^2}+q_0(x)$ on $L^2(\mathbb{R})$ with
even (on $[0,1]$) periodic potential $q_0(x) = q_0(x+1)$,
assuming   that $q_0(\cdot) \in L^2[0,1]$, and
using the fixed identification of the edges with $[0,1]$, the authors pullback the function $q_0(x)$ to
a potential $q(x)$ on the graphene. It is shown that the $sc$-spectrum  $\sigma_{sc}(H)$ is empty,
$\sigma_{sc}(H)= \emptyset$,  and the $ac$-spectrum $\sigma_{ac}(H)$ has band-gap structure and is given by
$$
\sigma_{ac}(H)=\sigma_{ac}(H^{\text{per}})=\left\{\lambda\in\mathbb{R}: |D(\lambda)|\leq 2\right\},
$$
where $D(\lambda)$ is the discriminant of $H^{\text{per}}$ (see \cite[Theorem 3.6]{KuchPost07}).

Besides,  the pure point spectrum $\sigma_{pp}(H)$ coincides with the spectrum
of the Dirichlet realization of  $-\frac{d^2}{dx^2}+q_0(x)$ in  $L^2(0,1)$
and thus, due to the evenness of $q_0$,
belongs to the union of the edges of spectral gaps of $\sigma(H^{\text{per}})=\sigma_{ac}(H)$.
  Moreover,  all eigenvalues are of infinite multiplicity.

We also  mention  the papers \cite{ESS14} and  \cite{KosNic19}  where the spectra
of radial trees and radial antitrees, respectively, are investigated.

In the first paper the authors consider a rooted radially symmetric metric tree $\mathcal{G}=(\mathcal{V},\mathcal{E})$ with the root $O\in\mathcal{V}$.
For any vertex $v\in\mathcal{V}$
one denotes by $b(v)$  the branching number of $v$, i.e. the number of forward nearest
neighbors of $v$,  and  by $t_n (>0)$
the length of the simple path connecting the vertex $v$ with the root $O$.
Assuming the conditions
\begin{equation}\label{tandb}
\inf_{n\in\mathbb{N}}(t_{n+1}-t_n)>0 \quad \text{and} \quad \inf_{n\in\mathbb{N}}b_n>1
\end{equation}
and  assuming the sequences
$\{t_{n+1}-t_n\}_1^{\infty}$ and $\{b_n\}_1^{\infty}$ to be  finite
it is shown in \cite{ESS14} that for  the Hamiltonian  ${\bf H}_{\text{kir}}$ to have  non-empty ac-spectrum  it is necessary that  the sequence $\{t_{n+1}-t_n,b_n\}_1^{\infty}$ is eventually periodic,
i.e. for certain  $N\in\mathbb{N}$ the sequence  $\{t_{n+1}-t_n,b_n\}_N^{\infty}$
is periodic.
This complements the previous
results by Breuer and Frank \cite[Theorem 1]{BreFra09}
for the discrete Laplacian $\Delta$ on $\mathcal{G}$
as well as for  sparse trees in the metric case.

On the other hand, Kostenko and Nicolussi  \cite{KosNic19}  have  found
sufficient conditions for the Kirchhoff Laplacian ${\bf H}_{\text{kir}}$  on infinite radially symmetric antitrees to satisfy $\sigma_{ac}({\bf H}_{\text{kir}}) = [0, \infty)$. Note that in this case
the presence of positive sc-spectrum is not excluded.

These results demonstrate the duality between
Hamiltonians  ${\bf H}_{\alpha}$ with $\alpha(v) = \alpha_0\text{deg}(v)$
(with  some $\alpha_0\in\mathbb{R}$) and discrete Laplacians  $\Delta$
on equilateral metric graphs $\mathcal{G}$
discovered by K. Pankrashkin  \cite{Pan2012}.
Namely, using the Weyl function technique he   established  unitary equivalence of the    Hamiltonian
${\bf H}_{\alpha}\text{E}_{{\bf H}_{\alpha}}(J)$
being the restriction of ${\bf H}_{\alpha}$
  to any gap  $J$ of the Dirichlet realization on the one hand
and certain function $\eta_{\alpha}^{-1}(\Delta)$ of the weighted
adjacency discrete operator $\Delta$,
on the other hand.
In particular, ${\bf H}_{\alpha}\text{E}_{{\bf H}_{\alpha}}(J)$ and
$\eta_{\alpha}^{-1}(\Delta)$ are absolutely continuous only simultaneously.

In this connection we mention that  the author of \cite{Pan2012}  generalized
and extended the result of the  proceeding paper  \cite{AlbBraMalNei05}
where  a symmetric operator $A$ in $\frak H$ with a gap $(a,b)$ was considered
and spectra within a gap $(a,b)$ of selfadjoint extensions $A_B\in \Ext_A$ were investigated
under the assumption on the  Weyl function to be of scalar type:
$M(\lambda) = m(\lambda)\cdot I_{\cH}$ which is more restrictive than the assumption
on the  Weyl function $M$ in \cite{Pan2012}.
Starting with an operator $T= T^*\in \cB(\cH)$ satisfying $\sigma(T)\subset (a,b)$
and setting $B = m(T)$  it was shown  in \cite{AlbBraMalNei05}  that $A_BE_{A_B}(a,b)$ and $B = m(T)$
are unitarily equivalent.
\end{remark}

\subsection{Hamiltonian with delta-interactions}\label{sect.PosPSp}

Let $\alpha: \mathcal{V}\rightarrow\mathbb{C}^{m\times m}, \quad\alpha(\cdot)=\alpha(\cdot)^*$, be
given.
In this section we consider  quantum graph with
$\delta$-type vertex condition (see  Subsection \ref{frame}):
\begin{equation}\label{delt}
\begin{cases}
f \quad\text{is continuous at}\quad v,\\
\sum_{e\in E_v}f'_e(v)=\alpha(v)f(v),
\end{cases} \quad v\in \mathcal{V}.
\end{equation}
We underline that derivatives in \eqref{delt} are $\text{\bf outgoing}$.

Let also $Q(\cdot)=\bigoplus_{j=1}^p Q_j(\cdot)=Q(\cdot)^*\in
L^1(\mathcal{G};\mathbb{C}^{m\times m})$. The Hamiltonian with delta-interactions ${\bf
H}_{\alpha}:={\bf H}_{\alpha,Q}$ is defined as follows:
  \begin{equation}\label{deltop}
\begin{gathered}
{\bf H}_{\alpha}:={\bf H}_{\alpha,Q}=A_{\max}\upharpoonright \dom({\bf H}_{\alpha}), \quad A_{\max}=A^*, \\
\dom({\bf H}_{\alpha})=\{f\in\dom(A^*): f \quad\text{satisfies}\   \eqref{delt}, \quad
v\in \mathcal{V}\}.
\end{gathered}
\end{equation}

Since $\dom(A_{\text{max}})=\bigoplus_{e\in \mathcal{E}}\dom(A_{e,\text{max}})$,
it follows from \eqref{dom_H_max} that $\dom(A_{\text{max}})\subset W^{2,1}(\mathcal{G};\mathbb{C}^m)$.
Note that in the case of $Q\in L^2(\mathcal{G};\mathbb{C}^{m\times m})$ we have $\dom(A_{\text{max}})=W^{2,2}(\mathcal{G};\mathbb{C}^m)$.

Since $\alpha=\alpha^*$, the Hamiltonian ${\bf H}_\alpha$ is the symmetric extension of
$A =A_{\min}.$ Taking into account \eqref{npmA} and that the number of
conditions in \eqref{delt} is $(p_1+2p_2)m$, we conclude that the Hamiltonian ${\bf
H}_\alpha$ is self-adjoint, i.e. ${\bf H}_{\alpha}={\bf H}_{\alpha}^*.$
%%In particular,
For  $\alpha=0$ condition \eqref{delt} is called the Kirchhoff vertex condition, and
the corresponding operator ${\bf H}_{0}= {\bf H}_{\text{kir}}$ -- the Kirchhoff
realization.

As a consequence of Theorem \ref{sc_spec} one gets the following result.
\begin{corollary}\label{corabscsp}
Assume the conditions of Theorem \ref{sc_spec}. Then:
$$
\sigma_{\ac}({\bf H}_{\alpha})= \mathbb R_+ \quad \text{and} \quad  \sigma_{sc}({\bf H}_{\alpha})=\emptyset.
$$
\end{corollary}

However these relations  do not exclude positive embedded eigenvalues for  ${\bf H}_{\alpha}$.
In this connection we mention the paper  by G. Berkolaiko and W. Liu \cite{BerLiu17}  devoted to investigation of genericity conditions of simplicity of spectrum of quantum graph. Their  main result  reads as follows.
    \begin{theorem}[{\cite[Theorem 3.6]{BerLiu17}}]\label{th36BL}
Let $\mathcal{G}$ be a connected graph with zero potential $Q=0$ and $\delta$-type conditions at vertices. If $\mathcal{G}$ is not equivalent to a circle,
then, after a small modification of edge lengths, the new graph $\widetilde{\mathcal{G}}$ will satisfy the following genericity conditions:

\item[\;\;\rm (i)] The spectrum  $\sigma(\widetilde{\mathcal{G}})$ is simple;

\item[\;\;\rm (ii)] For each eigenfunction $f$ of $\widetilde{\mathcal{G}}$,
either $f(v)\neq 0$ for each vertex $v,$ or $\supp f=L$ for only one loop $L$ of $\widetilde{\mathcal{G}}$.
More precisely, in the space of all possible edge length, the set on which the above conditions are satisfied is residual (comeagre).
\end{theorem}

This theorem implies  that in  the scalar case ($m=1$) the positive  point  spectrum of ${\bf H}_{\alpha}$ for graphs with $p_1>0$ leads   and $Q=0$
is generically absent, i.e. the positive  part ${\bf H}_{\alpha}{\text E}_{{\bf H}_{\alpha}}(0,\infty)$ of ${\bf H}_{\alpha}$  is generically purely absolutely continuous.
Apparently  Theorem \ref{th36BL} as well as its just mentioned consequence on embedded positive eigenvalues remains  valid also in the matrix case for summable potentials $Q\in L^1(\mathcal{G};\mathbb{C}^{m\times m}).$  It will be discussed elsewhere.

Note also that   L. Friedlander \cite{Fried2005} showed that the set $\mathcal{M}$ in the parameter space $\mathbb{R}_{+}^{|\mathcal{E}|}$ of metrics, for which all eigenvalues of connected metric graph $\mathcal{G}$ are simple, is residual under certain simple conditions.

 Y. Colin de Verdi\'ere and F. Truc \cite{Colin15}, \cite{ColinTruc18} give an explicit classification of graphs that can possibly  support  embedded eigenvalues.
Also the set of non-normalized semi-classical measures on quantum graphs is described, and
the set of resonances $\text{Res}_{\mathcal{G}}^{\overrightarrow{l}}$ of the Laplacian
$\Delta_{\mathcal{G}}^{\overrightarrow{l}}$ is studied.

\subsection{Hamiltonians with $\delta$-interactions on the line}

Here we apply the previous results to  operators  with $\delta$-interactions on the line (half-line).
To this end let us consider the following formal Schr\"odinger differential expression:
   \begin{equation}\label{helfrGMNP}
\mathcal{A}_{X,\alpha,Q}:=-\frac{d^2}{dx^2}+Q+\sum_{n=1}^N \alpha_n\delta(x-x_n).
\end{equation}
This operator describes $\delta$-interactions on a finite set $X:=\{x_n\}_{n=0}^N (\subset \mathbb{R})$,
with matrix (selfadjoint) coefficients    $\{\alpha_n\}_{n=1}^N \subset \mathbb{C}^{m\times m}$, $N\in \Bbb N$.
In the scalar case $\{\alpha_n\}$ describes the \emph{strength} of the interaction at the point $x=x_n$
and  is called the \emph{intensity}  (see \cite{AlbGesHoeHol05}).

Consider  the maximal operator $A_{\text{max}}$ associated in $L^2(\mathbb{R};\mathbb{C}^m)$
with expression \eqref{helfrGMNP}. It is given by the  expression $-\frac{d^2}{dx^2}+Q$  on the domain
\begin{equation}\label{hellfrGMNP-1}
\dom(A_{\text{max}})
=\left\{f\in L^2(\mathbb{R}\setminus X; \mathbb{C}^m):
\begin{matrix}
f, f'\in AC_{\text{loc}}(\mathbb{R}\setminus X; \mathbb{C}^m),  \\
\mathcal{A}_{X,\alpha,Q}(f)\in L^2(\mathbb{R};\mathbb{C}^m), \\
\end{matrix}
\right\}.
\end{equation}
The domain of the minimal operator $A:=A_{\text{min}}$ is given by
\begin{equation}\label{hGMNPdAm}
\dom(A)
=\left\{f\in \dom(A_{\text{max}}):
\begin{matrix}
f(x_0)=0, \quad f(x_n\pm)=0, \quad n\in\{1,\ldots, N\}, \\
f'(x_0)=0, \quad f'(x_n\pm)=0, \quad n\in\{1,\ldots, N\}
\end{matrix}
\right\}.
\end{equation}
It is known that $A^* =A_{\text{max}}$.
We also consider the Hamiltonian ${\bf H}_{X,\alpha,Q}$ associated with the  expression
\eqref{helfrGMNP}  in  $L^2(\mathbb{R};\mathbb{C}^{m})$ in the following way:
  \begin{equation}\label{hellfrGMNP-2}
\begin{gathered}
{\bf H}_{X,\alpha,Q}=A^*\upharpoonright \dom({\bf H}_{X,\alpha,Q}), \quad \text{where} \\
\dom({\bf H}_{X,\alpha,Q})
=\left\{f\in \dom(A^*):
\begin{matrix}
&f(x_n+)=f(x_n-), \quad 1\leq n\leq N, \\
&f'(x_n+)-f'(x_n-)=\alpha_nf(x_n), \\
\end{matrix}
\right\}.
\end{gathered}
\end{equation}
The operator  ${\bf H}_{X,\alpha,Q} = {\bf H}_{X,\alpha,Q}^*$    has  several important features
missing for general  realizations of $\mathcal{A}_{X,\alpha,Q}$.
If $\alpha_n=0$ for all $n\in\{1,\ldots, N\}$, then ${\bf H}_{X,0,Q}$ coincides with the
free  realization
of the expression $\mathcal{A}_{X,0,Q}$ on the line $\mathbb{R}$.
   \begin{proposition}\label{newPropforR}
Let $A$ be a minimal symmetric operator in $L^2(\mathbb{R};\mathbb{C}^m)$   given by \eqref{hGMNPdAm}
and let $Q(\cdot)=Q(\cdot)^*\in L^1(\mathbb{R};\mathbb{C}^{m\times m})$.
Let also $\widetilde{A}$ be an arbitrary self-adjoint extension of  $A$.
Then:

\item[\;\;\rm (i)] The singular continuous spectrum $\sigma_{sc}(\widetilde{A})$ is empty, i.e. $\sigma_{sc}(\widetilde{A})  =\emptyset$;

\item[\;\;\rm (ii)] The absolutely continuous spectrum $\sigma_{ac}(\widetilde{A})$ fills in the half-line
$\mathbb{R}_+$, \\ $\sigma_{ac}( \widetilde{A}) =[0,\infty),$  and is of the constant multiplicity $2m$;
\item[\;\;\rm (iii)] Operator $\widetilde{A}$ is semi-bounded from below and its negative spectrum is either finite
or forms countable sequence tending to zero;
\item[\;\;\rm (iv)] Possible positive eigenvalues embedded in the $ac$-spectrum $\sigma_{ac}(\widetilde{A})$
form a discrete set, i.e. the set  $\sigma_{pp}(\widetilde{A})\cap\mathbb{R_+}$ is discrete;

\item[\;\;\rm (v)] The spectrum  $\sigma({\bf H}_{X,\alpha,Q})\cap\mathbb{R}_+ = \mathbb{R}_+$ is purely absolutely continuous, i.e. the Hamiltonian ${\bf H}_{X,\alpha,Q}$ has no positive embedded eigenvalues.
\end{proposition}
\begin{proof}
(i)-(iv).  The line $\mathbb{R}$ together with points $X$ forms the graph $\mathcal{G}=(\mathcal{V},\mathcal{E})$, where $\mathcal{V}:=X$, and the set $\mathcal{E}$
consists of $N$ finite edges $(x_0,x_1)$, $\ldots$, $(x_{N-1}, x_N)$ and two leads: $(-\infty,x_0)$, $(x_N,\infty)$.
Since $Q(\cdot)$ is summable, the respective quantum graph
meets the conditions of Theorem \ref{sc_spec},
hence the required conclusions hold.

(v). Assuming the contrary we find $\lambda_0> 0$ and $f\in L^2(\mathcal{G};\mathbb{C}^m)$ satisfying
${\bf H}_{X,\alpha,Q}f= \lambda_0 f$. By Corollary \ref{corof 3.4},  $f=0$ on the leads $(-\infty,x_0)$ and $(x_N,\infty)$.
Therefore $f(x_0)=f'(x_0)=f(x_N)=f'(x_N)=0.$ Applying the Cauchy uniqueness theorem step by step
we obtain that $f=0$ on all  edges
$(x_0,x_1)$, $\ldots$, $(x_{N-1}, x_N)$, i.e. $f=0$ on $\mathcal{G}.$ Thus, $\lambda_0\notin\sigma_p({\bf H}_{X,\alpha,Q}).$
    \end{proof}

\begin{remark}

\item[\;\;\rm(i)] Proposition \ref{newPropforR} is also valid for realizations of $\mathcal{A}_{X,\alpha,Q}$ in $L^2(\mathbb{R}_+;\mathbb{C}^{m})$  %%on $\Bbb R_+$
including the Hamiltonian ${\bf H}_{X,\alpha,Q}$  (cf. \cite[Theorem 4.6]{GMNP17});

\item[\;\;\rm(ii)] In the case $\alpha_1=\ldots=\alpha_N=0$ each realization in $L^2(\mathbb{R}_+;\mathbb{C}^{m})$
of Sturm-Liouville expression
with a potential matrix $Q\in L^1(\mathbb{R}_+;\mathbb{C}^{m\times m})$ has the purely
absolutely continuous spectrum of the  multiplicity $m$ (see {\cite{GMNP17}}).
In turn, this result generalizes the classical Titchmarsh's result for the scalar
operator associated in $L^2(\mathbb{R}_+)$ with the expression
$\mathcal{A}:=-\frac{d^2}{dx^2}+q$, where $q=\overline{q}\in L^1(\mathbb{R}_+)$ (see
Chapter 5 in \cite{Tit1946}).
\end{remark}

\section{Bargmann type estimates}\label{sect.BargEst}

Here we evaluate the number of negative squares  of the operator ${\bf H}_{\alpha, Q}$
given  by \eqref{deltop}  assuming in addition that $xQ \in L^1(\mathcal{G};
\mathbb{C}^{m\times m})$.  To this end we recall that $W^{1,2}(\mathcal{G}\setminus
\mathcal{V};\mathbb{C}^m) = \oplus_{e\in \mathcal{E}}W^{1,2}(e;\mathbb{C}^m)$.

     \subsection{Quadratic form}
Alongside  the operator ${\bf H}_{\alpha, Q}$ we consider the form
  \begin{equation}\label{eq:main_quadr_form}
\mathfrak{t}_{\alpha, Q}[f] = \int_{\mathcal{G}}|f'(x)|^2\,dx +
\int_{\mathcal{G}}\left<Qf,f\right>\,dx
+ \sum_{v\in \mathcal{V}}\left<\alpha(v)f(v),f(v)\right>,
  \end{equation}
on the domain
\begin{equation}\label{eq:dom_of_main_form}
\begin{split}
& \dom\mathfrak{t}_{\alpha, Q} = W^{1,2}(\mathcal{G};\mathbb{C}^m) \\
& :=  \left\{f\in W^{1,2}(\mathcal{G}\setminus \mathcal{V};\mathbb{C}^m): f\quad \text{is continuous at each}\quad v\in \mathcal{V}  \right\}.
\end{split}
\end{equation}
We show that this form is well defined and is the closure of the form
${\mathfrak{t}'}_{\alpha, Q}$,  ${\mathfrak{t}'}_{\alpha, Q}[f] := ({\bf H}_{\alpha,
Q}f,f)$, $\dom {\mathfrak{t}'}_{\alpha, Q} = \dom {\bf H}_{\alpha, Q}$, naturally
generated by ${\bf H}_{\alpha, Q}$.
\begin{proposition}\label{lem_form_t_H_alpha,Q}
Let  ${\bf H}_{\alpha, Q} = {\bf H}_{\alpha, Q}^*$  be the Hamiltonian  given by
 \eqref{deltop}  and let $Q \in L^1(\mathcal{G}; \mathbb{C}^{m\times m})$. Let also $\mathfrak{t}_{\alpha, Q}$
be  the form given by \eqref{eq:main_quadr_form}--\eqref{eq:dom_of_main_form}.
 Then:

\item[\;\;\rm(i)] The form $\mathfrak{t}_{\alpha, Q}$  is semibounded below  and closed;

\item[\;\;\rm(ii)]  $\dom{\bf H}_{\alpha, Q} \subset W^{1,2}(\mathcal{G};\mathbb{C}^m)$
and
  \begin{equation}\label{eq:frak_t'=frak_t_on_dom_H}
{\mathfrak{t}'}_{\alpha, Q}[f,g] = ({\bf H}_{\alpha, Q}f,g) = \mathfrak{t}_{\alpha,
Q}[f,g], \quad   f\in \dom{\bf H}_{\alpha, Q}, \ g\in %%\dom \mathfrak{t}_{\alpha, Q} =
W^{1,2}(\mathcal{G};\mathbb{C}^m).
    \end{equation}

\item[\;\;\rm(iii)] The operator  associated with the form \eqref{eq:main_quadr_form} in accordance
with the first representation theorem  coincides with the Hamiltonian  ${\bf H}_{\alpha,
Q} = {\bf H}_{\alpha, Q}^*$;

\item[\;\;\rm(iv)] The domain $\dom{\bf H}_{\alpha, Q}$ is the core for $\mathfrak{t}_{\alpha, Q}$,
i.e. the  closure of ${\mathfrak{t}'}_{\alpha, Q}$ is ${\mathfrak{t}}_{\alpha, Q}$,
$\overline{\mathfrak{t}'}_{\alpha, Q} = \mathfrak{t}_{\alpha, Q}$. In particular, for
any $\alpha: \mathcal{V}\rightarrow\mathbb{C}^{m\times m}$ and $Q \in L^1(\mathcal{G}; \mathbb{C}^{m\times m})$   the domain $\dom{\bf
H}_{\alpha, Q}$ is dense in $W^{1,2}(\mathcal{G};\mathbb{C}^m)$.
  \end{proposition}
  \begin{proof}
(i).  First  we show that the form  \eqref{eq:main_quadr_form}  is semibounded below and
closed. To this end we note that   in accordance with the Sobolev embedding theorem
$W^{1,2}(e;\Bbb C^m) \hookrightarrow C(e;\Bbb C^m)$ and for any $\varepsilon>0$ there
exists $C_1:=C_1(\varepsilon)>0$ such that on each edge $e\in \mathcal{E}$ (finite or infinite)
the following estimate holds
  \begin{equation}\label{embed_theorem}
\|f_e\|^2_{C(e;\mathbb{C}^m)}\le \varepsilon\|f'_e\|^2_{L^2(e;\mathbb{C}^m)} + \varepsilon^{-1}C_1
\|f_e\|^2_{L^2(e;\mathbb{C}^m)}, \quad f_e\in W^{1,2}(e;\Bbb C^m).
%\  j\le p.
  \end{equation}
Here  $C(e;\mathbb{C}^m)$ denotes the Banach  space of continuous vector-functions on the edge $e$.
Since $Q_e\in L^{1}(e;\Bbb C^{m\times m})$, it follows from \eqref{embed_theorem} that
for any $f\in W^{1,2}(\mathcal{G}\setminus \mathcal{V}; \Bbb C^{m})$
   \begin{eqnarray}\label{estimate_for_Q}
 \left|\int_{e}  \left<Q_e(x)f_e(x),f_e(x)\right>\,dx\right|
 \le\int_{e}|Q_e(x)|\cdot|f_e(x)|^2\,dx  \nonumber  \\
\le  \|Q_e\|_{L^1(e)}\cdot\|f_e\|^2_{C(e)} \le
\|Q_e\|_{L^1(e)}\left(\varepsilon\|f'_e\|^2_{L^2(e)} + \varepsilon^{-1}C_1
\|f_e\|^2_{L^2(e)}\right).
  \end{eqnarray}
Here $|Q_e(x)| := |Q_e(x)|_{\Bbb C^{m\times m}}$ denotes the matrix norm, and it is  put for brevity $C(e):=C(e;\mathbb{C}^m)$,
$\|Q_e\|_{L^1(e)} := \|Q_e\|_{L^1(e; \Bbb C^{m\times m})}$.

Similarly, one gets from  \eqref{embed_theorem}   that for any $f\in
W^{1,2}(\mathcal{G}; \Bbb C^{m})$
   \begin{eqnarray}\label{estimate_for_alpha}
\left|\sum_{v\in \mathcal{V}}\left<\alpha(v)f(v),f(v)\right>\right| \le \sum_{v\in \mathcal{V}}
|\alpha(v)|\cdot|f(v)|^2 \nonumber \\
\le C_2({\alpha}) \sum_{e\in \mathcal{E}}\left(\varepsilon\|f'_e\|^2_{L^2(e)} + \varepsilon^{-1}C_1 \|f_e\|^2_{L^2(e)}\right),
  \end{eqnarray}
where $C_2({\alpha}) := \max \{|\alpha(v)|: v\in \mathcal{V}\}$.

To proceed further we treat  the form $\mathfrak{t}_{\alpha,Q}$   as a
perturbation of the simplest (unperturbed) form $\mathfrak{t}_{0,0}$ given by
$$
 \mathfrak{t}_{0,0}[f] = \int_{\mathcal{G}}\left|f'(x)\right|^2\,dx = \sum_{e\in \mathcal{E}} \int_{e}\left|f'_e(x)\right|^2\,dx, \qquad
\dom\mathfrak{t}_{0,0} = W^{1,2}(\mathcal{G};\mathbb{C}^{m}).
$$
With this aim  we consider form perturbations  $\frak t_q$ and  $\frak t_{\alpha}$ given
by the second and third summands in \eqref{eq:main_quadr_form}  on the domain
$W^{1,2}(\mathcal{G};\mathbb{C}^{m})$. Note that  although the form $\frak
t_q$ is naturally defined on a wider domain $W^{1,2}(\mathcal{G}\setminus \mathcal{V};\mathbb{C}^{m})$, the
definition of  $\frak t_{\alpha}$ requires the continuity of functions $f$ at each
vertex $v\in \mathcal{V}$, hence cannot be extended on a wider subset of
$W^{1,2}(\mathcal{G}\setminus \mathcal{V};\mathbb{C}^{m})$.

Choosing    $\varepsilon :=
2^{-1}\left(\|Q\|_{L^1(\mathcal{G})} + C_2({\alpha})\right)^{-1}$ and  combining estimates
\eqref{estimate_for_Q}  and  \eqref{estimate_for_alpha}  with this $\varepsilon$ we arrive at the following
estimate  for each
$f\in W^{1,2}(\mathcal{G}; \mathbb{C}^{m})$
    \begin{align}\label{eq:KLMN_estimate}
|\frak t_q[f] + \frak t_{\alpha}[f]| &= \left|\sum_{e\in \mathcal{E}}\int_{e}
\left<Q_e(x)f_e(x),f_e(x)\right>\,dx  +  \sum_{v\in \mathcal{V}}\left<\alpha(v)f(v),f(v)\right>\right| \nonumber   \\
\le & \|Q\|_{L^1(\mathcal{G})} \left(\sum_{e\in \mathcal{E}}
\varepsilon\|f'_e\|^2_{L^2(e;\mathbb{C}^m)} + \varepsilon^{-1}C_1 \|f_e\|^2_{L^2(e;\mathbb{C}^m)}\right) \nonumber  \\
   & + C_2({\alpha}) \sum_{e\in
\mathcal{E}}\left(\varepsilon\|f'_e\|^2_{L^2(e;\mathbb{C}^m)} + \varepsilon^{-1}C_1 \|f_e\|^2_{L^2(e;\mathbb{C}^m)}\right) \nonumber \\
& =  2^{-1}\sum_{e\in \mathcal{E}} \left(\|f'_e\|^2_{L^2(e;\mathbb{C}^m)} + 4C_3 \|f_e\|^2_{L^2(e;\mathbb{C}^m)}\right)
    \end{align}
with $C_3 = C_1 (\|Q\|_{L^1({\mathcal{G}}, \Bbb C^{m\times m})} + C_2({\alpha}))^2$.  By
the KLMN theorem (\cite[Ch. 6]{Kato66}, \cite[Theorem X.17]{ReedSim78}  (see also
\cite[Theorem 3.38]{DerMal17}) the form ${\mathfrak{t}}_{\alpha,Q} := \mathfrak{t}_{0,0} +
\frak t_q + \frak t_{\alpha}$ is lower semibounded and closed on the domain
$\dom{\mathfrak{t}}_{\alpha,Q} = W^{1,2}(\mathcal{G}; \mathbb{C}^{m})$.

(ii).   The inclusion $\dom{\bf H}_{\alpha, Q} \subset
W^{1,2}(\mathcal{G};\mathbb{C}^m)$ is obvious. Let us prove identity
\eqref{eq:frak_t'=frak_t_on_dom_H}. Let  $l\in \mathcal{E}_{v,\infty}$  be  a lead outgoing from
the vertex $v$.
Since $Q_l\in L^1(l; \mathbb{C}^{m\times m})$,
the deficiency indices of the minimal Schr\"odinger operator  $A_{{l},\min}$  in $L^2(l;
\mathbb{C}^{m})$ are minimal, $\mathrm{n}_{\pm}(A_{{l},\min}) = m$.
It follows that the Wronskian
$\left<f_l(x),g'_l(x)\right>-\left<f'_l(x),g_l(x)\right>$ tends to zero as $x\rightarrow\infty, x\in l$.
Therefore integrating by parts along $l$ and taking the last relation into account
we get that for any $f_l\in\dom(A_{l,\max})$ and $g_l\in W^{1,2}(l;\mathbb{C}^m)$
\begin{equation}\label{ibplead}
\int_{l}\left(-\left<f''_l,g_l\right>+\left<Q_lf_l,g_l\right>\right)dx =
\int_{l}\left(f'_l \overline{g}'_l + \left<Q_lf_l,g_l\right>\right)dx +
\left<f'_l(v),g_l(v)\right>.
\end{equation}

Let further $e_k=\overrightarrow{vv_k}\in \mathcal{E}_{v,\text{fin}}$  be any finite edge and let $f:=f_{e_k}\in\dom(A_{e_k, \max})$.
Since the derivatives at each vertex are \textbf{outgoing},
we have
$f'_{\overrightarrow{vv_k}}(v_k)=-f'_{\overrightarrow{v_kv}}(v_k)$.
Therefore  integrating by
parts along the edge $e_k=\overrightarrow{vv_k}$,
we derive:
\begin{equation}\label{impth-2}
\begin{gathered}
\int_{e_k}\left(-\left<f'',g\right>+\left<Q_kf,g\right>\right)dx = \int_{e_k}\left(f'\overline{g}' + \left<Q_kf,g\right>\right)dx\\
-\left<f'_{\overrightarrow{vv_k}}(v_k), g(v_k)\right>+\left<f'_{\overrightarrow{vv_k}}(v), g(v)\right> =\\
=\int_{e_k}\left(f'\overline{g}' + \left<Q_kf,g\right>\right)dx +
\left<f'_{\overrightarrow{v_kv}}(v_k), g(v_k)\right> +
\left<f'_{\overrightarrow{vv_k}}(v), g(v)\right>.
\end{gathered}
\end{equation}

Taking the sum over all leads and edges $e\in \mathcal{E}$,
using definition \eqref{eq:main_quadr_form} of the form
$\mathfrak{t}_{{\alpha,Q}}$,  and taking  identities
\eqref{ibplead}-\eqref{impth-2} and condition \eqref{delt} into  account, we obtain
 that for all $f\in\dom({\bf H}_{\alpha,Q})$ and $g\in W^{1,2}(\mathcal{G};\mathbb{C}^m)$
  \begin{equation}\label{quadform}
\begin{gathered}
\mathfrak{t}'_{{\alpha,Q}}[f,g] = \sum_{e\in \mathcal{E}}\int_{e}\left(-\left<f''_e,g_e\right>+\left<Q_ef_e,g_e\right>
\right)dx\\
=\sum_{e\in \mathcal{E}}\int_{e}\left(f'_e(x)\overline{g'_e(x)}+\left<Q_e(x)f_e(x),g_e(x)\right>\right)dx \\
+\sum_{e\in \mathcal{E}_{\text{fin}}}\left(\left<f'_{-e}(v_{in}),g_e(v_{in})\right>+\left<f'_{e}(v_o),g_e(v_o)\right>\right)
+\sum_{e\in \mathcal{E}_{\infty}}\left<f'_e(v),g_e(v)\right>\\
=\int_{\mathcal{G}}f'(x)\overline{g'(x)}dx + \int_{\mathcal{G}}\left<Qf,g\right>dx +
\sum_{v\in \mathcal{V}}\left(\sum_{e\in \mathcal{E}_v}\left<f'_e(v),g_e(v)\right>\right)  \\
=\int_{\mathcal{G}}f'(x)\overline{g'(x)}dx + \int_{\mathcal{G}}\left<Qf,g\right>dx + \sum_{v\in \mathcal{V}}\left<\alpha(v)f(v), g(v)\right>
= \mathfrak{t}_{{\alpha,Q}}[f,g].
\end{gathered}
\end{equation}

This identity  coincides with \eqref{eq:frak_t'=frak_t_on_dom_H}  and proves
(ii).

(iii)-(iv).   Since the form  ${\mathfrak{t}}_{\alpha,Q}$  is closed (see (i)),
 the first representation theorem (\cite[Theorem VI.2.1.(ii)]{Kato66}),
 ensures that the domain  $\dom {\bf H}_{\alpha, Q}$ is the core of the form
${\mathfrak{t}}_{\alpha,Q}$ which proves (iv).

 Moreover, this theorem implies when combining (i) with identity \eqref{eq:frak_t'=frak_t_on_dom_H}
 that the operator ${\bf H}_{\alpha, Q} = {\bf H}_{\alpha, Q}^*$ is associated
with the (closed) form  ${\mathfrak{t}}_{\alpha,Q}$. This proves  the statement (iii).
   \end{proof}

  \subsection{Classical Bargmann estimate}
First slightly modifying reasonings from \cite{Malamud1992}  we present an abstract
version of the Birman-Schwinger principle.

As usual,   for any  operator $T = T^*\in \cC(\frak H)$ we denote by
$\kappa_{\mp}(T):=\dim E_T(\Bbb R_{\mp}),$
the dimension of the "negative/positive" spectral subspace.
  \begin{proposition}\label{prop_B-Schwinger_prin}
Let $A_0$ be a nonnegative selfadjoint  operator in $\frak H$, $K\in\mathcal B(\frak
H)$, and let $A_1 = A_0 - K^* K$. Let also $\cH = \overline {\ran(K)} = (\ker
K^*)^{\perp}$ and let
  \begin{equation}\label{Weyl_func_for_pair_A_0,A_1}
M(z) := -I_{\cH} + K(A_0-z)^{-1}K^*,\qquad z\in\rho(A_0),
  \end{equation}
be the operator function with values in $\cB(\cH)$ and $K^*(\in \cB(\cH, \frak H))$ is
treated as an operator from $\cH$  to $\frak H$ and $M(0-) := s-R-\lim_{\varepsilon
\downarrow 0}M(\varepsilon)$. Then
  \begin{equation}
\kappa_-(A_1)  = \kappa_+\bigl(M(0-)\bigr).
  \end{equation}
\end{proposition}
%%%%%%%%%%%%%%%%%%%%%%%%%%%%%%%%%%%%%%%%%%
\begin{proof}
Let $\{\lambda_j\}^N_1$ be the sequence of negative eigenvalues of the operator $A_1$.
Fixing  $\lambda_j\in\sigma_p(A_1)\cap\Bbb R_-$ one has
$0 = (A_1-\lambda_j I)h = (A_0-\lambda_j I)h - K^*K h$ with some  vector $h\in \cH$.
This equality is equivalent to
$h = (A_0-\lambda_j)^{-1}K^* K h,$
i.e.
$ 1\in \sigma_p\bigl((A_0-\lambda_j)^{-1}K^* K\bigr) =
\sigma_p\bigl(K(A_0-\lambda_j)^{-1}K^*\bigr). $
Moreover (see, \cite[Ch. 3.10]{BirSol1987}),
  $$
\dim\ker\bigl((A_0-\lambda_j)^{-1}K^* K - I\bigr) =
\dim\ker\bigl(K(A_0-\lambda_j)^{-1}K^* - I\bigr).
$$
In turn, this equality  is equivalent to the relation $\dim\ker(A_1 - \lambda_j I) =
\dim\ker M(\lambda_j)\not =\{0\}$. Applying \cite{DerMal91} (see also \cite[Lemma 8.69]{DerMal17}) one gets
$$
\kappa_-(A_1) = \sum_j\dim\ker(A_1 - \lambda_j I) = \sum_j\dim\ker M(\lambda_j) =
\kappa_+\bigl(M(0-)\bigr).
$$
This proves the statement.
  \end{proof}
Next we present the matrix version of the classical Bargmann estimate. It is well known
in  the scalar case and can be proved similarly. We present the proof for the sake of
completeness following  \cite{Malamud1992}
%%while it can be proved in the vector case in just the same way as it is done in
(other approaches can be found in \cite{Birman1961}, \cite{BerShub91}, and
\cite{AlbKosMalNei13}).

  Denote by  $P_{\pm}(x) = E_{Q(x)}$\   $x\in \Bbb R_+$,  the spectral projection of
the matrix $Q(\cdot) (\in C(\Bbb R_+))$ on the positive/negative  part of its spectrum
and put $Q_{\pm}(x) := \pm P_{\pm}(x)Q(x)$. Clearly, $Q_{\pm}(x)\ge 0$ and $Q(x)=
Q_{+}(x) - Q_{-}(x)$.
%% on the edge $e$.
%
   \begin{proposition}\label{prop:matrix_Barg_estim-for_S-L_oper}
Let $Q = Q^*$ be continuous and  $xQ\in L^1(\Bbb R_+; \mathbb{C}^{m\times m}),$ and let
${\bf H}_D:={\bf H}_{D,Q}$ be the Dirichlet realization of $-d^2/dx^2 + Q$. Then the following
estimate holds
  \begin{equation}\label{Barg_type_est_for_matrix_Dir_oper}
\kappa_{-}({\bf H}_D)\leq  \int_{\Bbb R_+}x \cdot {\tr}(Q_{-}(x))\,dx.
  \end{equation}
   \end{proposition}
%%%%%%%%%%%%%%%%%%%%%%%%%%%%%%%%%%%%%%%%%%%%%%%%%
\begin{proof}
Since $\kappa_{-}({\bf H}_{D,Q}) \leq \kappa_{-}({\bf H}_{D,Q_-})$, it suffices to consider the
case of negative $Q= -Q_-$. Now the Weyl function \eqref{Weyl_func_for_pair_A_0,A_1}
takes the form
\begin{equation}
(I_\cH + M(\lambda))h = \int^{\infty}_0 Q_-^{1/2}(x)G(x,t;\lambda)Q^{1/2}_-(t)h(t)\, dt
=: G(\lambda)h,\quad h\in \cH,
\end{equation}
where $\cH =L^2(\supp Q_-)$, $\supp Q_- = \overline {\{x\in \Bbb R_+: Q_-(x)\not =0\}}$,
and
  \begin{equation}
G(x,t;\lambda)=\frac{1}{\sqrt{-\lambda}}
\begin{cases}
\sinh (t\sqrt{-\lambda})\cdot \exp(-x\sqrt{-\lambda})\cdot I_m,&  t\le x,\\
\sinh (x\sqrt{-\lambda})\cdot \exp(-t\sqrt{-\lambda})\cdot I_m,&  t\ge x,
\end{cases}
  \end{equation}
is the Green function of the Dirichlet realization ${\bf H}_{D,0}$ with $Q=0$ and $\lambda <
0$.

Clearly,  $G(x,t;0) := \lim_{\varepsilon\downarrow 0}G(x,t;-\varepsilon) = t$ for $t\le
x$.
Due to the assumptions on $Q$  the integral operator $G(-\varepsilon)$ is of trace class
with the continuous kernel, hence $\text{tr} G(-\varepsilon) = \int^{\infty}_0 \sinh
(t\sqrt{\varepsilon})\cdot \exp(-t\sqrt{\varepsilon}) \cdot \text{tr} Q_-(t)\,dt$. Combining
these facts with Lemma \ref{prop_B-Schwinger_prin}  and denoting by
$\{\lambda_j(-\varepsilon)\}_1^\infty$ the sequence  of (necessary positive) eigenvalues
of the operator $G(-\varepsilon)$  one gets
  \begin{eqnarray}\label{eq:bargmann_est-te_with_varepsilon}
\dim {\text E}_{{\bf H}_D}(-\infty, -\varepsilon)  = \kappa_+\bigl(M(-\varepsilon)\bigr)  =
\sum_{\lambda_j(-\varepsilon)>1} 1 \le
\sum_{\lambda_j(-\varepsilon)>1} \lambda_j(-\varepsilon) \nonumber \\
\le \sum^{\infty}_{j=1}\lambda_j(-\varepsilon) = \tr G(-\varepsilon)
%%\tr\int^{\infty}_0 \sinh (t\sqrt{\varepsilon})\cdot \exp(-t\sqrt{\varepsilon}) Q_-(t)\,dt \nonumber \\
\le   \int^{\infty}_0 t\cdot \tr Q_-(t)\,dt.
    \end{eqnarray}
Noting that the right-hand side of \eqref{eq:bargmann_est-te_with_varepsilon} does not
depend on $\varepsilon$  and passing here to the limit as $\varepsilon \downarrow 0$ we
arrive at \eqref{Barg_type_est_for_matrix_Dir_oper}.
 \end{proof}

\subsection{Bargmann type estimate for graphs}
Now we are ready to state and prove the Bargmann type estimate for the Hamiltonian ${\bf
H}_{\alpha,Q}$,  the main result of the section.

For a symmetric quadratic form $\mathfrak{t}$ one  puts:
\begin{equation*}
\kappa_-(\mathfrak{t})=\left\{\max\dim L_-:L_-\subset\dom(\mathfrak{t}), \quad
\mathfrak{t}[f]<0 \quad \text{for all} \quad f\in L_-\setminus \{0\} \right\}.
\end{equation*}
If the form $\mathfrak{t}$ is closed, and  $T =T^*$ is the operator associated with it,
then
\\$\kappa_-(\mathfrak{t})=\kappa_-(T)$ due to the mini-max principle.
   \begin{theorem}\label{Bargkappa}
Let $Q=\bigoplus_{j=1}^p Q_j=Q^*$ and  $x_eQ_e\in L^1(e; \mathbb{C}^{m\times m}),$ \
$e\in \mathcal{E}$.
Then  for any $\alpha: \mathcal{V}\rightarrow
\mathbb{C}^{m\times m}$, $\alpha(\cdot)=\alpha(\cdot)^*$,  the operator ${\bf
H}_{\alpha,Q}$ admits the following estimate
  \begin{equation}\label{Main_Barg_type_est_for_H_alpha,Q}
\kappa_{-}({\bf H}_{\alpha,Q})\leq
\sum_{e\in \mathcal{E}} \left[\int_{e}x_e \cdot {\tr}(Q_{e,-}(x))\,dx\right] + m|\mathcal{V}|.
  \end{equation}
where $[a]$ denotes the integer part of a number $a\in \Bbb R$.
\end{theorem}
\begin{proof}
Alongside the Hamiltonian  ${\bf H}_{\alpha,Q}$ we consider the Dirichlet operator ${\bf
H}_{D} = {\bf H}_{D,Q} = \bigoplus_{j=1}^p {\bf H}_{j,D}$ where ${\bf H}_{j,D}$ is the
Dirichlet realization of the differential expression $\mathcal{A}_{j}$,
$j\in\{1,\ldots,p\}$ (see the step $\text{(i)}_2$ of the proof in  Theorem
\ref{sc_spec}). Letting $W_0^{1,2}(\mathcal{G}; \Bbb C^{m}) = \bigoplus_{e_j\in \mathcal
E} W_0^{1,2}(e_j; \Bbb C^{m})$, i.e.
$$
W_0^{1,2}(\mathcal{G}; \Bbb C^{m})=\{f\in W^{1,2}(\mathcal{G};\mathbb{C}^m): f(v)=0 \quad \text{for each} \quad v\in\mathcal{V}\},
$$
we consider the following form
\begin{equation}\label{Dirichlet_form}
\begin{gathered}
\mathfrak{t}_D[f] =  \mathfrak{t}_{D, Q}[f] = \sum_{j=1}^p\left(\int_{e_{j}}\left|f'_{j}(x)\right|^2dx +
\int_{e_j}\left<Q_jf_j,f_j\right>dx\right)\\ =
\int_{\mathcal{G}}\left(\left|f'(x)\right|^2+\left<Qf,f\right>\right)dx, \qquad
\dom\mathfrak{t}_{D} =  W_0^{1,2}(\mathcal{G}; \Bbb C^{m}).
\end{gathered}
\end{equation}
It is well known (\cite{Kre47}, see also  \cite[Ch. 3]{DerMal17}) that the form $\mathfrak{t}_{D}$ is closed and the
operator associated with $\frak t_D$ is generated by  the Dirichlet operator ${\bf
H}_{D}$.

In accordance with Proposition \ref{prop:matrix_Barg_estim-for_S-L_oper} for each
operator ${\bf H}_{j,D}$ one has
%%the following Bargmann estimate holds
%
%
  \begin{equation}\label{Barg_est_for_one_edge}
\kappa_-(\frak t_{j,D}) = \kappa_{-}({\bf H}_{j,D})\leq  \left[\int_{e_j}x \cdot \tr
Q_{j,-}(x)\, dx\right], \quad j\in \{1,\ldots, p\}.
 \end{equation}
%
%
%%This estimate is well known in  the scalar case while it can be proved in the vector case
%%in just the same way as it is done in \cite{Birman1961}, \cite{Malamud1992} and \cite{AlbKosMalNei13}.
Taking the sum over all edges
we arrive at the following estimate for the Dirichlet operator ${\bf H}_{D}$ on the
graph  $\mathcal{G}$
  \begin{equation}\label{Barg_est}
\kappa_-(\frak t_D) = \kappa_{-}({\bf H}_{D})\leq\sum_{j=1}^p\left[\int_{e_j}x \cdot \tr Q_{j,-}(x)\, dx\right].
 \end{equation}
Next we consider the form   $\mathfrak{t}_{\alpha, Q}$ given by
\eqref{eq:main_quadr_form}--\eqref{eq:dom_of_main_form}.   It follows with account of
\eqref{Dirichlet_form} that the restriction of the form $\mathfrak{t}_{\alpha, Q}$ to
the domain $\dom\frak t_D = W_0^{1,2}(\mathcal{G}; \Bbb C^{m})$ coincides with
the Dirichlet form $\frak t_D$, i.e. $\mathfrak{t}_{\alpha, Q}\upharpoonright \dom\frak
t_D =  \frak t_D$. Therefore
  \begin{equation}\label{dim_form_t/form_D}
\dim(\dom\frak t_{\alpha, Q}/\dom\frak t_D) = m|\mathcal{V}|.
  \end{equation}
Combining this relation with estimate \eqref{Barg_est} and applying the mini-max principle one gets
  \begin{equation}\label{Barg_type_est_for_H_alpha,Q}
\kappa_-(\frak t_{\alpha, Q}) \leq \kappa_-(\frak t_D) + m|\mathcal{V}| \leq  \sum_{j=1}^p\left[\int_{e_j}x \cdot \tr Q_{j,-}(x)\, dx\right] + m|\mathcal{V}|.
 \end{equation}

On the other hand,  by Proposition  \ref{lem_form_t_H_alpha,Q}, the operator associated
with the form $\frak t_{\alpha, Q}$ is  the Hamiltonian  ${\bf H}_{\alpha, Q}$.
Therefore $\kappa_-({\bf H}_{\alpha, Q}) = \kappa_-(\frak t_{\alpha, Q})$.
  Combining this equality with  \eqref{Barg_type_est_for_H_alpha,Q}  we arrive at
  required estimate \eqref{Main_Barg_type_est_for_H_alpha,Q}.
\end{proof}
Note that estimate \eqref{Main_Barg_type_est_for_H_alpha,Q} does not distinguished between  realizations
${\bf H}_{\alpha,Q}$  and   the Kirchhoff  realization ${\bf H}_{0,Q}$  of $\mathcal A$.
  \begin{corollary}
Let  ${\bf H}_{\emph{kir}} =  {\bf H}_{0,Q}$ be  the Kirchhoff  realization  of $\mathcal A$.
Then
    \begin{equation}\label{eq:barg_estim_if_AN>0}
\kappa_-({\bf H}_{\alpha, Q}) \le \kappa_-({\bf H}_{\emph{kir}}) \ + \  \sum_{v\in
\mathcal{V}}\kappa_-(\alpha(v)).
    \end{equation}
In particular, if the  Kirchhoff  realization ${\bf H}_{\emph{kir}}$ of $\mathcal A$ is
non-negative, ${\bf H}_{\emph{kir}} \ge 0$,  then  $\kappa_-({\bf H}_{\alpha, Q}) \le
\sum_{v\in \mathcal{V}}\kappa_-(\alpha(v)).$
\end{corollary}
\begin{proof}
Setting $\frak t_{\text{kir}} := \frak t_{{\bf H}_{\text{kir}}}$
one  rewrites the form  $\frak t_{\alpha, Q}$ (see \eqref{eq:main_quadr_form}--\eqref{eq:dom_of_main_form})   as
     \begin{equation}\label{eq:main_quadr_form_with_H_kir}
\frak t_{\alpha, Q}[f] = \frak t_{\text{kir}}[f]  \  + \ \sum_{v\in
\mathcal{V}}\left<\alpha(v)f(v),f(v)\right>, \quad  f\in \dom\mathfrak{t}_{\alpha, Q} =
\dom\mathfrak{t}_{\text{kir}}.
    \end{equation}
Note  that   Proposition  \ref{lem_form_t_H_alpha,Q}  implies
$\kappa_-(\frak t_{\text{kir}}) = \kappa_-({\bf H}_{\text{kir}})$  and
$\kappa_-({\bf H}_{\alpha, Q}) = \kappa_-(\frak t_{\alpha, Q})$. Combining these relations
with  \eqref{eq:main_quadr_form_with_H_kir} yields
  $$
\kappa_-({\bf H}_{\alpha, Q}) = \kappa_-(\frak t_{\alpha, Q}) \le \kappa_-(\frak
t_{\text{kir}}) + \sum_{v\in \mathcal{V}}\kappa_-(\alpha(v)) = \kappa_-({\bf H}_{\text{kir}})  + \sum_{v\in
\mathcal{V}}\kappa_-(\alpha(v)).
  $$
This proves the required estimate.
       \end{proof}
Alongside the Hamiltonian  ${\bf H}_{\alpha,Q}$  consider the Neumann  operator ${\bf
H}_{N} = {\bf H}_{N,Q} = \bigoplus_{j=1}^p {\bf H}_{j,N}$ where ${\bf H}_{j,N}$ is the
Neumann  realization of the differential expression $\mathcal{A}_{j}$,
$j\in\{1,\ldots,p\}$. It is well known that the quadratic form $\mathfrak{t}_{j,N} :=
\mathfrak{t}_{{\bf H}_{j,N}}$ corresponding to ${\bf H}_{j,N}$ is given by
$$
\mathfrak{t}_{j,N}[f_j] =  %%\mathfrak{t}_{{\bf H}_{j,N}} =
\int_{e_j}\left(\left|f'_j(x)\right|^2 + \left<Q_j(x)f_j(x),f_j(x)\right>\right)\,dx, \
f_j\in  \dom\mathfrak{t}_{j,N} =  W^{1,2}(e_j; \Bbb C^{m}).
$$
Therefore the quadratic form $\mathfrak{t}_N :=  \mathfrak{t}_{{\bf H}_{N, Q}}$
corresponding to ${\bf H}_{N}$  is defined by
\begin{equation}\label{eq:Neumann_form}
\begin{gathered}
\mathfrak{t}_N[f] =  \mathfrak{t}_{{\bf H}_{N, Q}}[f] =
\int_{\mathcal{G}}\left(\left|f'(x)\right|^2+\left<Qf,f\right>\right)dx, \quad
\dom\mathfrak{t}_{N} =  W^{1,2}(\mathcal{G}\setminus \mathcal{V}; \Bbb C^{m}).
\end{gathered}
\end{equation}
\begin{corollary}
Let  ${\bf H}_{N} = \widehat A_N$ be  the Neumann realization  of $\mathcal A$.  Then
    \begin{equation}\label{eq:barg_estim_if_AN>0-1}
{\bf H}_{N} \le {\bf H}_{\emph{kir}} \le {\bf H}_{D}, \quad i.e. \quad  \frak t_{N} \le \frak
t_{\emph{kir}} \le \frak t_{D}.
    \end{equation}
Moreover,  $\frak t_{N}[f] = \frak t_{\emph{kir}}[f] = \frak t_{D}[f]$ \  for \  $f\in
\dom(\frak t_{D}) = W^{1,2}_0(\mathcal{G}; \Bbb C^{m}).$
  \end{corollary}
  \begin{remark}
Note that the assumption   ${\bf H}_{\text{kir}} \ge 0$ is essential. Namely, simple examples
(see Remarks \ref{rem-10}, \ref{remfr5} and \ref{rem-l}) show that the estimate  \eqref{eq:barg_estim_if_AN>0} becomes false
whenever the condition ${\bf H}_{\text{kir}} \ge 0$ is replaced  by  ${\bf H}_D = \widehat
A_F \ge 0$.
   \end{remark}

\section{Non-compact star graphs with finitely many edges} \label{sect.StarGr}

\subsection{Boundary triplets and Weyl functions}

Now let us consider the star graph $\mathcal{G}$ consisting of one common vertex $v_0:=0$, $p_1>0$ leads
and $p_2\geq 0$ edges,
$p := p_1+p_2$. As earlier we equip the graph $\mathcal{G}$ with the minimal Sturm-Liouville operator $A$ defined by \eqref{symop}.

The boundary triplet for $A^*$ is chosen in accordance with \eqref{eq: 410}, \eqref{eq: 400} and \eqref{triples}:
 \begin{equation*}\label{eq: triples-5}
\begin{gathered}
\mathcal{H}_{j}=\mathbb{C}^m, \quad \Gamma_0^{(j)} f_{e_j}=f_{e_j}(0), \quad \Gamma_1^{(j)} f_{e_j}=f'_{e_j}(0), \quad j\leq p_1, \\
\mathcal{H}_j=\mathbb{C}^{2m}, \quad  \Gamma_0^{(j)} f_{e_j}=\begin{pmatrix}f_{e_j}(0)\\f_{e_j}(|e_j|)\end{pmatrix}, \quad
\Gamma_1^{(j)} f_{e_j}=\begin{pmatrix}f'_{e_j}(0)\\f'_{e_j}(|e_j|)\end{pmatrix}, \quad j>p_1, \\
\mathcal{H}:=\bigoplus_{j=1}^{p}\mathcal{H}_j:=\mathcal{H}_{\infty}\bigoplus\mathcal{H}_{\text{fin}},
\qquad \Gamma_a:=\bigoplus_{j=1}^{p}\Gamma_a^{(j)}, \quad a\in \{0,1\}.
\end{gathered}
\end{equation*}
We introduce the corresponding Weyl function in the same way as in Subsection \ref{abssc},
\begin{equation}\label{W.f.}
M(z)=\left(\bigoplus\limits_{j=1}^{p_1} M_{j}(z)\right)\bigoplus\left(\bigoplus\limits_{j=p_1+1}^{p} M_{j}(z)\right)
:=M_{\infty}(z)\oplus M_{\text{fin}}(z),
\end{equation}
where for each lead $e_j\in \mathcal{E}_{\infty},$ i.e. for $j\in\{1,...,p_1\}$
\begin{equation}\label{eq.5.70}
N_{1,j}(z)M_{j}(z)=N_{2,j}(z),
\end{equation}
see \eqref{eq:Weyl_fNew}-\eqref{eq:M1_M2New},
and for each finite edge $e_j\in \mathcal{E}_{\text{fin}},$ i.e. for $j\in\{p_1+1,...,p\}$
\begin{equation}\label{eq.5.80}
\begin{gathered}
M_{j}(z):=\begin{pmatrix}
-S_j(|e_j|,z)^{-1}C_j(|e_j|,z)  & S_j(|e_j|,z)^{-1}\\
(S_j(|e_j|,\overline{z})^*)^{-1}   &   -S'_j(|e_j|,z)S_j(|e_j|,z)^{-1}
\end{pmatrix} \\
:=\begin{pmatrix}
M_{e_j}^{11}(z) & M_{e_j}^{12}(z) \\
M_{e_j}^{21}(z) & M_{e_j}^{22}(z)
\end{pmatrix} =
\begin{pmatrix}
M_{e_j}^{11}(z) & M_{e_j}^{12}(z) \\
M_{e_j}^{12}(\overline{z})^* & M_{e_j}^{22}(z)
\end{pmatrix},
\end{gathered}
\end{equation}
see \eqref{eq: 3.24A}.

\subsection{Negative spectrum}\label{sect.NegSpec}

Here we clarify and substantially  complete  Theorem \ref{Bargkappa} for the case
of star graph and assuming the  minimal operator $A$  to be non-negative, $A\geq 0$.
In particular, we show that Bargmann type estimate \eqref{eq:barg_estim_if_AN>0}
is not sharp.
    \begin{theorem}\label{th.kappaminus}
Let $\mathcal{G}$ be the star graph with $p_1(>0)$ leads and $p_2(\geq 0)$ finite edges.
Let also ${\bf H}_{\alpha,Q}$ be the Hamiltonian given by \eqref{delt}-\eqref{deltop}
and let the minimal operator $A$  defined by \eqref{symop} is non-negative, $A\geq 0.$
Further, let $\Pi=\{\mathcal{H},\Gamma_0, \Gamma_1\}$ be the boundary triplet for $A^*$ given by \eqref{triples},
and let $M(\cdot)$ be the corresponding Weyl function.
Then:
\begin{equation}\label{eq.5.90}
\kappa_{-}({\bf H}_{\alpha,Q})=\kappa_{-}(T),
\end{equation}
where $T\in\mathbb{M}((p_2+1)m)$ is the following matrix:
\begin{equation}\label{matr.T}
{\tiny{
T=\begin{pmatrix}
\alpha(0)-\left(\sum_{j=1}^{p_1}M_{l_j}(0)+\sum_{j=p_1+1}^{p}M_{e_j}^{11}(0)\right) & -M_{e_{p_1+1}}^{12}(0) & \dots & -M_{e_{p}}^{12}(0)\\
-M_{e_{p_1+1}}^{21}(0) & \alpha(v_1)-M_{e_{p_1+1}}^{22}(0) & \ldots & \mathbb{O}_m \\
\vdots & \ldots & \ddots & \ldots \\
-M_{e_{p}}^{21}(0) & \mathbb{O}_m & \ldots & \alpha(v_{p_2})-M_{e_{p}}^{22}(0)
\end{pmatrix}.
}}
\end{equation}
Here  all nontrivial offdiagonal entries of $T$ are located only in the first row and
column.
\end{theorem}
\begin{proof}
To apply the Weyl function technique one should  find the  parametrization  of the Hamiltonian
${\bf H}_{\alpha,Q}$ within the boundary triplet $\Pi$ in accordance with the general  formula \eqref{C-D}.
To this end we introduce  the  block matrices $C,D\in\mathbb{C}^{m_p\times m_p}$ with $m_p:=(p_1+2p_2)m$. First we put
\begin{equation}\label{matr.C}
{\tiny{ C=\begin{pmatrix}
\alpha(0) & \mathbb{O}_m & \dots & \mathbb{O}_m & \mathbb{O}_m & \mathbb{O}_m & \dots & \mathbb{O}_m & \mathbb{O}_m\\
C_m^1 & -1_{11}  & \dots & -1_{p_1-1,1} & -1_{p_1,1} & \mathbb{O}_{(p-1)\times m} & \dots & -1_{p-1,1} & \mathbb{O}_{(p-1)\times m}\\
C_m^2 & -1_{12} & \dots & -1_{p_1-1,2} & -1_{p_1,2} & \mathbb{O}_{(p-1)\times m} & \dots & -1_{p-1,2} & \mathbb{O}_{(p-1)\times m}\\
\hdotsfor{9}\\
C_m^m & -1_{1,m} & \dots & -1_{p_1-1,m} & -1_{p_1,m} & \mathbb{O}_{(p-1)\times m} & \dots & -1_{p-1,m} & \mathbb{O}_{(p-1)\times m}\\
\mathbb{O}_m & \mathbb{O}_m & \dots & \mathbb{O}_m & \mathbb{O}_m & \alpha(v_1) & \dots & \mathbb{O}_m & \mathbb{O}_m\\
\hdotsfor{9}\\
\mathbb{O}_m & \mathbb{O}_m & \dots & \mathbb{O}_m & \mathbb{O}_m & \mathbb{O}_m & \dots & \mathbb{O}_m & \alpha(v_{p_2})
\end{pmatrix}.
}}
\end{equation}
Here $C_m^r(\in\mathbb{C}^{(p-1)\times m})$ is the matrix with  unities in the $r$-th column, and zeros otherwise. Besides,
$1_{k_1,k_2} (\in \mathbb{R}^{(p-1)\times m})$  denotes the matrix with unity on $({k_1,k_2})$-place and zeros otherwise.
Clearly,  the matrices  $\{1_{k_1,k_2}\}$ form the  standard basis  in $\mathbb{R}^{(p-1)\times m}$.
Next we put:
\begin{equation}\label{matr.D}
D=\begin{pmatrix}
I_m & I_m & \dots & I_m & I_m & \mathbb{O}_m & \dots & I_m & \mathbb{O}_m\\
{\bf 0} & \hdotsfor{7} & {\bf 0}\\
\hdotsfor{9}\\
{\bf 0} & \hdotsfor{7} & {\bf 0}\\
\mathbb{O}_m & \mathbb{O}_m & \dots & \mathbb{O}_m & \mathbb{O}_m & I_m & \dots & \mathbb{O}_m & \mathbb{O}_m\\
\hdotsfor{9}\\
\mathbb{O}_m & \mathbb{O}_m & \dots & \mathbb{O}_m & \mathbb{O}_m & \mathbb{O}_m & \dots & \mathbb{O}_m & I_m
\end{pmatrix},
\end{equation}
where $I_m (\in\mathbb{C}^{m\times m})$ denotes the standard identity  $m\times m$ matrix.

Note that the $p-1$ block-rows of the matrix $D$ corresponding to the $p-1$ block-rows of the matrix $C$ including $\widetilde{C}_m^r$, are identically zero.
Also at the last $p_2$ block-rows of the matrix $C$ the matrices $\alpha(v_j), j\in\{1,\ldots, p_2\}$ correspond to the matrices $I_m$
at the last $p_2$ block-rows of the matrix $D$.

It is easily seen that in the boundary triplet $\Pi=\{\mathcal{H},\Gamma_0,\Gamma_1\}$ of the form \eqref{triples}
the Hamiltonian ${\bf H}_{\alpha,Q}$
is given by
\begin{equation}\label{StarHam}
{\bf H}_{\alpha,Q}:=A^*\upharpoonright\ker(D\Gamma_1-C\Gamma_0),
\end{equation}
i.e. the pair of matrices $\{C,D\}$ defines the boundary relation
of the Hamiltonian ${\bf H}_{\alpha,Q}$ in accordance with general  parametrization  \eqref{C-D}.

It follows from \eqref{triples} that $A_0=\bigoplus_{e\in \mathcal{E}}A_{e,D}$ is the Dirichlet realization of the expression $\mathcal{A}$, where $A_{e,D}$ is the Dirichlet realization of the expression $\mathcal{A}_e$ on each edge $e\in \mathcal{E}.$
It is well known that $A_{e,D}$ is the Friedrichs extension of the operator $A_e$, hence $A_0=\bigoplus_{e\in \mathcal{E}}A_{e,D}=\bigoplus_{e\in \mathcal{E}}\widehat{A}_{e,F}=\widehat{A}_F$
is the Friedrichs extension of the operator $A$ (see e.g. \cite[Corollary 3.10]{MalNei2012}).

Now  Proposition \ref{A_0} applies and with account of \eqref{StarHam}  yields
  \begin{equation}\label{stark}
\kappa_-({\bf H}_{\alpha,Q})=\kappa_-(CD^*-DM(0)D^*),
  \end{equation}
with $C$ and $D$ given by \eqref{matr.C} and \eqref{matr.D}, respectively.
Further, one easily checks that
\begin{equation}\label{CD*}
CD^*=DC^*=\text{diag}\{\alpha(0)\quad \mathbb{O}_{m(p-1)}\quad \alpha(v_1)\quad \ldots\quad \alpha(v_{p_2})\}.
\end{equation}

Furthermore,
\begin{align}
& DM(0)D^*= \nonumber \\
& {\tiny{
\begin{pmatrix}
\left(\sum_{j=1}^{p_1}M_{l_j}(0)+\sum_{j=p_1+1}^{p}M_{e_j}^{11}(0)\right) & \mathbb{O}_m & \dots & \mathbb{O}_m & M_{e_{p_1+1}}^{12}(0) & \dots & M_{e_{p}}^{12}(0)\\
\mathbb{O}_m & \hdotsfor{5} & \mathbb{O}_m\\
\vdots & & & & & & \vdots\\
\mathbb{O}_m & \hdotsfor{5} & \mathbb{O}_m\\
M_{e_{p_1+1}}^{21}(0) & & & & M_{e_{p_1+1}}^{22}(0)\\
\vdots & & & & & \ddots\\
M_{e_{p}}^{21}(0) & & & & & & M_{e_{p}}^{22}(0)
\end{pmatrix}. \label{DMD}
}}
\end{align}
Combining \eqref{CD*} with \eqref{DMD} we compute the matrix $\widehat{T}_0 =
CD^*-DM(0)D^*.$  On the other hand, it is easily seen that
$\widehat{T}_0$ is unitarily equivalent to the matrix
  \begin{equation}\label{hatT}
\widehat{T}=
\begin{pmatrix}
T & \mathbb{O}_m & \ldots & \mathbb{O}_m \\
\mathbb{O}_m & \mathbb{O}_m & \ldots & \mathbb{O}_m\\
\vdots & \vdots & \ddots & \vdots\\
\mathbb{O}_m & \mathbb{O}_m & \ldots & \mathbb{O}_m
\end{pmatrix} = T\oplus\mathbb{O}_{m(p-1)}
\in \mathbb{C}^{m_p\times m_p},
\end{equation}
where $T$ is given by \eqref{matr.T}. In turn, combining this fact with \eqref{stark} implies
\begin{equation}
\kappa_{-}({\bf H}_{\alpha,Q})=\kappa_-(CD^*-DM(0)D^*)=\kappa_{-}(\widehat{T})=\kappa_-(T).
\end{equation}
This completes the proof.
\end{proof}
\begin{corollary}\label{wondres}
Assume the conditions of Theorem \ref{th.kappaminus}.
Then:
\begin{equation}
\kappa_-({\bf H}_{\alpha,Q})\leq (p_2+1)m.
\end{equation}
\end{corollary}
\begin{proof}
The required inequality immediately follows from \eqref{eq.5.90} with account of the size of the matrix $T$.
\end{proof}

The following statement  clarifies Theorem \ref{th.kappaminus} for the case of star graphs without
finite edges, i.e. with $p_2=0$.
   \begin{corollary}\label{cormised}
Assume the conditions of Theorem \ref{th.kappaminus}. Assume in addition that $p_2=0.$
Then:
   \begin{equation}\label{ksglo}
\kappa_-({\bf H}_{\alpha,Q})=\kappa_-\left(\alpha(0)-\sum_{e\in \mathcal{E}_{\infty}}M_e(0)\right)\leq m.
  \end{equation}
In particular, ${\bf H}_{\alpha,Q}\geq 0$ if and only if $\alpha(0)\geq\sum_{e\in \mathcal{E}_{\infty}}M_e(0).$
\end{corollary}
\begin{proof}
Since the finite edges are missing, the matrix $\widehat{T}$ of the form \eqref{hatT} is reduced to the one block-entry:
$T=\alpha(0)-\sum_{e\in \mathcal{E}_{\infty}}M_e(0)\in\mathbb{C}^{m\times m}.$ Theorem \ref{th.kappaminus} completes the proof.
\end{proof}
\begin{corollary}
Assume the conditions of Corollary \ref{cormised}.
If $\kappa_-({\bf H}_{\text{\emph{kir}}})=m$,  then $M(0)=\sum_{e\in \mathcal{E}_{\infty}}M_e(0)\geq 0.$
\end{corollary}
\begin{proof}
Since ${\bf H}_{\text{kir}}={\bf H}_{0,Q}$, it follows from \eqref{ksglo} that
\begin{equation}\label{kapkir}
\kappa_-({\bf H}_{\text{kir}})=\kappa_-\left(-\sum_{e\in \mathcal{E}_{\infty}}M_e(0)\right)=\kappa_+\left(\sum_{e\in \mathcal{E}_{\infty}}M_e(0)\right)=m.
\end{equation}
Noting  that $M_e(0)\in\mathbb{C}^{m\times m}$ for each $e$,  we get the required.
\end{proof}
\begin{remark}\label{rem-10}
There exist matrix functions  $\alpha: \mathcal{V}\rightarrow
\mathbb{C}^{m\times m}$  such that $\alpha(v)>0$ for each $v\in \mathcal{V}$
although $\kappa_-({\bf H}_{\alpha,Q})>0$.
Indeed,  it suffices to choose  $\alpha$ satisfying  $0<\alpha(v)<\sum_{e\in \mathcal{E}}M_e(0)$
for every $v\in \mathcal{V}.$
Then $\kappa_-({\bf H}_{\alpha,Q})=\kappa_-(\widehat{T}_0)>0$.
Therefore the estimate $\kappa_-({\bf H}_{\alpha,Q})\leq\sum_{v\in \mathcal{V}}\kappa_-(\alpha(v))$ is violated.
\end{remark}

In the following corollary we consider a special  case of the star graph of general form.
\begin{corollary}
Assume the conditions of Theorem \ref{th.kappaminus}. Assume also that $p_2=1$, $\det T_{11}\neq 0$, where
\begin{equation}\label{Tkj}
\begin{gathered}
T_{11}=\alpha(v_0)-\sum_{e\in \mathcal{E}_{\infty}}M_e(0)-M_e^{11}(0), \\
T_{12}=-M_e^{12}(0), \quad T_{22}=\alpha(v_1)-M_e^{22}(0).
\end{gathered}
\end{equation}
Then:
\begin{equation}\label{kap1plkap2}
\kappa_-({\bf H}_{\alpha,Q})=\kappa_-(T_{11})+\kappa_-\left(T_{22}-T_{12}^*T_{11}^{-1}T_{12}\right),
\end{equation}
\end{corollary}
\begin{proof}
In this case the matrix $T$ takes the form:
\begin{equation}
\begin{pmatrix}
\alpha(0)-\sum_{e\in \mathcal{E}_{\infty}}M_e(0)-M_e^{11}(0) & -M_e^{12}(0) \\
-M_e^{12}(0)^* & \alpha(v_1)-M_e^{22}(0)
\end{pmatrix}.
\end{equation}
Applying the Silvestre criterion (see \cite[Ch. 4]{DerMal17})
with account \eqref{Tkj} we arrive at \eqref{kap1plkap2}.
\end{proof}
\begin{remark}

\item[\;\;\rm (i)] Assume that $\det T_{12}\neq 0$ and $\alpha(v_1)=M_e^{22}(0)$, i.e. $T_{22}=\mathbb{O}_m.$ Then $\kappa_-({\bf H}_{\alpha,Q})=\text{rank}(T_{11})$;

\item[\;\;\rm (ii)] Assume that $\det T_{12}\neq 0$ and $\alpha(0)=\sum_{e\in \mathcal{E}_{\infty}}M_e(0)+M_e^{11}(0)$, i.e. \\ $T_{11}=\mathbb{O}_m.$
Then $\kappa_-({\bf H}_{\alpha,Q})=\text{rank}(T_{22})$.
\end{remark}
\begin{corollary}\label{cor.5.20}
Assume the conditions of Theorem \ref{th.kappaminus}.
Let also $Q=\mathbb{O}_m.$ Then:

\item[\;\;\rm (i)] $\kappa_-({\bf H}_{\alpha,0})=\kappa_-(T_1)$, where $T_1$ is given by:
\begin{equation}\label{matrt-1}
T_1=\begin{pmatrix}
\alpha(0)+\sum_{k=1}^{p_2}a_k & -a_{1} & \dots & -a_{p_2}\\
-a_{1} & \alpha(v_1)+a_{1} & \ldots & \mathbb{O}_m\\
\vdots & \ldots & \ddots & \ldots\\
-a_{p_2} & \mathbb{O}_m & \ldots & \alpha(v_{p_2})+a_{p_2}
\end{pmatrix},
\end{equation}
and $a_k:=\frac{1}{|e_{p_1+k}|}\cdot I_m$ for $k\in\{1,\ldots,p_2\}$;

\item[\;\;\rm (ii)] The following inequality holds
\begin{equation}\label{kapQz}
\kappa_-({\bf H}_{\alpha,0}) = \kappa_-(T_1) \leq \sum_{k=0}^{p_2}\kappa_-\left(\alpha(v_k)\right).
\end{equation}
\end{corollary}
\begin{proof}
(i). In the case of $Q=\mathbb{O}_m$ one  has:
\begin{equation}\label{eq.5.130}
M_{l_j}(0)=\mathbb{O}_m, \quad M_{e_j}(0)=-\frac{1}{|e_j|}\begin{pmatrix} I_m & -I_m\\ -I_m & I_m \end{pmatrix}, \quad j\in\{1,...,p\}.
\end{equation}

Using \eqref{hatT} we get:
\begin{equation}\label{DMD-0}
\widehat{T}_0
=\begin{pmatrix}
\alpha(0)+\sum_{k=1}^{p_2}a_k & \mathbb{O}_m & \dots & \mathbb{O}_m & -a_{1} & \dots & -a_{p_2}\\
\mathbb{O}_m & \hdotsfor{5} & \mathbb{O}_m\\
\vdots & & & & & & \vdots\\
\mathbb{O}_m & \hdotsfor{5} & \mathbb{O}_m\\
-a_{1} & & & & \alpha(v_1)+a_{1}\\
\vdots & & & & & \ddots\\
-a_{p_2} & & & & & & \alpha(v_{p_2})+a_{p_2}
\end{pmatrix}.
\end{equation}
On the other hand,  the matrix $\widehat{T}_0=CD^*-DM(0)D^*$ is unitarily equivalent to the matrix $T_1\oplus\mathbb{O}_{m(p-1)}$,
where $T_1$ is given by \eqref{matrt-1}. Therefore $\kappa_-({\bf H}_{\alpha,0})=\kappa_-(CD^*-DM(0)D^*)=\kappa_-(T_1).$

(ii). Note that  \eqref{kapQz} is a special case of  inequality \eqref{eq:barg_estim_if_AN>0}. However,
we provide here its simple algebraic proof.
Clearly,  $T_1=T_2 + T_3$, where
\begin{equation}\label{matrt2t3}
\begin{gathered}
T_2=\text{diag}\{\alpha(0),\alpha(v_1),\ldots,\alpha(v_{p_2})\}, \\
T_3=\begin{pmatrix}
\sum_{k=1}^{p_2}a_k & -a_{1} & \dots & -a_{p_2}\\
-a_{1} & a_{1}\\
\vdots & & \ddots\\
-a_{p_2} & & & a_{p_2}
\end{pmatrix}.
\end{gathered}
\end{equation}
One easily finds a non-singular $G\in  \mathbb{C}^{m\times m}$ such that
$G^*T_3G = \text{diag}\{\mathbb{O}_m, a_1, \ldots, a_{p_2}\} \ge 0$.
Therefore  due to the mini-max principle, $ \kappa_-(T_1) \le  \kappa_-(T_2) =
\sum_{k=0}^{p_2}\kappa_-\left(\alpha(v_k)\right)$.
Now  \eqref{kapQz} is immediate from (i).
  \end{proof}
\begin{remark}\label{remfr5}

\item[\;\;\rm (i)] Here we present simple  examples showing  that in contrast to the   equality
$\kappa_-({\bf H}_{\alpha,0})=\kappa_-(T_1)$  estimates  \eqref{kapQz} and \eqref{eq:barg_estim_if_AN>0}
are  not sharp.
To this end  consider a special  case of the star graph with one finite edge, i.e. $p_2=1$.
In this case matrix $T_1$ from \eqref{matrt-1}
turns into
  \begin{equation}\label{unbT-1}
T_1=\begin{pmatrix}
\alpha(0)+a_1 & -a_1 \\
-a_1 & \alpha(v_1)+a_1
\end{pmatrix}.
\end{equation}
Assume that $\kappa_-(\alpha(0))>0$ and $\kappa_-(\alpha(v_1))>0.$ However  the matrix $T_1$
is positively definite provided that
\begin{equation}\label{conda1}
\alpha(0)+a_1>0, \quad \alpha(v_1)+a_1 - a_1(\alpha(0)+a_1)^{-1}a_1 >0.
\end{equation}
For example, letting  $m=2$  and
  \begin{equation}
\alpha(0):=\begin{pmatrix} -1 & 0 \\ 0 & 5 \end{pmatrix}, \quad
\alpha(v_1):=\begin{pmatrix} 6 & 0 \\ 0 & -1 \end{pmatrix}, \quad
a_1:=\begin{pmatrix} 2 & 0 \\ 0 & 2 \end{pmatrix},
\end{equation}
we get
\begin{equation}
T_1=\begin{pmatrix} 1 & 0 & -2 & 0 \\
0 & 7 & 0 & -2 \\
-2 & 0 & 8 & 0 \\
0 & -2 & 0 & 1 \end{pmatrix}>0, \quad \text{i.e.} \quad \kappa_-(T_1)=0.
\end{equation}
Summing up we get $\kappa_-({\bf H}_{\alpha,0}) = \kappa_-(T_1) =0$, while
$\kappa_-(\alpha(0)) + \kappa_-(\alpha(v_1)) =2$.

\item[\;\;\rm (ii)]
Inequality \eqref{kapQz} turns into the equality when $\alpha(0)$ and $\alpha(v_1)$ are negatively definite, i.e. $\alpha(0)<0$ and $\alpha(v_1)<0$,
and the matrices $a_k$ are sufficiently small, i.e. the following conditions hold:
\begin{equation}
\alpha(0)+\sum_{k=1}^{p_2}a_k<0, \quad \alpha(v_k)+a_k<0, \quad k\in\{1,\ldots, p_2\}.
\end{equation}
This means that the finite edges $e_{p_1+1},\ldots, e_p$ have sufficiently large lengths.
In this case both inequalities \eqref{kapQz} and \eqref{eq:barg_estim_if_AN>0} turn into the equalities.
\end{remark}

The following simple statement  demonstrates that in the scalar case the matrix $T_1$ given by   \eqref{unbT-1}
cannot be positive.
\begin{corollary}
Let  $m=1$ and let  $\alpha(0)<0$, and  $\alpha(v_1)<0$. Then the matrix $T_1$ (see \eqref{unbT-1}) is not
positively definite.
  \end{corollary}
\begin{proof}
Assuming the contrary, we note that in accordance with the Silvestre criterion,
 $T_1$ is positively definite if and only if
\begin{equation}\label{unbineq}
\begin{cases}
  \alpha(0)+a_1>0, \\
\left(\alpha(0) + a_1\right)\left(\alpha(v_1) + a_1\right)-a_1^2>0.
\end{cases}
\end{equation}
Combining both  inequalities in  \eqref{unbineq} yields  $\alpha(v_1)+a_1>0$. In turn, combining this inequality
with the first one in  \eqref{unbineq} and noting that $\alpha(0)<0$ and  $\alpha(v_1)<0$, implies
$0 < \left(\alpha(0)+a_1\right)\left(\alpha(v_1)+a_1\right)<a_1^2$. This contradicts
 the second inequality in \eqref{unbineq} and completes the proof.
\end{proof}
  \begin{remark}\label{rem-l}
However,  in the scalar case there are simple examples of matrices $T_1$ showing that both
inequalities \eqref{kapQz} and \eqref{eq:barg_estim_if_AN>0} are not sharp. Namely, one can define
$T_1$ such   that $\kappa_-(T_1)=1$ while  $\kappa_-(\alpha(0))=\kappa_-(\alpha(v_1))=1$, i.e.
$\kappa_-(\alpha(0))+\kappa_-(\alpha(v_1))=2$.
It  suffices  to put $\alpha(0)=\alpha(v_1)=-1$, and $a_1=2$.
   \end{remark}
\begin{corollary}
Assume the conditions of Corollary \ref{cor.5.20}.
Let also $\alpha=\mathbb{O}_m.$ Then $\kappa_-(T_1)=0.$
\end{corollary}
\begin{remark}
The Hamiltonian ${\bf H}_{0,0}$ was treated  in \cite[Theorem 1]{BerLug10}.
\end{remark}

\subsection{Scattering matrix for the star graph}\label{sect.ScatMatr}

Here we find the scattering matrix $\{S\left({\bf H}_{\alpha}, {\bf H}_D; \lambda\right)\}_{\lambda\in\mathbb{R}_+}$
for the pair $\{{\bf H}_{\alpha}, {\bf H}_D\}$
where ${\bf H}_D:={\bf H}_{D,Q}$ is the Dirichlet operator on $\mathcal{G},$ and ${\bf H}_{\alpha}:={\bf H}_{\alpha, Q}$
is the Hamiltonian of the form \eqref{deltop} with the potential matrix $Q$ of the form \eqref{aux-3.1.1}.

As earlier for each lead $l_j\in \mathcal{E}_{\infty}$, $j \in\{1,\ldots,p_1\}$, we consider the minimal operator $A_j:=A_{l_j}$ associated with \eqref{LQ}
in $L^2(l_j;\mathbb{C}^m)$ on the domain \eqref{DOMa} with $v=0.$ Herewith $\mathrm{n}_{\pm}(A_j)=m,$ and
\begin{equation}
\mathcal{H}_j:=\mathbb{C}^m, \quad \Gamma_0^{(j)}f=f(0), \quad \Gamma_1^{(j)}f=f'(0), \quad f\in\dom(A_j^*),
%\quad j\in\{1,\ldots,p_1\}
\end{equation}
is a boundary triplet $\Pi_j:=\{\mathcal{H}_j,\Gamma_0^{(j)},\Gamma_1^{(j)}\}$ for $A_j^*.$
The corresponding Weyl function $M_j:=M_{l_j}$ is given by \eqref{eq.5.70} and takes the values in $\mathcal{B}(\mathcal{H}_j).$

For each edge $e_k\in \mathcal{E}_{\text{fin}}$, $k \in\{1,\dots,p_2\}$, we associate the symmetric operator $T_k:=T_{e_k}(\supset A_{e_k, \text{min}})$ on the domain
\begin{equation}\label{b.c.NEW}
\dom(T_k)=\{f\in\dom(A_{e_k, \text{max}}): f(0)=f'(0)=f'(v_k)=0\},
\end{equation}
where $A_{e_k,\text{max}}$ is given by \eqref{dom_H_max}.
In this case $v_o:=0,$ and $v_{in}:=v_k,$ $k\in\{1,\ldots,p_2\}.$
Notice that
\begin{equation}
\dom(T_k^*)=\{f\in\dom(A_{e_k, \text{max}}): f'(v_k)=0\}.
\end{equation}
It is easily seen that the operator $T_k$ is an extension of the minimal operator $A_{e_k}$ given by
\eqref{DOMa-1}, and $\mathrm{n}_{\pm}(T_k)=m.$
We take the boundary triplet $\widehat{\Pi}_k:=\{\widehat{\mathcal{H}}_k,\widehat{\Gamma}_0^{(k)},\widehat{\Gamma}_1^{(k)}\}$ for $T_k^*$ in the following form:
\begin{equation}
\widehat{\mathcal{H}}_k:=\mathbb{C}^m, \quad \widehat{\gG}^{(k)}_0 f = f(0), \quad \widehat{\gG}^{(k)}_1 f = f'(0), \quad f \in \dom(T^*_{k}).
\end{equation}
The corresponding Weyl function is of the form
\begin{equation}\label{Wfedg}
M_{k}(z):=M_{e_k}(z)=-C'_k(v_k,z)C_k(v_k,z)^{-1}
\end{equation}
and takes the values in $\kB(\mathcal{H}_k)$.
Here $C_k(x,z)$, $k\in\{1,\ldots,p_2\}$ is the solution of the following problem:
\begin{equation}
\begin{gathered}
\mathcal{A}_kC_k(x,z)=zC_k(x,z), \quad x\in e_k, \quad z\in\mathbb{C}, \\
C(0,z)=I_m, \quad C'(0,z)=\mathbb{O}_m.
\end{gathered}
\end{equation}

Further, $\widehat{A}:=(\bigoplus_{j=1}^{p_1} A_j)\bigoplus(\bigoplus_{k=1}^{p_2} T_k)$ is the symmetric operator on the graph $\mathcal{G}$
with the deficiency indices $\mathrm{n}_{\pm}(\widehat{A})=mp.$
It is easily seen that the triplet $\widehat{\Pi}:=\{\widehat{\mathcal{H}},\widehat{\Gamma}_0, \widehat{\Gamma}_1\},$ where
\begin{equation}\label{botrnew}
\begin{gathered}
\widehat{\mathcal{H}}:=\left(\bigoplus_{j=1}^{p_1}\mathcal{H}_j\right)\bigoplus\left(\bigoplus_{k=1}^{p_2}\widehat{\mathcal{H}}_k\right)=
\bigoplus_{j=1}^{p_1+p_2} \mathbb{C}^m=\mathbb{C}^{mp}, \\
\widehat{\Gamma}_a:=\left(\bigoplus_{j=1}^{p_1}\Gamma_a^{(j)}\right)\bigoplus\left(\bigoplus_{k=1}^{p_2}\widehat{\Gamma}_a^{(k)}\right), \quad a\in\{0,1\},
\end{gathered}
\end{equation}
is  the boundary triplet for $\widehat{A}^*.$ Herewith ${\bf H}_D:=\widehat{A}^*\upharpoonright \ker(\widehat{\Gamma}_0)$ is given by
\begin{equation}
{\bf H}_D=\left(\bigoplus_{j=1}^{p_1} {\bf H}_{j,D}\right)\bigoplus\left(\bigoplus_{k=1}^{p_2} T_{k,DN}\right),
\end{equation}
where $T_{k,DN}$ is Sturm-Liouville operator with the Dirichlet boundary condition at the vertex 0, and with the Neumann boundary condition at the loose end $v_k.$
The corresponding Weyl function $M(\cdot)$ is given by
\begin{equation}\label{Wefnew}
\begin{gathered}
M(z) = \diag\{M_1(z),\ldots, M_{p}(z)\}, \quad z \in \rho({\bf H}_D), \\
M_j(z):=M_{l_j}(z)=N_{1,j}(z)^{-1}N_{2,j}(z), \quad j\in\{1,\ldots,p_1\}, \\
M_j(z):=M_{e_{j-p_1}}(z), \quad j\in\{p_1+1,\ldots,p\}.
\end{gathered}
\end{equation}
The Hamiltonian ${\bf H}_{\alpha}$ is the extension of the operator $\widehat{A}$ and with respect to the boundary triplet \eqref{botrnew} is given by
  \begin{equation}\label{eq.H-alpha=A_theta}
{\bf H}_{\alpha}:=\widehat{A}^*\upharpoonright \ker(\widehat{D}\widehat{\Gamma}_1-\widehat{C}\widehat{\Gamma}_0), \quad \widehat{C}\widehat{D}^*=\widehat{D}\widehat{C}^*,
  \end{equation}
where $\widehat{C},\widehat{D}\in\mathbb{C}^{mp\times mp}$ are the following block matrices:
\begin{equation}\label{mC}
\widehat{C} = \kbordermatrix{
       & 1      & 2            & 3            & \dots & p_1            & p_1+1          & \dots & p\\
1      &\ga(0)     & \mathbb{O}_m & \mathbb{O}_m & \dots & \mathbb{O}_m & \mathbb{O} & \dots & \mathbb{O}_m\\
2      &I_m    & -I_m        & \mathbb{O}_m & \dots & \mathbb{O}_m & \mathbb{O} & \dots & \mathbb{O}_m\\
3      &I_m    & \mathbb{O}_m & -I_m        & \dots & \mathbb{O}_m & \mathbb{O} & \dots & \mathbb{O}_m\\
\vdots & \vdots & \vdots       & \vdots & \ddots& \vdots       & \vdots     & \ddots & \vdots\\
p_1      & I_m & \mathbb{O}_m & \mathbb{O}_m & \dots & -I_m        &\mathbb{O} & \dots & \mathbb{O}_m\\
p_1+1    &I_m & \mathbb{O}_m  & \mathbb{O}_m & \dots & \mathbb{O}   & -I_m     & \dots & \mathbb{O}_m\\
\vdots & \vdots & \vdots       & \vdots & \ddots& \vdots       & \vdots     & \ddots & \vdots\\
p &I_m & \mathbb{O}_m & \mathbb{O}_m & \dots & \mathbb{O}   & \mathbb{O}    & \dots & -I_m
},
\end{equation}
and
\begin{equation}\label{mD}
\widehat{D} = \kbordermatrix{
       & 1      & 2      & 3      & \dots  & p_1      & p_1+1    & \dots  & p\\
1      & I_m   & I_m   & I_m   & \dots  & I_m   & I_m   & \dots  & I_m\\
2      & \bO_m  & \bO_m  & \bO_m  & \dots  & \bO_m  & \bO_m  & \dots  & \bO_m\\
3      & \bO_m  & \bO_m  & \bO_m  & \dots  & \bO_m  & \bO_m  & \dots  & \bO_m\\
\vdots & \vdots & \vdots & \vdots & \ddots & \vdots & \vdots & \ddots & \vdots\\
p_1      & \bO_m  & \bO_m  & \bO_m  & \dots  & \bO_m  & \bO_m  & \dots  & \bO_m\\
p_1+1    & \bO_m  & \bO_m  & \bO_m  & \dots  & \bO_m  & \bO_m  & \dots  & \bO_m\\
\vdots & \vdots & \vdots & \vdots & \ddots & \vdots & \vdots & \ddots & \vdots\\
p    & \bO_m  & \bO_m  & \bO_m  & \dots  & \bO_m  & \bO_m  & \dots  & \bO_m
}.
\end{equation}

Further, let $S_j(x,z),$ $j\in\{1,\ldots,p_1\}$, be the matrix solution of the following problem:
   \begin{equation}
\begin{gathered}
\mathcal{A}_j S_j(x,z)=zS_j(x,z), \quad x\in l_j, \quad z\in\mathbb{C}, \\
S(0,z)=\mathbb{O}_m, \quad S'(0,z)=I_m.
\end{gathered}
\end{equation}
Furthermore, let $E_{p_1}\in\mathbb{C}^{p_1\times p_1}$ be the following matrix of the rank one:
\begin{equation}\label{EpEp-1}
E_{p_1}
:= \kbordermatrix{
 & 1 & \dots & p_1\\
1&  1 & \dots & 1\\
\vdots &\vdots &\ddots&\vdots\\
p_1 & 1 & \dots & 1
}:
\bC^{p_1} \longrightarrow \bC^{p_1}.
\end{equation}

With  these notations the main result of this subsection is given by the following theorem.
\begin{theorem}\label{th.ScS}
Let $\kG$ be the star graph consisting of $p_1(>0)$ leads and $p_2(\geq 0)$ edges, $p_1+p_2:=p$, and let
$Q:=\bigoplus_{j=1}^p Q_j\in L^1(\mathcal{G};\mathbb{C}^{m\times m}).$
Let $\widehat{\Pi}=\{\mathcal{H}, \widehat{\Gamma}_0, \widehat{\Gamma}_1\}$ be the boundary triplet of the form \eqref{botrnew},
and let $M(\cdot)$ be the corresponding Weyl function of the form \eqref{Wefnew}. Then the following holds:

\item[\;\;\rm(i)] The Hilbert space $L^2(\bR_+,d\gl;\kH_{ac})$, where
\begin{equation}
\kH_{ac} := \bigoplus^{p_1}_{j=1} \kH_j = \underbrace{\bC^m \oplus \ldots \oplus \bC^m}_{p_1} = \bC^{m\cdot p_1} \subseteq \kH
\end{equation}
performs a spectral representation of ${\bf H}_D^{ac}$;

\item[\;\;\rm(ii)]
With respect of the spectral representation $L^2(\bR_+,d\gl;\kH_{ac})$ the scattering matrix $\{S({\bf H}_\ga, {\bf H}_D;\gl)\}_{\gl\in\bR_+}$
of the scattering system $\{{\bf H}_\ga, {\bf H}_D\}$ admits the following  representation for a.e. $\lambda\in\mathbb{R}_+$:
\begin{equation}\label{scatmatr}
\begin{split}
S({\bf H}_\ga,&{\bf H}_D;\gl)\\
=&I_{\kH_{ac}}+\frac{i}{2\sqrt{\lambda}}
(N_1(\gl)^*)^{-1}\left((\ga(0)-K(\gl))^{-1}\otimes E_{p_1}\right)\cdot N_1(\gl)^{-1},
\end{split}
\end{equation}
where
\begin{equation}\label{KandN}
\begin{gathered}
K(\gl) := \sum_{j=1}^{p_1} M_j(\gl+i0): \bC^m \longrightarrow \bC^m,\quad \gl \in \bR_+,\\
N_{1,j}(\lambda):=\frac{I_m}{2i\sqrt{\lambda}}+\frac{1}{2i\sqrt{\lambda}}\int_{l_j}e^{i\sqrt{\lambda}t}Q_j(t)S_j(t,\lambda)dt, \quad \lambda\in\mathbb{R}_+, \\
N_1(\gl) :=  \bigoplus_{j=1}^{p_1} N_{1,j}(\gl) :\kH_{ac} \longrightarrow \kH_{ac},\quad \gl \in \bR_+.
\end{gathered}
\end{equation}
\end{theorem}
\begin{proof}
It follows from \eqref{themM-Prelim} that
\begin{equation}\label{themM-1}
\left(\Theta - M(z)\right)^{-1}=\left(\widehat{C}-\widehat{D}M(z)\right)^{-1}\widehat{D}.
\end{equation}

Taking into account \eqref{Wefnew} and \eqref{mC}-\eqref{mD} we get
\begin{displaymath}
\begin{split}
&\gL(z) :=\widehat{C}-\widehat{D}M(z) =\\
&\kbordermatrix{
       & 1      & 2            & \dots  & p_1       & \dots  & p\\
1      & \ga(0)-M_1(z) & -M_2(z) &  \dots  & -M_{p_1}(z)   & \dots  & -M_{p}(z)\\
2      & I_m  & -I_m  &  \dots  & \bO_m  & \dots  & \bO_m\\
\vdots & \vdots & \vdots & \ddots & \vdots & \vdots & \vdots \\
p_1      & I_m  & \bO_m  &  \dots  & -I_m  & \dots  & \bO_m\\
\vdots & \vdots & \vdots & \vdots & \vdots & \ddots & \vdots \\
p    & I_m  & \bO_m  &  \dots  & \bO_m  &  \dots  & -I_m
}.
\end{split}
\end{displaymath}
Let
\begin{equation}
\begin{gathered}
\gL_{11}(z): \kL_1 \longrightarrow \kL_1, \quad \gL_{12}(z): \kL_2 \longrightarrow \kL_1, \\
\gL_{21}: \kL_1 \longrightarrow \kL_2, \quad \gL_{22}: \kL_2 \longrightarrow \kL_2,
\end{gathered}
\end{equation}
where
\begin{displaymath}
\begin{split}
&\gL_{11}(z) := \ga(0)-M_1(z), \\
&\gL_{12}(z) :=
\kbordermatrix{
       & 2      & 3      & \dots  & p_1      & \dots  & p\\
     & -M_2(z) & -M_3(z)   & \dots  & -M_{p_1}(z)   & \dots  & -M_{p}(z)
}, \\
&\gL_{21} :=
\kbordermatrix{
       & \\
2      & I_m \\
3      & I_m\\
\vdots & \vdots\\
p_1      & I_m \\
\vdots & \vdots \\
p    & I_m
}, \quad
\gL_{22} :=
\kbordermatrix{
       & 2      & 3      & \dots  & p_1      & \dots  & p\\
2      & -I_m  & \bO_m  & \dots  & \bO_m  & \dots  & \bO_m\\
3      & \bO_m  & -I_m  & \dots  & \bO_m  & \dots  & \bO_m\\
\vdots & \vdots & \vdots & \ddots & \vdots & \vdots & \vdots\\
p_1      & \bO_m  & \bO_m  & \dots  & -I_m  & \dots  & \bO_m\\
\vdots & \vdots & \vdots & \vdots & \vdots & \ddots & \vdots \\
p    & \bO_m  & \bO_m  & \dots  & \bO_m  & \dots  & -I_m
},
\end{split}
\end{displaymath}
and
\begin{equation}
\kL_1 := \bC^m, \quad
\kL_2 := \bigoplus^{p}_{j=2} \cH_j = \underbrace{\bC^m \oplus \bC^m \oplus \ldots \oplus \bC^m}_{p-1}.
\end{equation}
Notice that $\gL_{22} = -1_{\kL_2}$.
Using the notation above we find
\begin{equation}
\widehat{C}-\widehat{D}M(z) =
\begin{bmatrix}
\gL_{11}(z) & \gL_{12}(z)\\
\gL_{21} & \gL_{22}
\end{bmatrix}:
\begin{matrix}
\kL_1 \\
\oplus\\
\kL_2
\end{matrix}
\longrightarrow
\begin{matrix}
\kL_1 \\
\oplus\\
\kL_2
\end{matrix}.
\end{equation}
The Schur complement $(\widehat{C}-\widehat{D}M(z))/\gL_{22}$ takes the form
\begin{equation}
\begin{split}
(\widehat{C}-\widehat{D}M(z))/\gL_{22} =& \gL_{11}(z)-\gL_{12}(z)\gL_{22}^{-1}\gL_{21}\\
=& \ga(0)-\sum^{p}_{j=1} M_j(z) = \ga(0) - K(z), \quad z\in\mathbb{C}_+.
\end{split}
\end{equation}
Hence the inverse matrix $(\widehat{C}-\widehat{D}M(z))^{-1}$ for  $z\in\mathbb{C}_+$ can be computed by the Frobenius formula:
\begin{equation}\label{Frobfor}
\begin{split}
(\widehat{C}&-\widehat{D}M(z))^{-1}=\\
&\begin{bmatrix}
(\ga(0) - K(z))^{-1}                     & -(\ga(0) - K(z))^{-1}\gL_{12}(z)\gL_{22}^{-1}\\
-\gL_{22}^{-1}\gL_{21}(\ga(0)-K(z))^{-1} & \gL_{22}^{-1} + \gL_{22}^{-1}\gL_{21}(\ga(0)-K(z))^{-1}\gL_{12}(z)\gL_{22}^{-1}
\end{bmatrix}.
\end{split}
\end{equation}
Since
$
-\gL_{12}(z)\gL_{22}^{-1}  = \gL_{12}(z)
$
and
$-\gL_{22}^{-1}\gL_{21} = \gL_{21},$ $z \in \bC_+$,  formula \eqref{Frobfor} is simplified to
\begin{equation}\label{LamD-0}
\begin{split}
(\widehat{C}&-\widehat{D}M(z))^{-1}= \\
&\begin{bmatrix}
(\ga(0) - K(z))^{-1}               & (\ga(0) - K(z))^{-1}\cdot \gL_{12}(z)\\
\Lambda_{21}\cdot (\ga(0) - K(z))^{-1} & -1_{\kL_2} + \Lambda_{21}\cdot (\ga(0) - K(z))^{-1}\cdot\Lambda_{12}(z)
\end{bmatrix}.
\end{split}
\end{equation}

Clearly, the matrix $\widehat{D}$ admits the representation
$
\widehat{D} =
\begin{bmatrix}
\widehat{D}_{11} & \widehat{D}_{12} \\
\widehat{D}_{21} & \widehat{D}_{22}
\end{bmatrix}
$
where
\begin{equation}\label{D-dec}
\begin{split}
\widehat{D}_{11} :=& I_m : \kL_1 \longrightarrow \kL_1, \quad
\widehat{D}_{12} :=
\kbordermatrix{
 & 2    & \dots & p\\
 & I_m & \dots & I_m
}: \kL_2 \longrightarrow \kL_1,\\
\widehat{D}_{21} :=&
\kbordermatrix{
  &\\
2 & \bO_m\\
3 & \bO_m\\
\vdots& \vdots\\
p & \bO_m
}: \kL_1 \longrightarrow \kL_2, \quad
\widehat{D}_{22} :=
\mathbb{O}_{m(p-1)}: \kL_2 \longrightarrow \kL_2.
\end{split}
\end{equation}
Combining  \eqref{LamD-0}  with  \eqref{D-dec} one gets
\begin{equation}\label{CDfactD}
\begin{split}
(\widehat{C}&-\widehat{D}M(z))^{-1}\widehat{D} \\
=&
{\tiny{
\begin{bmatrix}
(\ga(0) - K(z))^{-1}               & (\ga(0) - K(z))^{-1}\cdot \gL_{12}(z)\\
%(\ga(0) - K(z))^{-1} \times \gL_{21} & -\1_{\kL_2}  + \left((\ga(0) - K(z))^{-1} \times\gL_{21}\right)\gL_{12}(z))
\Lambda_{21}\cdot (\ga(0) - K(z))^{-1} & -1_{\kL_2} + \Lambda_{21}\cdot (\ga(0) - K(z))^{-1}\cdot\Lambda_{12}(z)
\end{bmatrix}
}}
{\tiny{
\begin{bmatrix}
\widehat{D}_{11} & \widehat{D}_{12} \\
\widehat{D}_{21} & \widehat{D}_{22}
\end{bmatrix}}}\\
=&
\begin{bmatrix}
(\ga(0) - K(z))^{-1} & (\ga(0) - K(z))^{-1} \times \widehat{D}_{12}\\
\Lambda_{21}\cdot (\ga(0) - K(z))^{-1} & \Lambda_{21}\cdot (\ga(0) - K(z))^{-1}\cdot \widehat{D}_{12}
\end{bmatrix} \\
=& (\ga(0)-K(z))^{-1} \otimes E_{p},
\end{split}
  \end{equation}
where ''$(\ga(0) - K(z))^{-1} \times \widehat{D}_{12}$'' means that each entry of $\widehat{D}_{12}$
is multiplied from the left by $(\ga(0) - K(z))^{-1}$ and
$
E_{p}:\bC^{p} \longrightarrow \bC^{p}
$
is given by \eqref{EpEp-1}.

To pass to the limit in this formula as $z\to\!\succ \gl \in \bR_+$ we first note  that in
accordance with \eqref{Wfedg} for each $j \in\{p_1+1,\ldots,p\}$ the Weyl function
$M_j(\cdot)$ is meromorphic, hence for all $\gl \in \bR_+$ except a discrete set $\Omega_j$  of positive poles,
the limit
\begin{equation}\label{Mjlime}
M_j(\gl) := \lim_{y\downarrow 0}M_j(\gl + iy) = \lim_{y\downarrow 0}M_{e_{j-p_1}}(\gl + iy), \quad j \in\{p_1+1,\ldots,p\},
\end{equation}
exists and  is  selfadjoint, i.e. $\text{Im}(M_{e_{j-p_1}}(\gl)) = 0$.

Further, since $Q_j(\cdot)\in L^1(l_j;\mathbb{C}^{m\times m})$, then due to  Theorem
\ref{corollary 141} the limit
\begin{equation}\label{Mjlim}
M_j(\gl) := \lim_{y\downarrow 0}M_j(\gl + iy), \quad j \in\{1,2,\ldots,p_1\},
\end{equation}
exists for each $\gl \in \bR_+$ and is invertible.
Therefore passing to the limit in \eqref{CDfactD} as $z\to\!\succ \gl$ and taking relations \eqref{Mjlime}, \eqref{Mjlim}
into account, we derive
\begin{equation}\label{zettolam}
\begin{split}
\left(\widehat{C}-\widehat{D}M(\lambda)\right)^{-1}\widehat{D} :=\left(\widehat{C}-\widehat{D}M(\lambda+i0)\right)^{-1}\widehat{D} \\
= \left(\alpha(0)-\sum_{j=1}^p M_j(\lambda+i0)\right)^{-1}\otimes E_p, \quad
\lambda\in\mathbb{R}_+\setminus \Omega,
\end{split}
\end{equation}
where $\Omega := \cup_{j=p_1+1}^p \Omega_j$.  Moreover, Theorem \ref{corollary 141} (see
\eqref{eq:3.51}) also ensures that $\text{Im}(M_j(\gl))$, $j \in\{1,\ldots,p_1\}$, is
invertible for each $\gl \in \bR_+$  and
  \begin{equation}\label{imGMNP}
\text{Im}(M_j(\lambda)):=\text{Im}(M_j(\lambda+i0))=\frac{1}{4\sqrt{\lambda}}(N_{1,j}\left(\lambda)^*N_{1,j}(\lambda)\right)^{-1},
\end{equation}
 where $N_{1,j}(\lambda)$ is given by \eqref{KandN}.  Hence
  \begin{equation}\label{Im12}
\sqrt{\text{Im}(M_j(\lambda))}=\frac{1}{2\sqrt[4]{\lambda}}U_j(\lambda)N_{1,j}^*(\lambda)^{-1},
 \end{equation}
where  $U_j(\cdot)$ is a measurable  family of unitary operators.

Setting $U(\lambda):=\bigoplus_{j=1}^p U_j(\lambda)$,  inserting expressions \eqref{themM-1},
\eqref{zettolam} and \eqref{Im12} into the general formula for the scattering matrix
\eqref{scatformula}, and taking into account that the scattering matrix is defined uniquely  up to
 a unitary factor we arrive at \eqref{scatmatr}.
   \end{proof}
   \begin{remark}
We emphasize   that a relatively simple formula \eqref{scatmatr} for the scattering matrix
$\{S({\bf H}_\ga, {\bf H}_D;\gl)\}_{\gl\in\bR_+}$
has been obtained due to  the treating of the Hamiltonians ${\bf H}_D$ and ${\bf H}_{\alpha}$ as extensions of
the intermediate extension $\widehat{A}(\supset A_{\text{min}})$ instead of the minimal operator  $A_{\text{min}}$.
    \end{remark}

The stationary formulation of the scattering problem on a quantum graph was first considered by Gerasimenko and Pavlov \cite{GerPav88}
(see also \cite{Gerasim1988} and \cite{BerKuch13}).

\begin{corollary}\label{cor:5.9}
Assume the conditions of  Theorem \ref{th.ScS}, and let $\mathcal G$ be  the star graph with one lead, i.e. $p_1=1$, $p_2=0$.
Then $L^2(\bR_+,d\gl;\bC^m)$ is a spectral representation of ${\bf H}_D^{ac}$ such that
   \begin{equation}\label{eq.6.194}
S({\bf H}_\ga, {\bf H}_D;\gl)=
I_m + \frac{i}{2\sqrt{\lambda}}\left(N_1(\lambda)\left(\alpha(0)-N_1(\lambda)^{-1}N_2(\lambda)\right)N_1(\lambda)^*\right)^{-1}.
  \end{equation}
    \end{corollary}
\begin{proof}
If $p_1=1$, then $E_{p_1} = 1$, and $\alpha(0)-K(\gl) = \alpha(0)-M(\gl)$, $\gl \in \bR_+$. In accordance with  \eqref{scatmatr}
\begin{equation}\label{eq:5.56}
S({\bf H}_\ga, {\bf H}_D;\gl)
=I_m+\frac{i}{2\sqrt{\lambda}}
(N_1(\gl)^*)^{-1} \left(\alpha(0)-M(\gl)\right)^{-1} N_1(\gl)^{-1}.
\end{equation}
Using \eqref{eq: 141} we find  $M(\gl) := M(\gl + i0)  = N_1(\gl)^{-1}N_2(\gl),$  \ $\gl \in \bR_+$.
Inserting  this expression  into \eqref{eq:5.56} we arrive at \eqref{eq.6.194}.
\end{proof}
Let $m = 1,$ $p_1>0$, $p_2=0$, and $\alpha=\mathbb{O}_m$. In order to distinguish matrix-valued quantities from scalar ones we use instead capital letters lower case characters. For instance the potential $Q_j(\cdot)$ is denoted by $q_j(\cdot)$, the Weyl function $M_j(\cdot)$
by $m_j(\cdot)$, $j \in\{1,2\ldots,p_1\}$, etc.
We set
\begin{equation}
m(z) = \bigoplus^{p_1}_{j=1} m_j(z), \quad z \in \bC_\pm,
\quad \mbox{and} \quad
m(\gl) = \bigoplus^{p_1}_{j=1} m_j(\gl), \quad \gl \in \bR_+.
\end{equation}
Notice that now
$\alpha-K(\gl) = -\sum^{p_1}_{j=1}m_j(\gl), \quad \gl \in \bR_+.$

In the scalar case we denote the quantities $N_{1,j}(\cdot)$ and $N_{2,j}(\cdot)$ by $n_{1,j}(\cdot)$ and $n_{2,j}(\cdot)$,
$ j \in\{1,2,\ldots,p_1\}$, respectively. However, we set
\begin{equation}
\begin{split}
N_1(\cdot) :=& \diag\{n_{1,1}(\cdot),n_{1,2}(\cdot),\ldots,n_{1,p_1}(\cdot)\}\\
N_2(\cdot) :=& \diag\{n_{2,1}(\cdot),n_{2,2}(\cdot),\ldots,n_{2,p_1}(\cdot)\}.
\end{split}
\end{equation}

\begin{corollary}\label{sc-pot-q}
Let $m =1$ and $\alpha=\mathbb{O}_m$. Then the following statements hold:

\item[\;\;\rm(i)] $L^2(\bR_+,d\gl;\bC^{p_1})$ is a spectral representation of ${\bf H}_D^{ac}$
and for $\gl \in \bR_+$
\begin{equation}
S({\bf H}_\ga, {\bf H}_D;\gl) = I_{p_1}-\frac{i}{2\sqrt{\lambda}\sum^{p_1}_{j=1}m_j(\gl)}
(N_1(\gl)^*)^{-1}\; E_{p_1} \;N_1(\gl)^{-1}.
\end{equation}

\item[\;\;\rm(ii)] If $q_j(\cdot) = q(\cdot)$, $j \in\{1,2,\ldots,p_1\}$, then
$L^2(\bR_+,d\gl;\bC^{p_1})$ is a spectral representation of ${\bf H}_D^{ac}$ such that
\begin{equation}\label{eq:5.64}
S({\bf H}_\ga, {\bf H}_D;\gl) = I_{p_1}-2i\frac{\im m(\gl)}{p_1\; m(\gl)}\;E_{p_1},\quad \gl \in \bR_+,
\end{equation}
 where $m(\cdot) := m_j(\cdot)$, $j \in\{1,2,\ldots,p_1\}$.
\end{corollary}
\begin{proof}
Clearly, one gets  $\alpha-K(\gl) = -p_1m(\gl)$, $\gl \in \bR_+$. Moreover, we have
$N_1(\gl) = n_1(\gl)\cdot I_{p_1}$, $\gl \in \bR_+$. Hence we obtain
\begin{equation}\label{eq:5.65}
S({\bf H}_\ga, {\bf H}_D;\gl) = I_{p_1}-\frac{i}{2\sqrt{\lambda}\;p_1\; m(\gl)}
\frac{1}{\overline{n_1(\gl)}n_1(\gl)}E_{p_1}, \quad \gl \in \bR_+.
\end{equation}
It follows from \eqref{eq:3.51} that
$\frac{1}{\overline{n_1(\gl)}n_1(\gl)} = 4\;\sqrt{\gl}\;\im m(\gl), \quad \gl \in \bR_+.$
Inserting this expression  into \eqref{eq:5.65} we arrive at  \eqref{eq:5.64}.
\end{proof}
\begin{remark}
At the first glance it is  not obvious that the scattering matrix \eqref{eq:5.64} is
unitary. However, using $E^2_{p_1} = p_1E_{p_1}$ we get
\begin{equation}
\begin{split}
S({\bf H}_\ga,& {\bf H}_D;\gl)^*S({\bf H}_\ga, {\bf H}_D;\gl)\\
=&
\left(I_{p_1}+2i\frac{\im m(\gl)}{p_1\; \overline{m(\gl)}}E_{p_1}\right)
\left(I_{p_1}-2i\frac{\im m(\gl)}{p_1 \;m(\gl)}E_{p_1}\right)\\
=& I_{p_1} +2i\frac{\im m(\gl)}{p_1}\left(\frac{1}{\overline{m(\gl)}}-\frac{1}{m(\gl)}\right)E_{p_1} +
4\frac{\left(\im m(\gl)\right)^2}{p_1^2|m(\gl)|^2}E^2_{p_1}\\
=& I_{p_1} - 4\frac{\left(\im m(\gl)\right)^2}{p_1|m(\gl)|^2}E_{p_1} + 4\frac{\left(\im
m(\gl)\right)^2}{p_1|m(\gl)|^2}E_{p_1} = I_{p_1}, \quad \lambda \in \mathbb{R}_+.
\end{split}
\end{equation}
\end{remark}
\begin{remark}
In the case of $\alpha(0)=\mathbb{O}_m$ formula \eqref{eq.6.194} is simplified to
\begin{equation}
S({\bf H}_\ga, {\bf H}_D;\gl)=
I_{m} - \frac{i}{2\sqrt{\lambda}} \left(N_2(\gl)N_1(\gl)^*\right)^{-1}, \qquad  \gl \in \bR_+.
\end{equation}
\end{remark}

  \subsection{Perturbation determinants}
Here using the results from \cite{MalNei2014} we compute the perturbation determinant
  $\widetilde{\Delta}_{{\bf H}_{\alpha}/{\bf H}_D}^{\widehat{\Pi}}(\zeta,z)$
of the pair $\{{\bf H}_{\alpha},{\bf H}_D\}$  with respect to  the boundary triplet
$\widehat{\Pi}$ given by \eqref{botrnew}. For the definition of perturbation determinant
$\widetilde{\Delta}_{{\bf H}_{\alpha}/{\bf H}_D}^{\widehat{\Pi}}(\zeta,z)$ see \cite{Yaf92}.
   \begin{proposition}
Let ${\bf H}_{\alpha}(\in \Ext_{\widehat{A}})$ be the Hamiltonian
given by   \eqref{deltop},
and let ${\bf H}_D$ be the Dirichlet operator on the star graph $\mathcal{G}$.
Let also $\widehat{\Pi}$ be the boudary triplet for $\widehat{A}^*$ defined by \eqref{botrnew}.
Then for any
$\zeta\in\rho({\bf H}_{\alpha})\cap\rho({\bf H}_D)$
the perturbation determinant with respect to the triplet $\widehat{\Pi}$ admits the   representation:
   \begin{equation}\label{perturbdet}
\widetilde{\Delta}_{{\bf H}_{\alpha}/{\bf H}_D}^{\widehat{\Pi}}(\zeta,z)=
\frac{\det\left(\widehat{C}-\widehat{D}M(z)\right)}{\det\left(\widehat{C}-\widehat{D}M(\zeta)\right)}, \qquad
z\in\rho({\bf H}_D),
\end{equation}
where $\widehat{C}$ and $\widehat{D}$ are given by \eqref{mC} and \eqref{mD} respectively, and $M(\cdot)$ is the Weyl function given by \eqref{Wefnew}.
\end{proposition}
  \begin{proof}
In accordance with  Proposition 4.3 from \cite{MalNei2014} for any proper  extension $\widetilde{A}\in \Ext_{\widehat{A}}$
the perturbation determinant $\widetilde{\Delta}_{\widetilde{A}/A_0}(\zeta,z)$ is given by
   \begin{equation}\label{perturbdet-old}
\widetilde{\Delta}_{\widetilde{A}/A_0}(\zeta,z)=\det\left(I_{\mathcal{H}}+(\Theta-M(\zeta))^{-1}(M(\zeta)-M(z))\right)
   \end{equation}
for $\zeta\in\rho(\widetilde{A})\cap\rho(A_0)$ and $z\in\rho(A_0)$.
Note that  in the boundary triplet $\widehat{\Pi}$ for   $\widehat{A}^*$ the Hamiltonian ${\bf H}_{\alpha}$
admits a representation  \eqref{eq.H-alpha=A_theta}, i.e. ${\bf H}_{\alpha} = {\widehat A}_{\Theta}$
with $\Theta = \ker \left(\widehat{C}\,\,\, \widehat{D}\right)$.
Further, recall that in accordance with  \eqref{themM-Prelim}
  \begin{equation}\label{themM}
(\Theta-M(\zeta))^{-1}=\left(\widehat{C}-\widehat{D}M(\zeta)\right)^{-1}\widehat{D}.
  \end{equation}
Noting that  in our case $\widetilde{A}={\bf H}_{\alpha}$, $A_0={\bf H}_D$, $\mathcal{H}=\mathbb{C}^{mp\times mp}$, and
inserting \eqref{themM} in \eqref{perturbdet-old}  we derive
   \begin{equation}\label{pdetn2}
\begin{split}
\widetilde{\Delta}_{{\bf H}_{\alpha}/{\bf H}_D}^{\widehat{\Pi}}(\zeta,z)=
\det\left(I_{mp}+\left(\widehat{C}-\widehat{D}M(\zeta)\right)^{-1}\widehat{D}\left(M(\zeta)-M(z)\right)\right) \\
=\det\left(\left(\widehat{C}-\widehat{D}M(\zeta)\right)^{-1}\left(\widehat{C}- \widehat{D}M(z)\right)\right).
\end{split}
  \end{equation}
Applying the chain rule for determinants
we arrive at \eqref{perturbdet}.
\end{proof}
\begin{remark}

\item[\;\;\rm (i)]  The matrix $(\widehat{C}-\widehat{D}M(\cdot))^{-1}$ is explicitly
computed  in  \eqref{LamD-0};

\item[\;\;\rm (ii)]  Inserting formula \eqref{CDfactD}  in the first equality in  \eqref{pdetn2} we arrive at another expression for
the perturbation determinant $\widetilde{\Delta}_{{\bf H}_{\alpha}/{\bf H}_D}^{\widehat{\Pi}}(\zeta,z):$
    \begin{equation}
\widetilde{\Delta}_{{\bf H}_{\alpha}/{\bf H}_D}^{\widehat{\Pi}}(\zeta,z)=
\det\left(I_{mp}+\left((\alpha(0)-K(\zeta))^{-1}\otimes E_p\right)(M(\zeta)-M(z))\right),
\end{equation}
where $K(\zeta):=\sum_{j=1}^p M_j(\zeta)$  and  $E_p\in\mathbb{C}^{p\times p}$ is the matrix
(of the rank one) consisting of $p^2$ units (see \eqref{EpEp-1}),
and $\zeta\in\mathbb{C}_+\cup\mathbb{R}_+$, $z\in\rho({\bf H}_D)$.
\end{remark}

{\bf Acknowledgments.} Research supported by
the ''RUDN University Program 5-100'' (M.M.) and by the European Research Council (ERC) under Grant No. AdG 267802 "AnaMultiScale" (H.N.).

The authors are indebted to anonymous referees for  useful remarks and comments
allowing us to improve the exposition.

{\bf Yaroslav Granovskyi}

Institute of Applied Mathematics and Mechanics, NAS of Ukraine

E-Mail: {\tt yarvodoley@mail.ru}

{\bf Mark Malamud}

Peoples Friendship University of Russia (RUDN University), Russian Federation

E-Mail: {\tt malamud3m@gmail.com}

{\bf Hagen Neidhardt}

(Deceased)

\end{document}